\documentclass{article}
\usepackage[T1]{fontenc}
\usepackage[utf8]{inputenc}
\usepackage[english]{babel}
\usepackage{graphicx}
\usepackage{amsmath}
\usepackage[all]{xy} 
\usepackage{lmodern}
\usepackage[pagebackref]{hyperref}
\usepackage{grffile}
\usepackage{adjustbox}
\usepackage{color}
\usepackage{xcolor}
\usepackage{amsthm} 
\usepackage{subcaption}
\usepackage{amssymb}
\usepackage{stmaryrd}
\usepackage{caption}
\usepackage{url}
\usepackage{faktor}
\usepackage{mathrsfs}
\usepackage{relsize,exscale}
\usepackage{multicol}
\usepackage{mathabx}
\usepackage{float}
\usepackage{amsbsy}
\usepackage{appendix}

\usepackage{xfrac}  
\usepackage{changebar}
\usepackage{vertbars}
\usepackage{pgf,tikz}
\usetikzlibrary{arrows}
\usetikzlibrary{decorations}
\usetikzlibrary{patterns}
\usepackage{comment}
\usepackage{wasysym}
\usepackage{titlesec}
\usepackage{bbold}
\usepackage{pgfplots}
\usepackage[rightcaption]{sidecap}
\usepackage[a4paper, total={6in, 9in}]{geometry}
\usepackage{thmtools, thm-restate}
\usepackage{algorithm}
\usepackage{algpseudocode}
\usepackage{ulem}
\usepackage{nicematrix}
\usepackage{pict2e}
\usepackage[nameinlink,capitalize]{cleveref}

\DeclareRobustCommand{\lowerrighttriangle}{%
	\begingroup
	\setlength{\unitlength}{1.7ex}%
	\begin{picture}(1,1)
		\polyline(1,0)(0,0)(1,1)(1,0)(.5,0)
	\end{picture}%
	\endgroup
}
\def \lt{\lowerrighttriangle}

\def \be{\begin{eqnarray*}}
	\def \ee{\end{eqnarray*}}
\def \ben{\begin{eqnarray}}
	\def \een{\end{eqnarray}}

\newtheorem{Definition}{Definition}[section]

\newtheorem{Theoreme}[Definition]{Theorem}
\newtheorem{Proposition}[Definition]{Proposition}
\newtheorem{Notation}[Definition]{Notation}
\newtheorem{Corollaire}[Definition]{Corollary}
\newtheorem{Lemma}[Definition]{Lemma}
\newtheorem{Remarque}[Definition]{Remark}
\crefname{figure}{fig.}{Fig.}
\crefname{section}{sec.}{Sec.}

\newcounter{c}
\def \bir{ \begin{itemize} \setcounter{c}{0}}
	\def \itr{\addtocounter{c}{1}\item[($\roman{c}$)]}
	\def \eir{\end{itemize}\vspace{-2em}~}

\newcounter{d}
\def \bia{\begin{itemize} \setcounter{d}{0}}
	\def \eia{\end{itemize}\vspace{-2em}~}

\newcounter{b}
\def \bi{\begin{itemize} \setcounter{b}{0}}
	
	\def \ei{\end{itemize}\vspace{-2em}~}

\newcommand{\ZZ}{\mathbb{Z}}

\newcommand{\RR}{\mathbb{R}}
\newcommand{\NNN}{\mathbb{Z}_{>0}}

\newcommand{\NN}{\mathbb{N}}
\newcommand{\entn}{\{1,\ldots,n\}}
\newcommand{\entk}{\{1,\ldots,\kappa\}}
\def \ent#1#2{\{#1,\ldots,#2\}}
\newcommand{\entty}{[0,+\infty]}
\newcommand{\zerun}{[0,1]}

\newcommand*\Bell{\ensuremath{\boldsymbol\ell}}
\newcommand*\Bcp{\ensuremath{\boldsymbol\cp}}
\newcommand*\BBb{\ensuremath{\boldsymbol\bb}}

\newcommand*\Bbeta{\ensuremath{\boldsymbol\beta}}
\newcommand*\Bdeta{\ensuremath{\boldsymbol\Delta}}
\newcommand*\Balpha{\ensuremath{\boldsymbol\alpha}}
\newcommand*\Bzeta{\ensuremath{\boldsymbol\zeta}}

\newcommand*\Bxi{\ensuremath{\boldsymbol\xi}}

\DeclareMathOperator{\NZS}{NZS}

\DeclareMathOperator{\Cste}{Cste}

\DeclareMathOperator{\AP}{AP}

\DeclareMathOperator{\Jac}{Jac}
\DeclareMathOperator{\cov}{cov}

\DeclareMathOperator{\Alg}{Alg}

\newcommand{\Leb}{\mathsf{Leb}}
\newcommand{\zn}{\mathbf{z}[n]}
\newcommand{\bz}{\mathbf{z}}
\newcommand{\zzn}{{z}[n]}
\newcommand{\Aff}{\mathsf{Aff}}
\newcommand{\wj}{{j+1}}

\newcommand{\conv}{{\sf conv}}

\newcommand{\ie}{\textit{i.e. }}
\newcommand{\cvg}{\underset{n\to\infty}{\longrightarrow}}
\newcommand{\Cst}{\mathsf{\Cste}}
\newcommand{\im}{{\sf Im}}

\def \sous#1#2{\mathrel{\mathop{\kern 0pt#1}\limits_{#2}}}

\def \sur#1#2{\mathrel{\mathop{\kern 0pt#1}\limits^{#2}}}
\def \flr#1{\left\lfloor#1\right\rfloor}
\def \norm#1{\vert\vert#1\vert\vert}
\def \abso#1{\left\vert#1\right\vert}
\def \app#1#2#3#4#5{\begin{array}{rccl} #1:&#2&\longrightarrow&#3\\ &#4&\longmapsto&#5\end{array}}
\def \eqd{\sur{=}{(d)}}

\def \proba{\xrightarrow[n]{(proba.)}}
\def \as{\xrightarrow[n]{(a.s.)}}
\def \dd{\xrightarrow[n]{(d)}}

\def \Qn#1{\mathbb{Q}^{(n)}_{#1}}
\def \tnorm#1{\tan\left(\frac{\pi}{2}#1\right)}
\def \Area#1{{\sf Area}\left(#1\right)}
\def \P{\mathbb{P}}

\def \Dom#1{{\sf Dom}(#1)}

\newcommand{\Sntk}{{\mathcal{S}}^{(n)}(s[\kappa])}
\newcommand{\Sntksj}{{\mathcal{S}}^{(n)}(\ell[\kappa],s[\kappa])}
\newcommand{\diffeo}{\chi_{s[\kappa]}}
\newcommand{\Snk}{\mathcal{S}^{(n)}(s[\kappa])}
\newcommand{\Order}{{\sf Order}_{\ell[\kappa],s[\kappa]}}
\newcommand{\Orderr}{{\sf Order}}
\newcommand{\CVnofn}{\mathcal{D}_\kappa(n)}
\newcommand{\CVn}{\mathcal{C}_\kappa(n)}
\newcommand{\CVnfull}{{\mathcal{C}_\kappa}(\Nkn)}
\newcommand{\CVnij}{{\mathcal{C}_\kappa}(s[\kappa])}

\newcommand{\nCC}{\lowerrighttriangle{\sf CC}}
\newcommand{\CC}{{\sf CC}}

\newcommand{\Zn}{\overset{\curvearrowleft}{\mathcal{Z}}_n}

\newcommand{\Niceset}{\mathsf{NiceSet}(s[\kappa])}
\newcommand{\fcs}{\mathsf{C}}
\newcommand{\fys}{\mathsf{Y}}

\newcommand{\thk}{\theta_\kappa}
\newcommand{\Ck}{\mathfrak{C}_\kappa}
\newcommand{\ECP}{\mathsf{ECP}}

\newcommand{\bb}{\mathsf{b}}
\newcommand{\corner}{\mathsf{corner}}

\newcommand{\chain}{\mathsf{Chain}}
\newcommand{\sk}{\mathbf{s}^{(n)}[\kappa]}
\newcommand{\ssk}{s[\kappa]}
\newcommand{\ssj}{\mathbf{s}_j}
\newcommand{\Nkn}{\mathbb{N}_\kappa(n)}
\newcommand{\bNkn}{\widebar{\mathbb{N}_{\kappa-1}}(n)}
\newcommand{\lk}{\Bell^{(n)}[\kappa]}
\newcommand{\ck}{\mathbf{c}^{(n)}[\kappa]}
\def \cp{{\sf cp}}
\newcommand{\Lj}{\ell[\kappa]}
\newcommand{\Cj}{c[\kappa]}

\newcommand{\sth}{\sin(\theta_\kappa)}
\newcommand{\cth}{\cos(\theta_\kappa)}
\newcommand{\cl}{\mathfrak{cl}}
\newcommand{\tth}{\tan(\theta_\kappa)}
\newcommand{\pk}{\mathbb{P}_\kappa(n)}
\newcommand{\ptk}{\widetilde{\mathbb{P}}_\kappa(n)}

\newcommand{\Dtn}{{\mathbb{D}}^{(n)}_{\kappa}}
\newcommand{\Un}{\mathbb{U}^{(n)}_{\kappa}}
\newcommand{\In}{\mathcal{J}}

\newcommand{\snj}{\mathbf{s}^{(n)}_{j}}

\algdef{SE}[DOWHILE]{Do}{doWhile}{\algorithmicdo}[1]{\algorithmicwhile\ #1}%
\begin{document}
	
	\author{Ludovic Morin}
	\date{\footnotesize Univ. Bordeaux, CNRS, Bordeaux INP, LaBRI, UMR 5800, F-33400 Talence, France}
	\title{Probability that $n$ points are in convex position in a regular $\kappa$-gon :\\ Asymptotic results.}
	\maketitle

	\subsection*{Abstract}

	Let $\pk$ be the probability that $n$ points $z_1,\ldots,z_n$ picked uniformly and independently in $\Ck$, a regular $\kappa$-gon with area $1$, are in convex position, that is, form the vertex set of a convex polygon.
	In this paper, we compute $\pk$ up to asymptotic equivalence, as $n\to+\infty$, for all $\kappa\geq 3$, which improves on a famous result of Bárány \cite{barany2}.
	The second purpose of this paper is to establish a limit theorem which describes the fluctuations around the limit shape of a $n$-tuple of points in convex position when $n\to+\infty$.
	Finally, we give an asymptotically exact algorithm for the random generation of $z_1,\ldots,z_n$, conditioned to be in convex position in $\Ck$.	
	\paragraph{Mathematics Subject Classification (2020): }Primary 52A22; 60D05
	\paragraph{Keywords: }Random convex chains, random polygon, Sylvester’s problem, stochastic geometry
	
	\section{Introduction}
	
	\begin{minipage}{0.7\textwidth}
		Let $\Ck$ be the regular $\kappa$-gon with area $1$ positioned on the $X$-axis, as represented in \Cref{fig1}, $r_\kappa=\left(4\tan\left(\frac{\pi}{\kappa}\right)/\kappa\right)^{1/2}$ be its side length, and $ \theta_\kappa=\frac{(\kappa-2)\pi}{\kappa}$ be the interior angle between two consecutive sides.
		
		%For any $n\geq1$, any $z$, notation $z[n]$ stands for the $n$-tuple $(z_1,\ldots,z_n)$.\\
		For any compact convex domain $K$ of area 1 in $\RR^2$ with non empty interior and for any $n\in\NN$, we let $\mathbb{U}_K^{(n)}$ denote the law of a $n$-tuple $\zn:=(\bz_1,\cdots,\bz_n)$, where the $\bz_i$ are independent and identically distributed (i.i.d.) and uniform in $K$. \par
		In the special case $K=\Ck$, we write for short $\mathbb{U}^{(n)}_{\kappa}:=\mathbb{U}^{(n)}_{\Ck}$.

	\end{minipage}
	\hspace{2ex} % eventuellement
	\begin{minipage}{0.3\textwidth}
		\centering
		\begin{tikzpicture}[scale=0.8]
			
			\draw[->] (-1,0) -- (2,0);
			\draw (2,0) node[right] {\tiny $x$};
			\draw[-] (0,-0.5) -- (0,0);
			\draw[->] (0,0.58) -- (0,3);
			\draw (0,3) node[above] {\tiny$y$};
			\node[inner sep=0.2pt] (M) at (0,0){};
			\node[below left] at (M) {\small $0$};

			\draw[blue,thick] (0,0)--(1.0491,0.0)node[midway,below]{ $r_\kappa$};
			\draw[blue,thick] (1.0491,0.0)--(1.7033,0.8202);
			\draw[blue,thick] (1.7033,0.8202)--(1.4698,1.8431);
			\draw[blue,thick] (1.4698,1.8431)--(0.5245,2.2983)--(-0.4206,1.8431)--(-0.6541,0.8202);
			\draw[blue,thick] (-0.6541,0.8202)--(0,0);
			
			\draw[blue] (0.1,0) arc (0:128.5:0.1) node[pos=0.01,above]{$\theta_\kappa$} ;
			
		\end{tikzpicture}
		
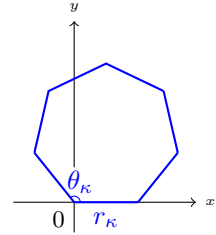
\captionof{figure}{$\mathfrak{C}_7$}\label{fig1}
	\end{minipage}\\
	A $n$-tuple of points $z[n]\in(\RR^2)^n$ is said to be in convex position if $\{z_1,\cdots,z_n\}$ is the vertex set of a convex polygon, which we will refer to as the $z[n]$-gon; the set of such $n$-tuples $\zzn$ is denoted by $\mathcal{Z}_n$.
	Hence
	\[\mathbb{P}_K(n):=\mathbb{P}\left(\zn\in \mathcal{Z}_n\right)=\mathbb{U}^{(n)}_K(\mathcal{Z}_n),\]
	is the probability that $n$ i.i.d. random points $\zn$ taken uniformly in $K$, are in convex position, and
	\[\pk:=\mathbb{P}_{\Ck}(n),\]is the corresponding probability in the regular $\kappa$-gon. \par
	The purpose of this paper is threefold. First, we give an equivalent of $\pk$ as $n\to\infty$ (see \Cref{thm-1} below), then we describe the fluctuations of a $\zzn$-gon with distribution $\mathbb{U}^{(n)}_\kappa$ conditioned to be in convex position (\Cref{thm:globfluc}), and we conclude by providing an algorithm to sample such a $n$-tuple $\zzn$ (\Cref{sec1}).
	One of the main contributions of this paper is thus the following theorem: 
	\begin{restatable}{Theoreme}{firstThm}\label{thm-1}
		Let $\kappa\geq 3$ be an integer. We have
		\[\mathbb{P}_{\kappa}(n)\underset{n\to +\infty}{\sim}C_\kappa\cdot\frac{e^{2n}}{4^n}\frac{\kappa^{3n}r_\kappa^{2n}\sth^{n}}{n^{2n+\kappa/2}},\]
		where
		\[C_\kappa = \frac{1}{\pi^{\kappa/2}\sqrt{\mathbb{d}_\kappa}}\frac{\sqrt{\kappa}^{\kappa+1}}{4^\kappa(1+\cth)^\kappa},\] and $\mathbb{d}_\kappa$ is the determinant of a deterministic matrix (see \Cref{thmonstre}), whose explicit formula is given by:
		\begin{align}\label{eq:det}
			\mathbb{d}_\kappa=\frac{\kappa}{3\cdot2^\kappa}\left(2(-1)^{\kappa-1}+(2-\sqrt{3})^{\kappa}+(2+\sqrt{3})^{\kappa}\right).
		\end{align}
	\end{restatable}
	
	Theorem \ref{thm-1} actually refines a famous result of Bárány \cite{barany2} in the case of $\kappa$-gons (note however, that  Bárány's result holds under weaker hypothesis):% , in the case of regular convex domains :
	\begin{Theoreme}\cite{barany2}\label{thm2}
		For any compact convex set $K$ of area 1 with non empty interior,
		\[\lim_{n\to+\infty} n^2\left(\mathbb{P}_{K}(n)\right)^{\frac{1}{n}}=\frac{1}{4}e^2\AP^*(K)^3,\]
		where $\AP^*(K)$ is the supremum of the affine perimeters of all convex sets $S\subset K.$ 
	\end{Theoreme}
	The definition of the affine perimeter will be recalled in \Cref{def1}; we send the interested reader to \cite{barany1} for additional details.
	In the $\kappa$-gon case, as will be shown in \Cref{lem3}, we have
	\[\AP^*(\Ck)=\kappa\left(r_\kappa^2\sth\right)^{1/3},\]
	so that one can check that in the particular $\kappa$-gon case, \Cref{thm-1} is compatible and more precise than \Cref{thm2}. 
	
	The quantity $\mathbb{P}_K(n)$ has been widely studied since the 19th century, and for a large variety of convex sets $K$, not just regular polygons. Sylvester \cite{sylvester} initiated the reflection on this matter looking at the probability that four points chosen at random in the plane were in convex position. Though Sylvester's question was ill-posed, it matured in its later formulation into the study of $\mathbb{P}_K(4)$, for any convex shape $K$ of area 1 (see Pfiefer \cite{pfiefer} for historical notes). In 1917, Blaschke \cite{blaschke1917affine} determined the convex domain $K$ that maximizes or minimizes the probability $\mathbb{P}_K(4)$ (on the set of nonflat compact convex domains of $\RR^2$) by proving that the lower bound is achieved when $K=\triangle$ is a triangle, and the upper bound when $K=\bigcirc$ is a disk, namely
	\[\frac{2}{3}=\mathbb{P}_\triangle(4)\leq\mathbb{P}_K(4)\leq \mathbb{P}_\bigcirc(4)=1-\frac{35}{12\pi^2}.\]
	In the same direction, Marckert and Rahmani \cite{marckert:hal-02913348} proved in 2021 that
	\[\frac{11}{36}=\mathbb{P}_\triangle(5)\leq\mathbb{P}_K(5)\leq \mathbb{P}_\bigcirc(5)=1-\frac{305}{48\pi^2}.\]
	This question can be generalized to different values of $n$, and other dimensions. To this day, the conjecture in dimension $2$, for all $n\geq3$:
	\[\mathbb{P}_\triangle(n)\leq\mathbb{P}_K(n)\leq \mathbb{P}_\bigcirc(n),\] remains open. Yet, the value $\mathbb{P}_\bigcirc(n)$ is computable for all $n\geq3$ since 2017 thanks to Marckert's algebraic formula \cite{marckert2017probability} in the disk case. Note also that Hilhorst, Calka and Schehr \cite{hilhorst:hal-00330444} managed in 2008 to derive an asymptotic expansion of $\log{\mathbb{P}_\bigcirc(n)}$.
	
	In the case of regular convex polygons, exact formulas are rare but Valtr \cite{Valtr1995} proved in 1995 that for $K$ a parallelogram, \[\mathbb{P}_4(n)=\mathbb{P}_\Box(n)=\frac{1}{(n!)^2}{2n-2\choose n-1}^2\underset{n\to +\infty}{\sim} \frac{1}{\pi^{2}2^5}\frac{4^{2n}e^{2n}}{n^{2n+2}}\] and in 1996 when $K$ is a triangle \cite{valtr1996probability} \[\mathbb{P}_3(n)=\mathbb{P}_\triangle(n)=\frac{2^n (3n-3)!}{(2n)!((n-1)!)^3}\underset{n\to +\infty}{\sim} \frac{\sqrt{3}}{4}\frac{1}{\pi^{3/2}3^3}\frac{3^{3n}e^{2n}}{2^nn^{2n+3/2}}.\]
	The equivalents given at the right-hand-side are of course consistent with Theorem \ref{thm-1}. Note however that our method will allow us to recover Valtr's formulas in \Cref{sec:valtr} (our approach avoids discretization arguments, but it largely relies on Valtr's ideas.)
	
	In dimension $d\geq3$,  if $\triangle^d$ and $\bigcirc^d$ denote respectively a simplex and an ellipsoïde of volume 1, the following generalization of Sylvester's question
	\[\mathbb{P}_{\triangle^d}(d+2)\leq\mathbb{P}_K(d+2)\leq \mathbb{P}_{\bigcirc^d}(d+2),\]
	for any convex domain $K\subset \RR^d$ of volume $1$, is a conjecture that remains to be proven (though the right inequality is known as a generalization of Blaschke's proof in dimension 2).
	For a comprehensive overview of these matters, we refer to Schneider \cite{schneider2017discrete}.
	\par
	
	\noindent\begin{minipage}{0.6\textwidth}	
		\vspace{0.2mm}
		\paragraph{Canonical ordering of $z[n]$-gons.}An element of $z[n]\in\mathcal{Z}_n$ (in convex position) is said to be in convex canonical order if it satisfies
		the following conditions (see \Cref{fig:CO}):\\
		$\bullet$ If $(x_i,y_i)$ is the coordinates of $z_i$ in $\RR^2$, $y_1\leq y_i$ for all $i$ (that is, $z_1$ has the smallest $y$-component), and among those having the minimal $y$ component, it has the smallest $x$ component.\\
		$\bullet$ The sequence $(\arg(z_{i+1}-z_i),1\leq i \leq n-1)$ is non-decreasing in $[0,2\pi]$.\par
		
	\end{minipage}
	\hspace{2ex} % eventuellement
	\begin{minipage}{0.4\textwidth}
			\centering
		\begin{tikzpicture}[scale=1.3]
			
			\draw[->] (-1,0) -- (2,0);
			\draw (2,0) node[right] {\small$x$};
			\draw[->] (0,-0.5) -- (0,2.5);
			\draw (0,2.5) node[above] {\small$y$};
			
			\draw[blue] (0,0)--(1.0491,0.0)node[midway,below]{};
			\draw[blue] (1.0491,0.0)--(1.7033,0.8202);
			\draw[blue] (1.7033,0.8202)--(1.4698,1.8431);
			\draw[blue] (1.4698,1.8431)--(0.5245,2.2983)--(-0.4206,1.8431)--(-0.6541,0.8202);
			\draw[blue] (-0.6541,0.8202)--(0,0);

			\node[inner sep=0.7pt,circle,draw=red,fill=red] (M) at (-0.225,1.661){};
			\node[inner sep=0.7pt,circle,draw=red,fill=red] at (0.535,2.111){};
			\node[inner sep=0.7pt,circle,draw=red,fill=red] (X) at (1.200,0.9798){};
			\node[inner sep=0.7pt,circle,draw=red,fill=red] (L) at (-0.28,1.3391){};
			\node[below,red] at (L) {\small $z_n$};
			\node[inner sep=0.7pt,circle,draw=red,fill=red] at (1.1736,1.1702){};
			\node[inner sep=0.7pt,circle,draw=red,fill=red] at (0.7977,2.03){};
			\node[inner sep=0.7pt,circle,draw=red,fill=red] (J) at (0.8402,0.5037){};
			\node[below,red] at (J) {\small $z_2$};
			\node[inner sep=0.7pt,circle,draw=red,fill=red] (I) at (0.3333,0.2553){};
			\node[below,red] at (I) {\small $z_1$};
			\node[inner sep=0.7pt,circle,draw=red,fill=red] (K) at (1.1,0.778){};
			\node[below,red] at (K) {\small $z_3$};
			\node[below left] at (0,0){\small $O$};
			
		\end{tikzpicture}
	
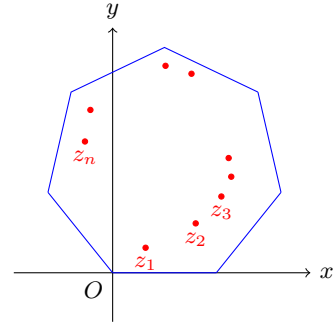
\captionof{figure}{\label{fig:CO} Some $z[n]$ in $\mathcal{C}_7(n)$}
	\end{minipage}
	
	We denote by $\Zn$ the subset of $\mathcal{Z}_n$ of $n$-tuples of points $\zzn$ in convex canonical order. The symmetric group ${\cal S}_n$ acts transitively on $\mathcal{Z}_n$ by relabelling the vertex indices;
	each orbit contains a unique element of $\Zn$.
	We put $\CVnofn=\mathcal{Z}_n\cap(\Ck)^n$ and $\CVn=\Zn\cap(\Ck)^n$.
	
	Since $\zn$ is picked according to the uniform distribution on $(\Ck)^n$, and since this measure is the Lebesgue measure $\Leb$ on this set, we have
	\[\pk=\Leb_{2n}(\CVnofn)=n!~\Leb_{2n}(\CVn).\]
	In what follows, we will abandon $\CVnofn$ and work mainly in $\CVn$ as the elements of this set are easier to parameterize.
	The argument we're going to detail for the computation of $\Leb(\CVn)$ is mainly deterministic, and we won't really be using random variables in the analysis, even though everything could be rewritten in terms of them (but the proof seems much more cumbersome in those terms).

	\begin{Notation}\label{not:Qn}
		From now on, we denote by ${\mathbb{Q}}^{(n)}_{K}$, the law of a $n$-tuple of points with distribution $\Un$, conditioned to be in  $\CVn$, and denote for short $\Qn{\kappa}:={\mathbb{Q}}^{(n)}_{\Ck}$, that is:
		\begin{align*}
			\mathrm{d}\Qn{\kappa}(z[n])=\frac{n!}{\pk}\mathbb{1}_{z[n]\in\CVn}\mathrm{d}z_1\ldots\mathrm{d}z_n.
		\end{align*}
		This formula represents the measure, but it is hardly exploitable for further computations; we will thus need an alternative geometrical understanding of $\CVn$, which was inspired by Valtr's papers.
	\end{Notation}

	\paragraph{Limit shape.} Bárány\cite{barany1} proved in 1999 that the convex hull of a $n$-tuple $\zn$ with distribution $\Qn{\kappa}$ converges in probability for the Hausdorff topology to an explicit deterministic domain $\Dom{K}$, having the important property that \[\AP^*(K)=\AP(\Dom{K}).\] 
	In the case of the $\kappa$-gon, we represent this domain $\Dom{\Ck}$ in \Cref{fig:AP2}. We will explain in \Cref{thm0} how $\Dom{\Ck}$ is determined using $\Ck$'s inner symmetries.
	
			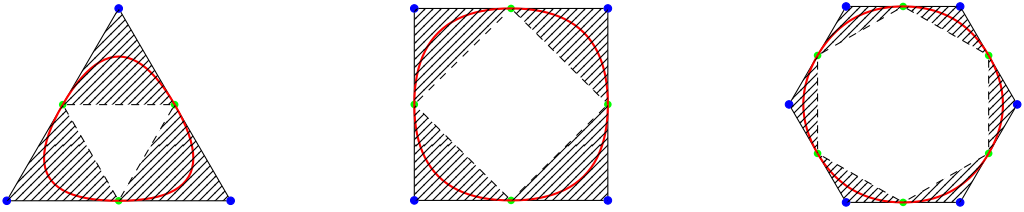
\begin{figure}[htbp]
			
			\begin{minipage}{0.3\textwidth}
				\centering
				\begin{tikzpicture}[scale=0.85]
					
					\def\R{2}
					\def\K{3}
					\def\alphaK{180/\K}

					%définition des points de l'hexagone extérieur
					\label{key}					\foreach \x in {1,2,...,\K} {\node[circle, draw=blue, fill = blue, inner sep = 1pt] (a\x) at ({90+360*\x/\K}:\R) {}; }
					%\foreach \x in {1,2,...,\K} {\node[above,sloped] at (a\x) {\tiny $y_\x$}}
					
					%	\node[] (a0) at ({360}:2*\R) {};
					
					\foreach \x [remember=\x as \y (initially \K)] in {1,2,...,\K} { 
						%point de l'hexagone intérieur et tracé de l'hexagone extérieur
						\path[draw] (a\x) -- (a\y) node[midway, circle, fill=green, inner sep = 1pt] (b\x) {}; 
					}
					\foreach \x [remember=\x as \y (initially \K)] in {1,2,...,\K} { 
						%tracé de l'hexagone intérieur
						\path[draw,dashed] (b\x) -- (b\y) node[midway] (c\x) {}; 
					}

					%		\foreach \i in {1,2,...,\K} {
						%		\begin{scope}[shift={(c\i)},rotate=-150+360*\i/\K] %*cos(180*(\K-2)/\K) %-135 pour K = 8
							%			\draw plot[ x={({sin(\thetaK/2)*\R/2/cos(\thetaK)},0)},y={(0,sqrt((1-1/(4*cos(\thetaK)*cos(\thetaK)) )/4 )*sin(\thetaK/2)*\R)},domain=-1:1,color=blue] (\x,{-1*\x*\x/2+1/2}); 
							%		\end{scope}
						%		}
					
					\foreach \i in {1,2,...,\K} {
						\begin{scope}[shift={(c\i)},rotate=240+360*\i/\K] %*cos(180*(\K-2)/\K)
							%					\draw[red] plot[ x={({sin(\thetaK/2)*\R/2/cos(\thetaK)},0)},y={(0,(1-sin(\thetaK/2)*sin(\thetaK/2) )*\R)},domain=-1:1] (\x,{-1*\x*\x/2+1/2}); 
							%					\draw[red] plot[ domain=-{cos(180/\K)*sin(180/\K)}:\Rc] (\x,{-1*\x*\x/2+1/2}); 
							\draw[red,thick] plot[ domain=-{cos(180/\K)*sin(180/\K)}:{cos(180/\K)*sin(180/\K)}] ({\x*\R},{-\R*\x*\x/2/cos(\alphaK)/cos(\alphaK)+\R*sin(\alphaK)*sin(\alphaK)/2}); 
						\end{scope}
					}
					\foreach \x [remember=\x as \y (initially \K)] in {1,2,...,\K} {	
						\fill[color=red,pattern=north east lines] (b\y.center)--(a\y.center)--(b\x.center)--cycle;
					}
					
				\end{tikzpicture}
			\end{minipage}
			$\quad$
			\begin{minipage}{0.3\textwidth}
				\centering
				\begin{tikzpicture}[scale=0.9]
					
					\def\R{2}
					\def\K{4}
					\def\alphaK{180/\K}

					%définition des points de l'hexagone extérieur
					\foreach \x in {1,2,...,\K} {\node[circle, draw=blue, fill = blue, inner sep = 1pt] (a\x) at ({45+360*\x/\K}:\R) {}; }
					%\foreach \x in {1,2,...,\K} {\node[above,sloped] at (a\x) {\tiny $y_\x$}}
					
					%	\node[] (a0) at ({360}:2*\R) {};
					
					\foreach \x [remember=\x as \y (initially \K)] in {1,2,...,\K} { 
						%point de l'hexagone intérieur et tracé de l'hexagone extérieur
						\path[draw] (a\x) -- (a\y) node[midway, circle, fill=green, inner sep = 1pt] (b\x) {}; 
					}
					\foreach \x [remember=\x as \y (initially \K)] in {1,2,...,\K} { 
						%tracé de l'hexagone intérieur
						\path[draw,dashed] (b\x) -- (b\y) node[midway, inner sep=0] (c\x) {}; 
					}

					%		\foreach \i in {1,2,...,\K} {
						%		\begin{scope}[shift={(c\i)},rotate=-150+360*\i/\K] %*cos(180*(\K-2)/\K) %-135 pour K = 8
							%			\draw plot[ x={({sin(\thetaK/2)*\R/2/cos(\thetaK)},0)},y={(0,sqrt((1-1/(4*cos(\thetaK)*cos(\thetaK)) )/4 )*sin(\thetaK/2)*\R)},domain=-1:1,color=blue] (\x,{-1*\x*\x/2+1/2}); 
							%		\end{scope}
						%		}
					
					\foreach \i in {1,2,...,\K} {
						\begin{scope}[shift={(c\i)},rotate=225+360*\i/\K] %*cos(180*(\K-2)/\K)
							%					\draw[red] plot[ x={({sin(\thetaK/2)*\R/2/cos(\thetaK)},0)},y={(0,(1-sin(\thetaK/2)*sin(\thetaK/2) )*\R)},domain=-1:1] (\x,{-1*\x*\x/2+1/2}); 
							%					\draw[red] plot[ domain=-{cos(180/\K)*sin(180/\K)}:\Rc] (\x,{-1*\x*\x/2+1/2}); 
							\draw[red,thick] plot[ domain=-{cos(180/\K)*sin(180/\K)}:{cos(180/\K)*sin(180/\K)}] ({\x*\R},{-\R*\x*\x/2/cos(\alphaK)/cos(\alphaK)+\R*sin(\alphaK)*sin(\alphaK)/2}); 
						\end{scope}
					}
					\foreach \x [remember=\x as \y (initially \K)] in {1,2,...,\K} {	
						\fill[color=red,pattern=north east lines] (b\y.center)--(a\y.center)--(b\x.center)--cycle;
					}

				\end{tikzpicture}
			\end{minipage}
			$\quad$
			\begin{minipage}{0.3\textwidth}
				\centering
				\begin{tikzpicture}[scale=0.75]
					\def\R{2}
					\def\K{6}
					\def\alphaK{180/\K}

					%définition des points de l'hexagone extérieur
					\foreach \x in {1,2,...,\K} {\node[circle, draw=blue, fill = blue, inner sep = 1pt] (a\x) at ({360*\x/\K}:\R) {}; }
					%\foreach \x in {1,2,...,\K} {\node[above,sloped] at (a\x) {\tiny $y_\x$}}
					
					%	\node[] (a0) at ({360}:2*\R) {};
					
					\foreach \x [remember=\x as \y (initially \K)] in {1,2,...,\K} { 
						%point de l'hexagone intérieur et tracé de l'hexagone extérieur
						\path[draw] (a\x) -- (a\y) node[midway, circle, fill=green, inner sep = 1pt] (b\x) {}; 
					}
					\foreach \x [remember=\x as \y (initially \K)] in {1,2,...,\K} { 
						%tracé de l'hexagone intérieur
						\path[draw,dashed] (b\x) -- (b\y) node[midway] (c\x) {}; 
					}

					%		\foreach \i in {1,2,...,\K} {
						%		\begin{scope}[shift={(c\i)},rotate=-150+360*\i/\K] %*cos(180*(\K-2)/\K) %-135 pour K = 8
							%			\draw plot[ x={({sin(\thetaK/2)*\R/2/cos(\thetaK)},0)},y={(0,sqrt((1-1/(4*cos(\thetaK)*cos(\thetaK)) )/4 )*sin(\thetaK/2)*\R)},domain=-1:1,color=blue] (\x,{-1*\x*\x/2+1/2}); 
							%		\end{scope}
						%		}
					
					\foreach \i in {1,2,...,\K} {
						\begin{scope}[shift={(c\i)},rotate=-150+360*\i/\K] %*cos(180*(\K-2)/\K)
							%					\draw[red] plot[ x={({sin(\thetaK/2)*\R/2/cos(\thetaK)},0)},y={(0,(1-sin(\thetaK/2)*sin(\thetaK/2) )*\R)},domain=-1:1] (\x,{-1*\x*\x/2+1/2}); 
							%					\draw[red] plot[ domain=-{cos(180/\K)*sin(180/\K)}:\Rc] (\x,{-1*\x*\x/2+1/2}); 
							\draw[red,thick] plot[ domain=-{cos(180/\K)*sin(180/\K)}:{cos(180/\K)*sin(180/\K)}] ({\x*\R},{-\R*\x*\x/2/cos(\alphaK)/cos(\alphaK)+\R*sin(\alphaK)*sin(\alphaK)/2}); 
						\end{scope}
					}
					\foreach \x [remember=\x as \y (initially \K)] in {1,2,...,\K} {	
						\fill[color=red,pattern=north east lines] (b\y.center)--(a\y.center)--(b\x.center)--cycle;
					}
					
				\end{tikzpicture}
			\end{minipage}
			\caption{\label{fig:AP2} For each case $\kappa=3$, 4 and 6, the inner curve drawn in red delimits a convex domain $\Dom{\Ck}$ inside $\Ck$. The red curve represents the limit shape of a $\zn$-gon taken under $\Un$, conditioned to be in convex position, when $n\to+\infty$. The curve can be drawn as follows: add the midpoint of the sides of the initial $\kappa$-gon, and between two consecutive midpoints, add the arc of parabola which is tangent to the sides and incident to these inner points. The sum of the hatched areas to the power $1/3$ corresponds to the supremum of affine perimeters (see an explanation in Appendix, \Cref{lem3}).}
		\end{figure}

	Denote by $d_H$ the Hausdorff distance on the set of compact sets of $\RR^2$, and for any tuple $\zzn\in(\RR^2)^n$, let $\conv(\zzn)$ be its convex hull.
	In the second main contribution of this paper, we detail the fluctuations of the $\zn$-gon having distribution $\Qn{\kappa}$ around its limit $\Dom{\Ck}$:
	\begin{Theoreme}\label{theo:cvhausd}
		Let $\kappa \geq 3$ be fixed, and let $\zn$ with distribution $\Qn{\kappa}$. When $n\to+\infty$, we have
		\[n^{1/2} d_H\left( \conv(\zn), \Dom{\Ck}  \right) \dd \Delta\]
		where $\Delta$ is a non trivial random variable.
	\end{Theoreme}
	This theorem will appear to be a consequence of the convergence of the fluctuations of  $\zn$-gon in distribution (at scale $1/\sqrt{n}$) in a functional space, which is stated in \Cref{thm:globfluc}. However, we renounce to state at this point this theorem since it would require to introduce too much material to do so, and then, we postpone this work to \Cref{sec:FLS}.
	
	Nonetheless, we disclose an element of the proof: the main idea is to partition each $\zzn$-gon of $\CVn$ into $\kappa$ suitable convex chains, one per corner of the initial polygon $\Ck$. Each of the convex chains will be shown to converge separately towards the arc of parabola associated with the corresponding "corner" of $\Ck$, as introduced in \Cref{fig:AP2}.\\
	
	The convergence results stated in  \Cref{thm:globfluc} are reminiscent of the limit theorems concerning lattice convex polygons: in this model, an integer $n$ is given, and a convex (lattice) polygon is a convex polygon contained in the square $[-n,n]^2$ and having vertices with integer coordinates (and any number of sides).
	Vershik asked whether it was possible to determine the number and typical shape of convex lattice polygons contained in  $[-n,n]^2.$ Three different solutions were brought to light by Bárány \cite{barany3}, Vershik \cite{vershik} and Sinai \cite{sinai} in 1994:
	
	A convex lattice polygon can be decomposed naturally in 4 parts (delimited by the extreme points in the North/East/South/West directions), which delimitate 4 ``polygonal convex lines'' between them. It is therefore natural to investigate the behaviour of these chains, that can be considered, in a first approximation, as convex chains  going from $(0,0)$ to $(n, n)$ in the square $[0,n]^2$ (up to rotations/translations). For these chains, they proved that when $n\to+\infty,$
	\begin{enumerate}
		\item the number of these convex polygonal lines is $\exp(3(\zeta(3)/\zeta(2))^{1/3} n^{2/3} + o(n^{2/3}))$ where $\zeta$ is the Riemann zeta function,
		\item the random number of vertices in such a chain is concentrated around $\left(\zeta(3)^2/\zeta(2)\right)^{-1/3} n^{2/3}$,
		\item the limit shape of such a chain, normalized in both direction by $n$, is an arc of parabola.
	\end{enumerate}
	These results were refined by Bureaux, Enriquez \cite{Bureaux_2016} in 2016, and generalized in larger dimensions by Bárány, Bureaux, Lund \cite{BARANY2018143} in 2018, as well as Buffière \cite{Buf} for zonotopes in 2023.
	
	On a related topic, the paper by Bodini, Jacquot, Duchon and Mutafchiev \cite{duduche} gives a characterization of digitally convex polyominoes using combinatorics on words.

	\paragraph{Random generation of a $\zn$-gon with distribution $\Qn{\kappa}$.}
	The naive way of sampling a $\zn$-gon with distribution $\Qn{\kappa}$ consists in rejection sampling, \ie sampling points that are $\mathbb{U}_n^{(\kappa)}$-distributed until they are in $\CVn$ (or in $\CVnofn$). This algorithm works fine for small values of $n$ but as $n$ grows, computation times are not acceptable anymore (by Theorem \ref{thm-1}, the probability of success is less than $\frac{k^n}{n^{2n}}$ for some constant $k$). In particular, the limit shape theorem proven by Bárány cannot be observed empirically with such a method. 
	
	A comprehensive understanding of the distribution $\Qn{\kappa}$ will allow us to determine another distribution $\Dtn$, for which we have an exact sampling algorithm (called $\kappa$-sampling and defined in \Cref{sec1}) that behaves asymptotically like $\Qn{\kappa}$, meaning $d_{V}(\Qn{\kappa},\Dtn)\underset{n\to+\infty}{\longrightarrow}0$ where $d_{V}$ is the total variation distance. The distribution $\Dtn$ is defined in \Cref{sec:DCP}, and can be viewed as $\Qn{\kappa}$ conditioned to satisfy a property which occurs with probability going to 1. This algorithm is asymptotically exact (in $n$, for $\kappa$ fixed), for the $\Qn{\kappa}$-sampling.
	\begin{Theoreme}\label{thm-2}
		The algorithm of $\kappa$-sampling samples a $n$-tuple of points with distribution $\Dtn$ with a complexity of $\mathcal{O}\left(n^{\kappa/2+1}\kappa\log(\kappa)\right).$
	\end{Theoreme}
	
	\paragraph{Contents of the paper.} In the second section of this paper, we analyze the properties of a $n$-tuple $\zzn\in\CVn$ in the light of a new geometric description. We derive in \Cref{sec:DCP} the distribution of the important variables of this geometric scheme, ouf of which we provide the proof of \Cref{thm-1} in \Cref{sec4}. \Cref{sec:FLS} is dedicated to the proof of \Cref{theo:cvhausd} and the understanding of the fluctuations of $\zzn$ around its limit. In \Cref{sec1}, we provide the aforementioned algorithm of $\kappa$-sampling and some alternative (more efficient) versions in the cases $\kappa=3$ and 4. As for the appendices, the first is dedicated to some complementary proofs of the paper, and the second provides a new demonstration of Valtr's formulas in the triangle and the parallelogram.

	\section{Geometrical aspects}
	\label{sec2}
	
	\paragraph{Notation.} In the sequel, $\kappa\geq3$ is considered to be fixed. We will work quite a lot with indices $j$ running through the set of integers $\entk$. By convention, in the case $j=1$, $j-1$ stands for $\kappa$, and when $j=\kappa$, $j+1$ stands for 1 (we do so to avoid tedious notation).\\
	
	We start by the definition of the "Equiangular Circumscribed Polygon" $\ECP(\zzn)$ associated to $\zzn\in \CVn$, any $n$-tuple of points in canonical convex order: as represented (in blue) on \Cref{fig2}, this is the polygon equal to the intersection of all equiangular polygons whose sides are parallel one by one to those of $\Ck$, and which contain $\zzn$.% (notation $\ECP$ stands for "Equiangular Circumscribed Polygon").

	We now define some quantities that will somehow allow the description of $\zzn$ in term of its circumscribed polygon (see also \Cref{fig2}).
	
	The distance from the $j^{th}$ side of $\Ck$  to $z[n]$ is denoted by $\ell_j:=\ell_j(\zzn)$. 
	The length of the side of $\ECP(\zzn)$ parallel to the X-axis is denoted by $c_1:=c_1(\zzn)$. Then, the consecutive side lengths of $\ECP(\zzn)$, sorted anticlockwise, are denoted by $c_1,c_2,\ldots,c_\kappa$, one or several $c_i$ being possibly zeroes.

	{	\begin{figure}[H]
\begin{minipage}{0.5\textwidth}
\centering
\begin{tikzpicture}[scale=1.6]

\draw (0,0)--(1.0491,0.0);
\draw (1.0491,0.0)--(1.7033,0.8202);
\draw (1.7033,0.8202)--(1.4698,1.8431);
\draw (1.4698,1.8431)--(0.5245,2.2983)--(-0.4206,1.8431)--(-0.6541,0.8202);
\draw (-0.6541,0.8202)--(0,0);

\draw[blue] (0.3333,0.2153)--(0.6622,0.2153) node[midway,above]{$c_1$};
\draw[blue] (0.6622,0.2153)--(1.256,0.974) node[midway,above,sloped]{$c_2$};
\draw[blue] (1.256,0.974)--(1.0419,1.912) node[midway,below,sloped]{$\ldots$};
\draw[blue] (1.0419,1.912)--(0.4181,2.212)--(-0.1803,1.924)--(-0.3744,1.0796);
\draw[blue] (-0.3744,1.0796)--(0.3333,0.2153) node[midway,above,sloped]{$c_\kappa$};

\draw[red,<->] (0.5,0)--(0.5,0.2153) node[midway,right]{$\ell_1$};
\draw[red,<->] (1.1,0.778)--(1.45,0.5) node[midway,above,sloped]{$\ell_2$};
\draw[red,<->] (-0.15,0.8)--(-0.45,0.56) node[midway,below,sloped]{$\ell_\kappa$};

\node[inner sep=0.7pt,circle,draw=red,fill=red] at (-0.225,1.661){};	
\node[inner sep=0.7pt,circle,draw=red,fill=red] at (0.535,2.111){};
\node[inner sep=0.7pt,circle,draw=red,fill=red] at (1.200,0.9798){};
\node[inner sep=0.7pt,circle,draw=red,fill=red] at (-0.00665,2.008){};
\node[inner sep=0.7pt,circle,draw=red,fill=red] at (1.0494,1.8795){};
\node[inner sep=0.7pt,circle,draw=red,fill=red] at (-0.2313,1.3391){};
\node[inner sep=0.7pt,circle,draw=red,fill=red] at (1.1736,1.1702){};
\node[inner sep=0.7pt,circle,draw=red,fill=red] at (0.7977,2.03){};
\node[inner sep=0.7pt,circle,draw=red,fill=red] at (0.8402,0.5037){};
\node[inner sep=0.7pt,circle,draw=red,fill=red] at (0.3333,0.2153){};
\node[inner sep=0.7pt,circle,draw=red,fill=red] at (-0.19,1.8832){};
\node[inner sep=0.7pt,circle,draw=red,fill=red] at (1.1,0.778){};
\end{tikzpicture}

\end{minipage}
\quad
\begin{minipage}{0.5\textwidth}
\centering
		\begin{tikzpicture}[scale=1.6]

\draw (0,0)--(1.0491,0.0);
\draw (1.0491,0.0)--(1.7033,0.8202);
\draw (1.7033,0.8202)--(1.4698,1.8431);
\draw (1.4698,1.8431)--(0.5245,2.2983)--(-0.4206,1.8431)--(-0.6541,0.8202);
\draw (-0.6541,0.8202)--(0,0);

\draw[blue] (0.3333,0.2153)--(1.17,1.32) node[midway,above]{};
\draw[blue] (1.17,1.32)--(1.0419,1.912) node[midway,below,sloped]{};
\draw[blue] (1.0419,1.912)--(0.4181,2.212)--(-0.1803,1.924)--(-0.3744,1.0796);
\draw[blue] (-0.3744,1.0796)--(0.3333,0.2153) node[midway,above,sloped]{};

\node[inner sep=0.7pt,circle,draw=red,fill=red] at (-0.225,1.661){};	
\node[inner sep=0.7pt,circle,draw=red,fill=red] at (0.535,2.111){};
\node[inner sep=0.7pt,circle,draw=red,fill=red] at (-0.00665,2.008){};
\node[inner sep=0.7pt,circle,draw=red,fill=red] at (1.0494,1.8795){};
\node[inner sep=0.7pt,circle,draw=red,fill=red] at (-0.2313,1.3391){};
\node[inner sep=0.7pt,circle,draw=red,fill=red] at (0.7977,2.03){};
\node[inner sep=0.7pt,circle,draw=red,fill=red] at (0.3333,0.2153){};
\node[inner sep=0.7pt,circle,draw=red,fill=red] at (-0.19,1.8832){};
\end{tikzpicture}
\end{minipage}
\captionof{figure}{\label{fig2}
On the left, in blue, $\ECP(\zzn)$ for a $n$-tuple taken in $\overset{\curvearrowleft}{\mathcal{C}_7}(n)$, with the distances $\ell[7]$ from the sides of $\mathfrak{C}_7$ to those of $\ECP(z[n]).$ The sides -which lengths are the tuple of values $c[7]$- of this latter polygon are colored in blue.
On the right, a $6$-sides $\ECP(\zzn)$ in $\mathfrak{C}_7$. One of the sides is reduced to a point: this happens when three consecutives distances $\ell_{{j-1}},\ell_j,\ell_{\wj}$ from $\Ck$ to a point in $\zzn$ are obtained on the same point $z_i$.}
\end{figure}
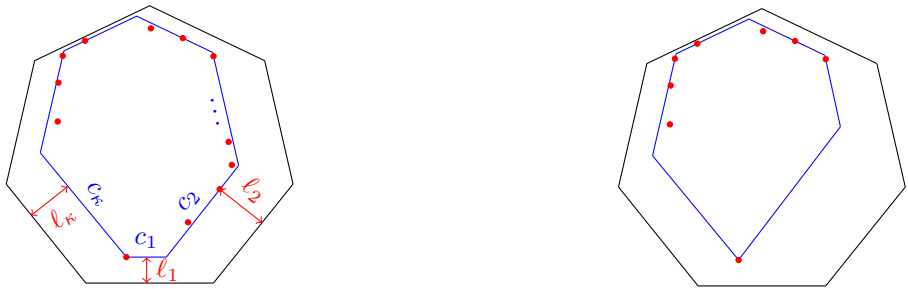}

	\begin{Remarque}\label{rem:zrng}
		If $\kappa=3$, the only possible internal polygons within $\mathfrak{C}_3$ are equilateral triangles. If $\kappa=4$, only rectangles are admitted. In both these cases, an internal polygon within $\Ck$ has exactly $\kappa$ sides. This is no longer true for $\kappa\geq5$, as we may see on the right picture of \Cref{fig2}. The number of "nonzero sides" of $\ECP(\zzn)$ is bounded above by $\kappa$, and below by 3 (in fact 4 for the $\kappa=4$ case, and it can technically be two, if all the points in $\zzn$ are aligned, but we can neglect this case).
	\end{Remarque}
	\paragraph{Some properties of equiangular circumscribed polygons.}
	
	A moment thought allows to see that $\ECP(\zzn)$ is characterized by the $\kappa$-tuple of distances $\ell[\kappa]:=(\ell_1(\zzn),\cdots,\ell_\kappa(\zzn))$, and that, in turn,  $\ell[\kappa]$ determines the side lengths of the $\ECP$, $c[\kappa]:=(c_1(\zzn),\cdots,c_\kappa(\zzn))$. In the sequel,	since there is no other set of points (except for $z[n]$) for which $\Lj$ or $\Cj$ would be defined, we deliberately  forget the mention to $\zzn$ when there is no ambiguity.
	\begin{Proposition}\label{prop1}
		Let $z[n]\in\CVn$, and the corresponding $\Cj,\Lj.$ 
		\bir
		\itr The $\Lj$ and $\Cj$ are related by the $\kappa$ equations
		\begin{equation}\label{eq:EEj} c_j=r_\kappa-\cl_j(\Lj),\quad \forall j \in \{1,\ldots,\kappa\},\end{equation}
		where for all $j\in\entk$, $\cl_j(\Lj):=\left(\ell_{j-1}+\ell_{j+1}+2\ell_j\cth\right)/\sth$ ($\cl$ standing for "linear combination")
		\itr The set $\mathcal{L}_\kappa=\ell[\kappa]\left(\CVn\right)$ (of all possible vectors $\ell[\kappa]$) is the set of solutions $\Lj$ to the system of inequalities
		\begin{align}\label{ineq}
			\Big\{\quad\cl_j(\Lj) \leq r_\kappa,\quad \forall j \in \{1,\ldots,\kappa\},
		\end{align}
		together with the conditions $\ell_j\geq0,j\in\entk.$
		\itr  The perimeter of the $\zzn$-gon satisfies 
		\[\sum_{j=1}^\kappa c_j + \frac{2(1+\cth)}{\sth}\sum_{j=1}^\kappa \ell_j = \kappa r_\kappa.\]
		\eir%\end{enumerate}
	\end{Proposition}
	
	\begin{proof}
		\textit{(i) }The formulas \eqref{eq:EEj}  may be deduced from routine computations on the angles and some proper applications of Thalès' theorem according to \Cref{fig8} below. Indeed, we have
		\[\frac{r_\kappa}{\ell_{j-1}/\sth+c_j+\ell_{j+1}/\sth}=\frac{a_\kappa}{a_\kappa+\ell_j/\sth},\]
		with (see \Cref{fig8}) $a_\kappa=\frac{-r_\kappa}{2\cth}$.
					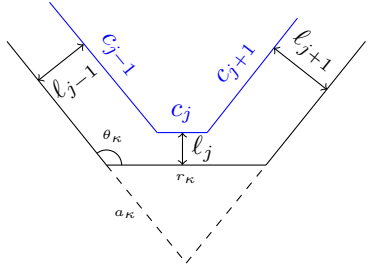
\begin{figure}[hbtp]
			\centering
			\begin{tikzpicture}[scale=2.0]
				
				\draw (0,0)--(1.0491,0.0) node[midway,below]{\tiny$r_\kappa$};
				\draw (1.0491,0.0)--(1.7033,0.8202);
				\draw (-0.6541,0.8202)--(0,0);
				\draw[dashed] (0,0)--(0.52455,-0.65) node[midway,left]{\tiny$a_\kappa$};
				\draw[dashed] (0.52455,-0.65)--(1.0491,0.);
				\draw[blue] (0.3333,0.2153)--(0.6622,0.2153) node[midway,above]{$c_j$};
				\draw[blue] (0.6622,0.2153)--(1.256,0.974) node[midway,above,sloped]{$c_{{j+1}}$};
				\draw[blue] (-0.3744,1.0796)--(0.3333,0.2153) node[midway,above,sloped]{$c_{{j-1}}$};
				
				\draw[<->] (0.5,0)--(0.5,0.2153) node[midway,right]{$\ell_j$};
				\draw[<->] (1.1,0.778)--(1.45,0.5) node[midway,above,sloped]{$\ell_{{j+1}}$};
				\draw[<->] (-0.15,0.8)--(-0.45,0.56) node[midway,below,sloped]{$\ell_{{j-1}}$};
				
				\draw (0.1,0) arc (0:128.5:0.1) node[midway,above]{\tiny$\theta_\kappa$} ;
				
			\end{tikzpicture}
			\caption{Characteristics of an internal polygon}
			\label{fig8}
		\end{figure}
		
		\textit{(ii) }It is clear that all elements of $\mathcal{L}_\kappa$ solve system \eqref{ineq}. Let $\Lj$ be a solution to \eqref{ineq}. Draw a $\kappa$-gon and add a straight line $(l_1)$ at distance $\ell_1$ parallel to the first side of $\Ck$, and another one $(l_2)$ at distance $\ell_2$ from $\Ck$'s second side. The intersection point of both these lines is a vertex $\bb_1$ of the $\ECP$. Since $c_2=r_\kappa-\cl_2(\Lj)\geq 0$, a second vertex $\bb_2$ of the $\ECP$ is at distance $c_2$ from $\bb_1$ on $(l_2).$ We can draw $(l_3)$ parallel to the third side of $\Ck$ passing through $\bb_2$. With $c_3=r_\kappa-\cl_3(\Lj)\geq 0$, we can set $\bb_3$ as the third vertex of the $\ECP$. Recursively, with all $c_j=r_\kappa-\cl_j(\Lj)\geq 0$, we get all vertices $(\bb[\kappa])$ and a full $\ECP.$ Hence, each solution of \eqref{ineq} is in $\mathcal{L}_\kappa$.
		
		To get $(iii)$, just sum all equations \eqref{eq:EEj} for all values of $j$.
		
	\end{proof}
	
	\paragraph{"Contact points".}
	
	For each $\zzn$ in $\CVn$, each side of $\ECP(\zzn)$ contains at least one element of $\{z_1,\cdots,z_n\}$. The $j^{th}$ "contact point" $\cp_j:=\cp_j(\zzn)$ is the point of $\{z_1,\cdots,z_n\}$, which is on the $j^{th}$ side of $\ECP(\zzn)$, and which is the smallest with respect to the lexicographical order among those with this property. Note that we will work with $n$-tuples $\zn$ of random variables, so that when $\zn$ is $\Qn{\kappa}$-distributed, there is a single point of $\{\bz_1,\cdots,\bz_n\}$ in the $j^{th}$ side of $\ECP(\zn)$ with probability 1, so that the particular choice of the lexicographical order has no importance.
	However, $\cp_j=\cp_{\wj}$ is possible, and has a positive probability for all $n\geq1$.
	
	Denote by $\bb_j$ the intersection point between the $j^{th}$ and $\wj^{th}$ sides of $\ECP(\zzn)$ for all $j\in\entk$ (the $j^{th}$ vertex of the $\ECP(\zzn)$). In the case where the $j^{th}$ side of $\ECP(\zzn)$ is reduced to a point, \ie $c_j=0,$ we have $\cp_{{j-1}}=\bb_{{j-1}}=\cp_j=\bb_j=\cp_{\wj}$. 
	
	The triangle with vertices $\cp_j,\cp_{\wj},\bb_j$ will be refered to as the $j^{th}$ corner of $\ECP(\zzn)$ or $\corner_j(\zzn)$ (see \Cref{fig:super} below for a summary). 
	
	\paragraph{Convex chains between contact points.}
	
	To get a comprehensive description of $\zzn$ with respect to its circumscribed polygon $\ECP(z[n])$, we need to enrich the decomposition between the contact points.
	In this regard, let $ABC$ be the triangle with vertices $A,B,C$ (taken in that order) in the plane, and for every integer $m\geq 0$, we denote by $\chain_m(ABC)$ the set of $(m+1)$-tuples $(A,z'_1,\ldots,z'_{m-1},B)$ such that $z'_1,\ldots,z'_{m-1}$ are in the triangle $ABC$, and $(A,z'_1,\ldots,z'_{m-1},B)\in\overset{\curvearrowleft}{\mathcal{Z}}_{m+1}$. Hence, $m$ is the number of vectors needed to join the points of any convex chain in $\chain_{m}(ABC)$. If $A=B,$ we define $\chain_m(ABC)$ only for $m=0$ as the set reduced to the trivial chain $(A,A).$
	We can now decompose the $\zzn$-gon between the contact points:
	
	For all $j\in \entk$, let $k:=k(j)\in\entn$ be such that $z_k=\cp_j$ and denote by $s_j:=s_j(z[n])$ the integer such that $z_{k+s_j}=\cp_{\wj}$ (eventually $s_j=0$); the quantity $s_j$ denotes the number of vectors joining the points of the convex chain $(z_k=\cp_j,\ldots,z_{k+s_j}=\cp_{\wj})$. We will refer to the tuple $\ssk$ as the {\bf size-vector} (see an example in \Cref{fig:super}).
	
	{\begin{figure}[H]
\centering
\begin{tikzpicture}[scale=2.5]

\draw (0,0)--(1.0491,0.0);
\draw (1.0491,0.0)--(1.7033,0.8202);
\draw (1.7033,0.8202)--(1.4698,1.8431);
\draw (1.4698,1.8431)--(0.5245,2.2983)--(-0.4206,1.8431)--(-0.6541,0.8202);
\draw (-0.6541,0.8202)--(0,0);

\draw[blue] (0.3333,0.2153)--(0.6622,0.2153) node[midway,above]{};
\draw[blue] (0.6622,0.2153)--(1.256,0.974) node[midway,above,sloped]{};
\draw[blue] (1.256,0.974)--(1.0419,1.912) node[midway,below,sloped]{};
\draw[blue] (1.0419,1.912)--(0.4181,2.212)--(-0.1803,1.924)--(-0.3744,1.0796);
\draw[blue] (-0.3744,1.0796)--(0.3333,0.2153) node[midway,above,sloped]{};

\node[inner sep=0.7pt,circle,draw=blue,fill=blue] (B2) at (0.6622,0.2153){};
\draw[blue] node[right] at (B2){\small $\bb_1$};

\node[inner sep=0.7pt,circle,draw=blue,fill=blue] (B3) at (1.256,0.974){};
\draw[blue] node[right] at (B3){\small $\bb_2$};

\node[inner sep=0.7pt,circle,draw=blue,fill=blue] (B4) at(1.0419,1.912){};
\draw[blue] node[above] at (B4){\small $\bb_3$};

\node[inner sep=0.7pt,circle,draw=blue,fill=blue] (B5) at (0.4181,2.212){};
\draw[blue] node[above] at (B5){\small $\bb_4$};

\node[inner sep=0.7pt,circle,draw=blue,fill=blue] (B6) at (-0.1803,1.924){};
\draw[blue] node[above] at (B6){\small $\bb_5$};

\node[inner sep=0.7pt,circle,draw=blue,fill=blue] (B7) at (-0.3744,1.0796){};
\draw[blue] node[left] at (B7){\small $\bb_6$};

\node[inner sep=1.2pt,circle,draw=blue,fill=red,line width=1pt] (CP1) at (0.3333,0.2153){};
\draw[blue] node[below] at (CP1){\small $\bb_7=\cp_7=\cp_1$};

\node[inner sep=1.2pt,circle,draw=blue,fill=red,line width=1pt] (CP2) at (1.1,0.778){};

\node[inner sep=1.2pt,circle,draw=blue,fill=red,line width=1pt] (CP2bis) at (1.0494,1.8795){};
\node[inner sep=1.2pt,circle,draw=blue,fill=red,line width=1pt] (CP3) at (0.7977,2.03){};
\node[inner sep=1.2pt,circle,draw=red,fill=red] (P5) at (0.535,2.111){};
\node[inner sep=1.2pt,circle,draw=blue,fill=red,line width=1pt] (CP4) at (-0.00665,2.008){};
\node[inner sep=1.2pt,circle,draw=blue,fill=red,line width=1pt] (CP5) at (-0.19,1.8832){};
\node[inner sep=1.2pt,circle,draw=red,fill=red] (P6) at (-0.225,1.661){};	
\node[inner sep=1.2pt,circle,draw=red,fill=red] (P7) at (-0.2313,1.3391){};

\fill[color=blue,pattern=north east lines] (CP1.center)--(B2.center)--(CP2.center)--cycle;
\node[inner sep=1.2pt,circle,draw=red,fill=red] (P1) at (0.8402,0.5037){};
\fill[color=blue,pattern=north east lines] (CP2.center)--(B3.center)--(CP2bis.center)--cycle;
\node[inner sep=1.2pt,circle,draw=red,fill=red] (P2) at (1.200,0.9798){};
\node[inner sep=1.2pt,circle,draw=red,fill=red] (P3) at (1.1736,1.1702){};
\draw[blue] (CP1)--(CP2)--(CP2bis);
\draw[red,line width=1pt] (CP1)--(P1)--(CP2)--(P2)--(P3)--(CP2bis)--(CP3)--(P5)--(CP4)--(CP5)--(P6)--(P7)--(CP1);

\end{tikzpicture}

\end{figure}

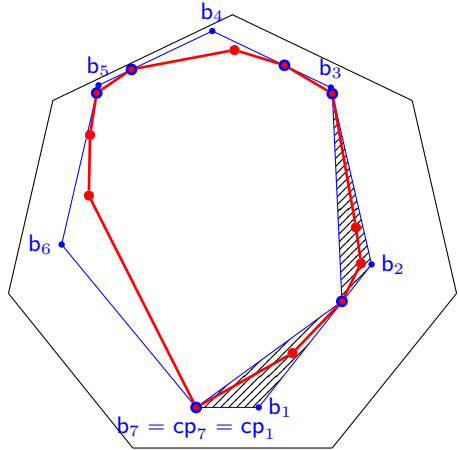
\captionof{figure}{In $\mathfrak{C}_7$, an example of $\zzn$-gon, the $\ECP(\zzn)$ and its vertices $\bb[7]$, and the first and second corner (the hashed areas). Here we have $s[7]=(2,3,1,2,1,3,0)$}\label{fig:super}}

	The main technical ingredient of the paper is now tackled in the following structural lemma:
	\begin{Lemma}\label{lem:rect} For a given $\ssk$ and $j\in\entk$, set $k=\sum_{t<j}s_t$ (so that $\cp_j=z_k$). Given $\ssk$, $\cp_j$ and $\cp_{\wj}$, the set of convex chains $(\cp_j,z_{k+1},\ldots,z_{k+s_j-1},\cp_{\wj})$ coincides with the set ${\sf \chain}_{s_j}(\corner_j)$. \par
		Hence, if $\zn$ is has distribution $\Qn{\kappa}$, conditional of $(\Bcp_j,\Bcp_{\wj},\ssj=s_j)$, the points in the tuple $({\bf z}_{k+1},\cdots, {\bf z}_{k+s_j-1})$ have the same distribution as that of $s_j-1$ points $({\bf z}'_{1},\ldots,{\bf z}'_{s_j-1})$ taken uniformly and independently in the triangle $\corner_j$, conditioned on $(\cp_j,{\bf z}'_{1},\ldots,{\bf z}'_{s_j-1},\cp_{\wj})$ is in $\overset{\curvearrowleft}{\mathcal{Z}}_{s_j+1}$.
	\end{Lemma}
	
	\begin{proof}
		The first statement is equivalent to say that there are no restriction on $(z_k,\ldots,z_{k+s_j})$ other than those defining $\chain_{s_j}(\corner_j)$:
		indeed, it is immediate to check that given two consecutive contact points  $z_k=\cp_j$ and $z_{s_j+k}=\cp_{\wj}$, the points of $\zzn$ are in convex position, if and only if both subsets $S_1$ and $S_2$ of points above and below the straight line joining $\cp_j$ and $\cp_{\wj}$ (where both $S_1$ and $S_2$ contains $\cp_j$ and $\cp_{\wj}$), are in convex position. An example is given in \Cref{fig:carre}.
		\bigbreak
		{
			\centering
\begin{tikzpicture}[scale=4.0]

\draw[thick,black] (-0.1,-0.1)--(1.,-0.1)--(1.,1.)--(-0.1,1)--(-0.1,-0.1);
\draw[thick,blue] (0,0)--(0.86,0)--(0.86,0.65)--(0,0.65)--(0,0);

\node[inner sep=1pt,circle,draw=red,fill=black,line width=1pt] (A) at (0,0.3){};	
\draw[red] node[right] at (A){\small $\cp_4$};
\node[inner sep=1pt,circle,draw=red,fill=black,line width=1pt] (B) at (0.8,0){};	
\draw[red] node[below] at (B){\small $\cp_{1}$};
\draw[red] (A)--(B);

\node[inner sep=1pt,circle,draw=green!80!blue,fill=black,line width=1pt] (A1) at (0.05,0.22){};
\node[inner sep=1pt,circle,draw=green!80!blue,fill=black,line width=1pt] (A2) at (0.15,0.15){};	
\node[inner sep=1pt,circle,draw=green!80!blue,fill=black,line width=1pt] (A3) at (0.32,0.07){};
\node[inner sep=1pt,circle,draw=green!80!blue,fill=black,line width=1pt] (A4) at (0.55,0.03){};		

\node[inner sep=1pt,circle,draw=yellow!70!orange,fill=black,line width=1pt] (B1) at (0.02,0.4){};
\node[inner sep=1pt,circle,draw=yellow!70!orange,fill=black,line width=1pt] (B2) at (0.08,0.5){};	
\node[inner sep=1pt,circle,draw=yellow!70!orange,fill=black,line width=1pt] (B3) at (0.3,0.63){};
\node[inner sep=1pt,circle,draw=yellow!70!orange,fill=black,line width=1pt] (B4) at (0.5,0.65){};		
\node[inner sep=1pt,circle,draw=yellow!70!orange,fill=black,line width=1pt] (B5) at (0.65,0.6){};
\node[inner sep=1pt,circle,draw=yellow!70!orange,fill=black,line width=1pt] (B6) at (0.8,0.45){};	
\node[inner sep=1pt,circle,draw=yellow!70!orange,fill=black,line width=1pt] (B7) at (0.86,0.3){};
\node[inner sep=1pt,circle,draw=yellow!70!orange,fill=black,line width=1pt] (B8) at (0.84,0.06){};		

\node[inner sep=1.5pt,circle,draw=red,fill=white,line width=1pt] (C1) at (1.3,0.45){};
\draw[black] node[left] at (C1){$S_1=$};
\node[inner sep=1.5pt,circle,draw=green!80!blue,fill=white,line width=1pt] (C2) at (1.43,0.45){};
\draw[black] node[left] at (C2){$+$};

\node[inner sep=1.5pt,circle,draw=red,fill=white,line width=1pt] (D1) at (1.3,0.35){};
\draw[black] node[left] at (D1){$S_2=$};
\node[inner sep=1.5pt,circle,draw=yellow!70!orange,fill=white,line width=1pt] (D2) at (1.43,0.35){};
\draw[black] node[left] at (D2){$+$};

\end{tikzpicture}

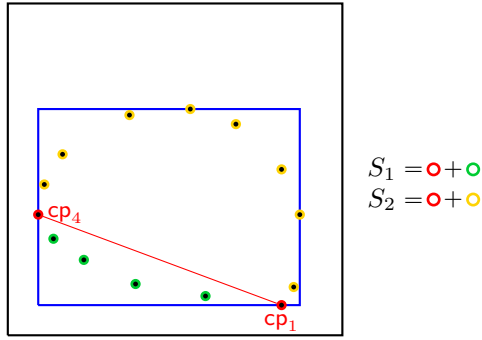
\captionof{figure}{\label{fig:carre}\small An example in the square case, where the $\ECP$ is always a rectangle.}
			
		}
		
		%given the fact that $\zzn$ is circumscribed in the convex polygon $\ECP(\zzn)$, for two points on the frontier of the domain $z_k=\cp_j$ and $z_{s_j+k}=\cp_{\wj}$, $\zzn$ are in convex position, if and only if both subsets of points above and below the straight line joining $\cp_j$ and $\cp_{\wj}$, together with $\cp_j$ and $\cp_{\wj}$, are in convex position.
		
		Because of this property, under  $\Qn{\kappa}$, the distribution of $({\bf z}_{k+1},\cdots, {\bf z}_{k+s_j-1})$, conditional to the position of $(\cp_j,\cp_{j+1})$, is the same as that of $({\bf z}_{k+1},\cdots, {\bf z}_{k+s_j-1})$ conditional to the position of all the other points, and it is then proportional to the Lebesgue measure on the set of point in convex position in the $j^{th}$ corner (that is in $\chain_{s_j}(\corner_j)$) which is equivalent to the second statement of the theorem.

		%The second statement comes from the fact that the conditional law is proportional to the initial law (restricted to the subspace, where the condition holds): 
	\end{proof}
	
	The law of chain $(A,{\bf u}_1,\cdots,{\bf u}_k,B)$ conditioned to be in $\chain_{k+1}(ABC)$ will be called the uniform law in $\chain_{k+1}(ABC)$.
	
	Denote by $\lowerrighttriangle$ the right triangle with vertices $(0,0),(1,1),(0,1).$ For a given nonflat triangle $ABC$, and an integer $m\geq1$, let $\Aff_{ABC}$ be the unique affine map that sends $ABC$ on $\lowerrighttriangle$ (meaning, $A,B,C$ on $(0,0),(1,0),(1,1)$ resp.). In the sequel, for $m\geq0$, we will denote $\nCC_m$ a random variable whose law is uniform in $\chain_{m}(\lt)$, and refer to this r.v. as a generic $\lt$-normalized convex chain of size $m$.
	
	From the fundamental property that affine maps preserve the convexity, we deduce \Cref{lem:aff}:
	\begin{Lemma}\label{lem:aff}$\bullet$ For a triangle $ABC$ (with non empty interior), and $k$ points ${\bf u}_1,\cdots,{\bf u}_k$ with distribution $\mathbb{U}^{(k)}_{ABC}$, the probability that the chain $(A,{\bf u}_1,\cdots,{\bf u}_k,B)$ is in $\chain_{k+1}(ABC)$ does not depend on $ABC$ (so that this value is the same as in the right triangle case).	
		
		$\bullet$ 	The map $\Aff_{ABC}$ sends $ABC$ on $\lt$, sends the uniform distribution on $ABC$ on that of $\lt$ (as well as $\mathbb{U}^{(k)}_{ABC}$ on $\mathbb{U}^{(k)}_{\lt}$) and sends the uniform distribution on $\chain_m(ABC)$ into that on $\chain_{m}(\lowerrighttriangle)$.
	\end{Lemma}
	
	\paragraph{The affine map $\varphi$.} In the following, we will work in the spirit of \Cref{lem:rect} by mapping every corner of an $\ECP$ in a right triangle. Let $\zzn\in\CVn$, consider the side lengths $\Cj$ of $\ECP(\zzn)$, $\bb[\kappa]$ the vertices of $\ECP(\zzn)$, and we impose the condition $s_j>0$ for convenience, so to have $\cp_j\neq\cp_{\wj}$ (the mapping is still definable otherwise). Let $A'_j=(0,c_j),B'_j=(0,0),C'_j=(c_{\wj},0)$, and define $\varphi_j$ as the unique affine map  that sends respectively $\bb_{{j-1}},\bb_j,\bb_{\wj}$ onto $A'_j,B'_j,C'_j$.
	\[A'_j:=\varphi_j(\bb_{{j-1}}),\quad B'_j:=\varphi_j(\bb_j),\quad C'_j:=\varphi_j(\bb_{\wj}).\]
	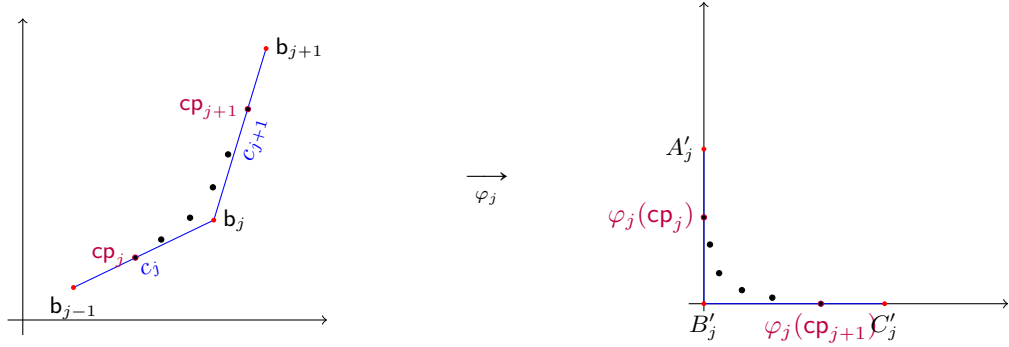
\begin{figure}[H]
		\begin{minipage}{0.5\textwidth}
			\centering
			\begin{tikzpicture}[scale=2.]

				\draw[black,->] (-0.1,0) -- (2.,0.);
				\draw[black,<-] (0.,2.) -- (0.,-0.1);
				
				\draw[blue] (0.3333,0.2153)--(1.2560,0.6622) node[midway,below,sloped]{\small $c_j$};
				\draw[blue] (1.2560,0.6622)--(1.6,1.8) node[midway,below,sloped]{\small $c_{{j+1}}$};
				
				\node[inner sep=0.7pt,circle,draw=purple,fill=black] (S) at (0.7402,0.4137){};
				\node[left,sloped] at (S) {\color{purple} \small $\cp_{j}$};
				\node[inner sep=0.7pt,circle,draw=black,fill=black] at (0.9102,0.5337){};
				\node[inner sep=0.7pt,circle,draw=black,fill=black] at (1.1,0.678){};
				\node[inner sep=0.7pt,circle,draw=black,fill=black] at (1.250,0.8798){};
				\node[inner sep=0.7pt,circle,draw=black,fill=black] at (1.350,1.098){};
				\node[inner sep=0.7pt,circle,draw=purple,fill=black] (T) at (1.48,1.398){};
				\node[left,sloped] at (T) {\color{purple}\small $\cp_{{j+1}}$};
				
				\node[inner sep=0.5pt,circle,draw=red,fill=red] (I) at (0.3333,0.2153){};
				\node[below] at (I) {\small ${\bb}_{{j-1}}$};
				\node[inner sep=0.5pt,circle,draw=red,fill=red] (U) at (1.2560,0.6622){};
				\node[right] at (U) {\small $\bb_j$};
				\node[inner sep=0.5pt,circle,draw=red,fill=red] (B) at (1.6,1.8){};
				\node[right] at (B) {\small ${\bb}_{{j+1}}$};

			\end{tikzpicture}
			
		\end{minipage}
		$\underset{\varphi_{j}}{\longrightarrow}$
		\begin{minipage}{0.5\textwidth}
			\centering
			\begin{tikzpicture}[scale=2.]

				\draw[black,->] (-0.1,0) -- (2.,0.);
				\draw[black,<-] (0.,2.) -- (0.,-0.05);
				
				\draw[blue] (0.,1.025)--(0.,0.) node[midway,below,sloped]{};
				\draw[blue] (0.,0.)--(1.1886,0.) node[midway,below,sloped]{};

				\node[inner sep=0.7pt,circle,draw=purple,fill=black] (S) at (0.,0.5725){};
				\node[left,sloped] at (S) {\color{purple} $\varphi_j(\cp_j)$};
				\node[inner sep=0.7pt,circle,draw=black,fill=black] at (0.04,0.3925){};
				\node[inner sep=0.7pt,circle,draw=black,fill=black] at (0.1,0.2025){};
				\node[inner sep=0.7pt,circle,draw=black,fill=black] at (0.25,0.09){};
				\node[inner sep=0.7pt,circle,draw=black,fill=black] at (0.45,0.04){};
				\node[inner sep=0.7pt,circle,draw=purple,fill=black] (T) at (0.7691,0.){};
				\node[below,sloped] at (T) {\color{purple} $\varphi_j(\cp_{{j+1}})$};

				\node[inner sep=0.5pt,circle,draw=red,fill=red] (I) at (0.,1.025){};
				\node[left] at (I) {\small $A'_j$};
				\node[inner sep=0.5pt,circle,draw=red,fill=red] (U) at (0.,0.){};
				\node[left,below] at (U) {\small $B'_j$};
				\node[inner sep=0.5pt,circle,draw=red,fill=red] (B) at (1.1886,0.){};
				\node[below] at (B) {\small $C'_j$};

			\end{tikzpicture}
		\end{minipage}
		\captionof{figure}{\label{fig:affine}The map $\varphi_{j}$}
	\end{figure}
	
	The map $\varphi_{j}$ can be seen as the composition of a rotation of the $j^{th}$ corner so to place the $2^{nd}$ side for the clockwise order parallel to the X-axis and then, the straightening of the angle of the obtained triangle to get a right triangle, and a translation (that does not play any role). Therefore, the Jacobian determinant of $\varphi_j$ is the determinant of the matrix $A_j(\thk)$ , defined as follows :
	\begin{align}\label{eq:Aj}
		A_j(\thk):=  \underbrace{\begin{pmatrix} 
				1&\cos(\beta_\kappa) \\ 0&\sin(\beta_\kappa)
			\end{pmatrix}^{-1}}_{\text{straightening}}
		\underbrace{\begin{pmatrix} 
				\cos(-j\beta_\kappa)&\sin(j\beta_\kappa) \\ -\sin(j\beta_\kappa)&\cos(-j\beta_\kappa) 
		\end{pmatrix}}_{\text{rotation}}
		,
	\end{align}
	where $\beta_\kappa=\pi-\theta_\kappa$.
	The Jacobian determinant of $\varphi_{j}$ is thus
	\begin{align}
		\Jac\varphi_j=\det\left(A_j(\thk)\right)=\frac{1}{\sth}.
	\end{align}
	%	\begin{Remarque}
		%		The map $\Aff_{j}:=\Aff_{\corner_j}$ (that we will use mainly in \Cref{sec:FLS}) may be defined in a similar way for all $j\in\entk$:
		%		\begin{align}\label{eq:aff}
			%			(0,0)=\Aff_{j}(\cp_j),\quad (1,0)=\Aff_{j}(\bb_j),\quad (1,1)=\Aff_{j}(\cp_{\wj}),
			%		\end{align}
		%		as introduced in \Cref{lem:rect}. The Jacobian of this map is a bit more complicated than this of $\varphi_j$, but our use of this map is different: for $\zzn\in\CVn$, $\Aff_j$ maps the $j^{th}$ corner of $\ECP(\zzn)$ in the right triangle $\lowerrighttriangle$ so to reduce the study of the convex chain in $\corner_j(\zzn)$ to this of a convex chain in $\lt.$
		%	\end{Remarque}

	\paragraph{Encoding convex chains in a triangle by simplex products.} 
	Denote for all $\ell\in\RR_+$ and $k\in\NNN$, the simplex
	\begin{align}\label{eq:P}
		P[\ell,k]&=\left\{(a_1,\ldots,a_k),0<a_1<\ldots<a_k<\ell\right\},
	\end{align} and the "reordered" simplex
	\begin{align}\label{eq:I}
		I [\ell,k]&=\left\{(b_1,\ldots,b_{k}) \text{ where }0<b_1<\ldots<b_k<\ell,\text{ and } \sum_{i=1}^kb_i=\ell\right\}.
	\end{align}
	An element $(a_1,\ldots,a_k)$ of the set $P[\ell,k]$ encodes $k$ points on the segment $[0,\ell]$, whereas an element $(b_1,\ldots,b_{k})$ of $I[\ell,k]$ must be seen as $k$ increasing intervals partitionning the segment $[0,\ell]$.
	
	Nonetheless, the set $I[\ell,k]$ can actually be identified as a subset of $P[\ell,k-1]$ whose increments are increasing. Indeed, if $(a_1,\ldots,a_{k-1})$ is in $P[\ell,k-1]$ and is such that $a_1<a_2-a_1<\ldots<a_{k-1}-a_{k-2}<\ell-a_{k-1},$ then $(a_1,a_2-a_1,\ldots,a_{k-1}-a_{k-2},\ell-a_{k-1})$ is in $I[\ell,k]$. We will sometimes do this identification to present some bijections, nevertheless it is important to remember that topologically, $I[\ell,k]$ remains a surface in $\mathbb{R}^k$, a $(k-1)$-dimensional simplex, and that useful bijections in measure theory are those for which their Jacobian determinant may be computed.
	
	Note that the Lebesgue measures of these sets are:
	\begin{align}\label{equ:simplex}
		\Leb_{k}(P[\ell,k])= \int_{P[\ell,k]}\mathrm{d}a[k]=\frac{\ell^{k}}{k!}\quad\text{ and }\quad\Leb_{k-1}(I[\ell,k])=\int_{I[\ell,k]}\mathrm{d}a[k-1]=\frac{\ell^{k-1}}{k!(k-1)!}
	\end{align}
	where for any tuple $a[k]=(a_1,\ldots,a_k)$, notation $\mathrm{d}a[k]$ stands for $\prod_{i=1}^k\mathrm{d}a_i.$

	\paragraph{A $m!$ to 1 map, piecewise linear, from $\chain$ to a simplex product.}
	Let $abc$ be a right triangle in $c$ of $\RR^2$, and denote by $d_1=ac,d_2=bc$ (the distance).\\
	For any convex chain $\left(a,u_1,\cdots,u_{m-1},b\right)\in\chain_{m}(abc)$ we may consider the vectors $v[m]$ joining the points of the convex chain in their order of appearance. Then, let the $x$ and $y$-coordinates  $x[m],y[m]$ of these vectors as $x_i=\pi_1(v_i),y_i=\pi_2(v_i)$ for all $i\in\{1,\ldots,m\},$ and let $(\overset{\circ}{x}_{1}<\ldots< \overset{\circ}{x}_{m}),(\overset{\circ}{y}_{1}<\ldots< \overset{\circ}{y}_{m})$ be the tuples of reordered coordinates.
	{\begin{figure}[H]
\centering
\begin{tikzpicture}[scale=7.]

\draw[black,->] (-0.1,0) -- (0.9,0.);
\draw[black,<-] (0.,0.7) -- (0.,-0.1);

\node[inner sep=1.3pt,circle,draw=purple,fill=black] (S) at (0.,0.5725){};
\node[left,sloped] at (S) { $a$};
\node[inner sep=1.2pt,circle,draw=black,fill=black] (A1) at (0.04,0.3925){};
\node[inner sep=1.2pt,circle,draw=black,fill=black] (A2) at (0.1,0.2025){};
\node[inner sep=1.2pt,circle,draw=black,fill=black] (A3) at (0.25,0.09){};
\node[inner sep=1.2pt,circle,draw=black,fill=black] (A4) at (0.45,0.04){};

\node[color=red] (X1) at (0.04,-0.1){\tiny $|$};
\node[color=red,below] at (X1) {};
\node[color=red] (X2) at (0.1,-0.1){\tiny$|$};
\node[color=red,below] at (X2) {};
\node[color=red] (X3) at (0.25,-0.1){\tiny$|$};
\node[color=red,below] at (X3) {};
\node[color=red] (X4) at (0.45,-0.1){\tiny$|$};
\node[color=red,below] at (X4) {};
\node[color=red] (X5) at (0.7691,-0.1){\tiny$|$};

\draw[red,thick] (0.,-0.1)--(0.7691,-0.1);
\draw[red, thick] (0.,-0.1)--(X1) node[midway,below]{\small $x_1$};
\draw[red, thick] (X1)--(X2) node[midway,below]{\small $x_2$};
\draw[red, thick] (X2)--(X3) node[midway,below]{\small $x_3$};
\draw[red, thick] (X3)--(X4) node[midway,below]{\small $x_4$};
\draw[red, thick] (X4)--(X5) node[midway,below]{\small $x_5$};

\node[color=red] at (-0.1,0.5725){$-$};

\node[color=red] (Y1) at (-0.1,0.3925){$-$};
\draw[red, thick] (-0.1,0.5725)--(-0.1,0.3925) node[midway,left]{\small $y_1$};

\node[color=red] (Y2) at (-0.1,0.2025){$-$};

\draw[red, thick] (-0.1,0.3925)--(-0.1,0.2025) node[midway,left]{\small $y_2$};

\node[color=red,left] at (Y2) {};
\node[color=red] (Y3) at (-0.1,0.09){$-$};
\draw[red, thick] (-0.1,0.2025)--(-0.1,0.09) node[midway,left]{\small $y_3$};

\node[color=red,left] at (Y3) {};

\draw[red, thick] (-0.1,0.09)--(-0.1,0.04) node[midway,left]{\small $y_4$};
\node[color=red] (Y4) at (-0.1,0.04){$-$};

\draw[red, thick] (-0.1,0.04)--(-0.1,0.) node[midway,left]{\small $y_5$};
\node[color=red,left] at (Y4) {};

\node[inner sep=1.3pt,circle,draw=purple,fill=black] (T) at (0.7691,0.){};
\node[below,sloped] at (T) { $b$};

\draw[red,thick,->] (S) -- (A1);		
\draw[red,thick,->] (A1) -- (A2);	
\draw[red,thick,->] (A2) -- (A3);	
\draw[red,thick,->] (A3) -- (A4);	
\draw[red,thick,->] (A4) -- (T);		

\node[inner sep=1.3pt,circle,draw=red,fill=red] (U) at (0.,0.){};
\node[below left] at (U) { $c$};

\draw[blue,very thick] (T)--(U) node[midway,below]{ $d_2$};
\draw[blue,very thick] (S)--(U) node[midway,left]{ $d_1$};
\draw[blue,thick] (T)--(S) node[midway,below]{};

\end{tikzpicture}
\end{figure}

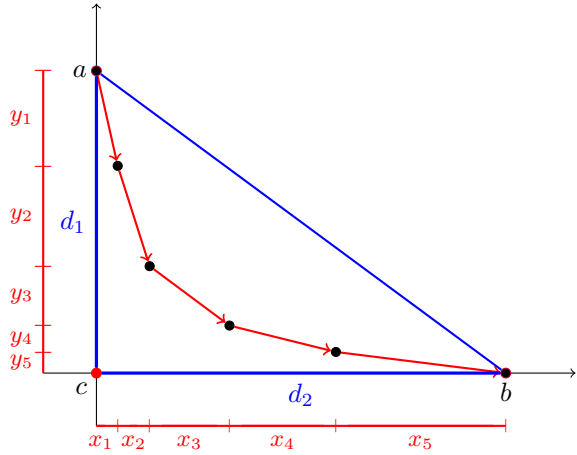
\captionof{figure}{A convex chain in a right triangle $abc$.}

			}
	Consider then the surjective mapping \footnote{Note that we will be working with vectors that are randomly distributed and such that $\P\left(\exists i\neq j \text{ s.t. } x_i=x_j \text{ or } y_i=y_j\right)=0$ a.s., which ensures that the map $\Orderr^{(m)}_{abc}$ is well-defined.}
	\begin{align}
		\app{\Orderr^{(m)}_{abc}}{\chain_{m}(a,b,c)}{I[d_1,m]\times I[d_2,m]}{\left(a,u_1,\cdots,u_{m-1},b\right)}{(\overset{\circ}{x}[m],\overset{\circ}{y}[m])}.
	\end{align}
	This map is piecewise linear (see \Cref{rmk:loclin} below) and has Jacobian determinant $1$ since we are in a right triangle. 
	
	\begin{Definition}[Notion of piecewise linear map]\label{rmk:loclin}
		A map $g:E\subset\RR^n\to \RR^n$ is said to be piecewise linear if:\\
		$\bullet$ there exists a collection of polytopes $(P_i)_{i\in\ent{1}{m}}$ such that $\bigcup_{i=1}^m P_i = E$ and the interiors $P^\circ_i$ of the sets $(P_i)_{i\in\ent{1}{m}}$ are pairwise disjoint,\\
		$\bullet$ for all $i\in\ent{1}{m},$ $g:P^\circ_i\to \RR^n $ is linear.\\
		The {\bf piecewise differentiability} may be defined in an analogous way (here the term "piecewise" must be understood as $g$ being piecewise differentiable on every $P^\circ_i$).\\
		Of course, "polytopes" can be replaced by more general Lebesgue measurable sets, whose union of interiors would partition $E$, up to a Lebesgue negligible set.
	\end{Definition}
	
	\begin{Remarque} Consider the mapping \[\app{g}{\RR^3}{\RR^3}{(x_1,x_2,x_3)}{(x_{(1)},x_{(2)},x_{(3)})}\] where $(x_{(1)}\leq x_{(2)}\leq x_{(3)})$ is the sorted sequence $(x_1,x_2,x_3)$. The map $g$ is clearly not linear, however, for any $(x\neq y\neq z)\in\RR^3$, there exists a neighborhood of $(x,y,z)$ on which $g$ is actually linear. At several places in the paper, we use this kind of reordering map and so we use the terminology {\bf piecewise linearity} (and \textbf{piecewise differentiability}) in these cases.
		
	\end{Remarque}
	
	\begin{Lemma}\label{lem:order1}
		Let $(\overset{\circ}{x}[m],\overset{\circ}{y}[m])\in I[d_1,m]\times I[d_2,m]$. We have
		\begin{align}
			\#\left(\Orderr^{(m)}_{abc}\right)^{(-1)}\left(\overset{\circ}{x}[m],\overset{\circ}{y}[m]\right)=m!
		\end{align}
	\end{Lemma}
	\begin{proof}
		There are $m!$ distinct ways of pairing every element of $\overset{\circ}{x}[m]$ with one of $\overset{\circ}{y}[m]$ to form $m$ vectors. There exists a unique order that sorts these vectors by increasing slope. This forms the boundary of a convex chain whose vertices in canonical convex order $(a,u_1,\ldots,u_{m-1},b)$ are in $\chain_{m}(abc)$.
	\end{proof}
	
	This lemma allows to obtain the Lebesgue measure of the set $\chain_{m}(a,b,c)$ by carrying the Lebesgue measure of $I[d_1,m]\times I[d_2,m]$ onto $\chain_{m}(abc)$. In order to compute $\chain_{m}(abc)$'s Lebesgue measure, we need to identify the convex chains with $m$ vectors as a subset of $\RR^{2(m-1)}$ (so that its dimension is $2(m-1)$, and appears as so). Introduce then $\chain'_m(a,b,c)=\{(z_1,\cdots,z_{m-1})~: (a,z_1,\cdots,z_{m-1},b)\in \chain_{m}(abc)\}$. By a change of variables we have
	\begin{align}\nonumber
		\Leb_{2(m-1)}\left(\chain'_{m}(abc)\right)&=\int_{(\RR^2)^{m-1}}\mathbb{1}_{\chain'_{m}(abc)}\mathrm{d}z[m-1]\\\nonumber
		&=m!(m-1)!\cdot\Leb_{m-1}(I[d_1,m])\cdot\Leb_{m-1}(I[d_2,m])\\
		&=\frac{(d_1d_2)^{m-1}}{m!(m-1)!}.
	\end{align}
	
	Note that the term in $(m-1)!$ on the second line appears for the relabelling of points $(u_1,\cdots,u_{m-1})$, and $m!$ because of \Cref{lem:order1}.
	
	\paragraph{Intuitively,} for $\zn$ with distribution $\Qn{\kappa}$, these lemmas reveal that conditional on the position of $\ECP(\zn)$, $\cp[\kappa](\zn)$ and $\ssk(\zn)$ (all together), the convex chains in each corner are independent. In this way, each corner can be considered separately, and by mapping the $j^{th}$ corner of the $\ECP(\zn)$ on $A'_j,B'_j,C'_j$ with $\varphi_j$ (see \eqref{eq:Aj}), we are brought back to the (simpler) study of a convex chain in a right triangle.
	However, even if this big picture is useful to understand the limit shape theorem, it is unfortunately not sufficient to compute the full asymptotic expansion of $\pk$, mainly because of the fact that the joint distribution of $(\Lj(\zn),\ssk(\zn),\cp[\kappa](\zn))$ is intricate and need to be understood. Hence, we need to introduce some more tools to work with the joint distribution.

	\paragraph{Number of sides of $\ECP(z[n])$.}
	For $z[n]\in\CVn$ and the corresponding $\Cj$, define the map $\NZS$ as
	\begin{align}
		\app{\NZS}{\CVn}{\mathcal{P}(\entk)}{z[n]}{\big\{i;c_i\neq0\big\}}
	\end{align} 
	which records the NonZero Sides indices of $\ECP(\zzn).$ 
	Let us also set\[\NN_\kappa(n)=\big\{\ssk\in\NN, \text{ such that }s_1+\ldots+s_\kappa=n \text{ and } s_{{j-1}}+s_j\neq 0 \text{ for all }j\in\entk\big\},\] and define 
	\[\CVnfull:=\left\{\zzn\in\CVn, \text{ such that }\ssk(\zzn)\in\Nkn\right\}.\]
	The following proposition states an equivalent condition on $\sk$ to ensure a "full-sided" $\ECP$:
	\begin{Proposition}\label{prop3}
		Let $\zn$ has distribution $\Qn{\kappa}$. Then, 
		$\NZS(\zn)=\entk$ is equivalent to $\ssk(\zn)\in\NN_\kappa(n).$
	\end{Proposition}
	
	\begin{proof}
		Suppose that the $\ECP(\zzn)$ has exactly $\kappa$ nonzero sides, \ie if $\mathbf{c}[\kappa]=c[\kappa](\zn)$, we have $\mathbf{c}_j>0$ for all $j\in\entk.$ Inside the tuple $\zn$, consider for all $j\in\entk$ the contact points $\Bcp_{{j-1}}$, $\Bcp_{j}$ and $\Bcp_{\wj}.$ A small picture suffices to see that we cannot have $\Bcp_{{j-1}}=\Bcp_{j}=\Bcp_{\wj}$ for this is equivalent to $\mathbf{c}_j=0,$ and thus is also equivalent to the fact that there exists a nonzero vector leading either $\Bcp_{{j-1}}$ to $\Bcp_{j}$ (\ie $\mathbf{s}_{{j-1}}\geq1$), or $\Bcp_{j}$ to $\Bcp_{\wj}$ (\ie $\mathbf{s}_j\geq1$).
	\end{proof}
	Therefore the set $\CVnfull$ admits another equivalent definition as
	\[\CVnfull:=\left\{\zzn\in\CVn, \text{ such that }\NZS(\zzn)=\entk\right\}.\] The following lemma ensures then that the overwhelming mass of $n$-tuples $\zzn\in\CVn$ is actually contained in $\CVnfull.$
	\begin{Lemma}\label{lem1}
		Let $\zn$ has distribution $\Un$. Denote by $\ptk:=n!~\mathbb{P}\left(\zn\in\CVnfull\right)$ the probability that the $\zn$ are in convex canonical order and additionally, that their $\ECP$ has $\kappa$ nonzero sides. We have \[\pk\underset{n\to+\infty}{\sim}\ptk.\]
	\end{Lemma}
	
	The proof of this result requires several arguments relative to Bárány's limit shape Theorem so we send the interested reader to \Cref{ann:LST} for a complete overview of the proof.
	
	\begin{Remarque}\label{rmk:impo}
		\Cref{lem1} is of paramount importance since it allows us to neglect a subset of $\CVn$ which Lebesgue measure becomes insignificant with regard to that of $\CVn$ as $n\to+\infty$. 
		To do so, we will assume that all $n$-tuples of points $z[n]$ we are working with are in $\CVnfull$ so to force -by \Cref{prop3}- the number of nonzero sides of $\ECP(\zzn)$ to be $\kappa$.
	\end{Remarque}
	
	\begin{Notation}\label{not:qtn}
		Denote by $\Dtn$ the distribution of a $n$-tuple of random points $\zn$ with distribution $\Un$, conditioned to be in $\CVnfull$.
	\end{Notation}
	\section{Distribution of a convex $\zzn$-gon}\label{sec:DCP}

	\paragraph{Notation.}From now on, we will work at size-vector $\ssk\in\Nkn$ fixed. In this regard, we denote by $\CVnij$ the subset of all $\zzn\in\CVnfull$ such that $\ssk(\zzn)=\ssk$, \ie the set of $n$-tuples $\zzn\in\CVn$ with a prescribed size-vector $\ssk$. 	We will write $N_j=s_j+s_{\wj}-1$ for all $j\in\entk$.
	
	The work with a "prescribed size-vector" is not only a technical tool: as a matter of fact, our analysis deeply relies on the computation of the distribution of the size-vector, and then, to the description of the chains with a prescribed size vector (a foretaste has been given in \Cref{lem:rect} for instance). We will see further in the paper, that the fluctuations of the $\zzn$-gon in each corner depend also on the fluctuations of the vector $\ssk$, so that this kind of considerations cannot be avoided.
	
	\subsection{Encoding $\zzn$-gons into side-partitions of $\ECP(\zzn)$}
	
	\paragraph{A new geometrical description : convex chains between contact points, convex chains in a right triangle and simplex product.}Let us fix now $\ssk\in\Nkn$. For all $z[n]\in\CVnij$, consider the corresponding side lengths $\Cj$ (which are thus all nonzero) and define for all $j\in\entk$ the {\bf side-partition} $u^{(j)}[N_j,c_j]=(u_1^{(j)},\ldots,u_{N_j}^{(j)})$ of the $j^{th}$ side length $c_j$ of $\ECP(\zzn)$, which is defined in \Cref{fig7} below and is an element of $P[c_j,N_j]$. For any {\bf side-partition} $u^{(j)}[N_j,c_j]$ defined as such, we build in addition the values $u^{(j)}_0:=0$, $u_{N_j+1}^{(j)}=c_j$ so that we have $u_0^{(j)}<u_1^{(j)}<\ldots<u_{N_j}^{(j)}<u_{N_j+1}^{(j)}$.
	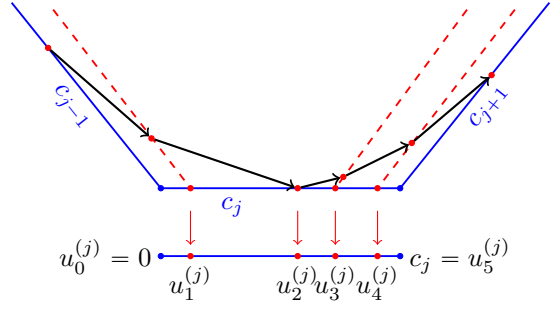
\begin{figure}[H]
		\centering
		\begin{tikzpicture}[scale=3.0]
			
			\draw[blue,line width=0.25mm] (0,0)--(1.0491,0.0) node[midway,below,pos=0.3]{$c_j$};
			\draw[blue,line width=0.25mm] (1.0491,0.0)--(1.7033,0.8202)node[midway,below,sloped]{$c_{{j+1}}$};
			\draw[blue,line width=0.25mm] (-0.6541,0.8202)--(0,0)node[midway,below,sloped]{$c_{{j-1}}$};
			
			\node[inner sep=0.7pt,circle,draw=blue,fill=blue] (R) at (1.0491,0.0){};
			\node[below] at (R) {};
			\node[inner sep=0.7pt,circle,draw=blue,fill=blue] (T) at (0,0){};
			\node[below] at (T) {};
			
			\node[inner sep=0.7pt,circle,draw=red,fill=red] (A) at (-0.4941,0.6202){};
			\node[inner sep=0.7pt,circle,draw=red,fill=red] (B) at (-0.041,0.2202){};
			\node[inner sep=0.7pt,circle,draw=red,fill=red] (C) at (0.6,0.){};
			\node[inner sep=0.7pt,circle,draw=red,fill=red] (D) at (0.8,0.05){};
			\node[inner sep=0.7pt,circle,draw=red,fill=red] (E) at (1.1,0.2){};
			\node[inner sep=0.7pt,circle,draw=red,fill=red] (F) at (1.45,0.5){};
			
			\node[inner sep=0.7pt,circle,draw=red,fill=red] (B1) at (0.13,0.){};
			\node[inner sep=0.7pt,circle,draw=red,fill=red] (D1) at (0.765,0.){};
			\node[inner sep=0.7pt,circle,draw=red,fill=red] (E1) at (0.95,0.){};
			
			\node[inner sep=0.7pt] (B2) at (-0.48,0.80){};
			\node[inner sep=0.7pt] (D2) at (1.35,0.8){};
			\node[inner sep=0.7pt] (E2) at (1.56,0.8){};
			
			\draw[red,line width=0.25mm,dashed] (B2)--(B1) node[midway,below]{};
			\draw[red,line width=0.25mm,dashed] (D2)--(D1) node[midway,below]{};
			\draw[red,line width=0.25mm,dashed] (E2)--(E1) node[midway,below]{};

			\draw[thick,->] (A)--(B);
			\draw[thick,->] (B)--(C);
			\draw[thick,->] (C)--(D);
			\draw[thick,->] (D)--(E);
			\draw[thick,->] (E)--(F);
			
			\node[inner sep=0.7pt,circle,draw=blue,fill=blue] (O1) at (0,-0.3){};
			\node[left] at (O1) {$u_0^{(j)}=0$};
			\node[inner sep=0.7pt,circle,draw=blue,fill=blue] (O2) at (1.0491,-0.3){};
			\node[right] at (O2) {$c_j=u_{5}^{(j)}$};
			\draw[blue,line width=0.25mm] (O1)--(O2) node[midway,below]{};
			\node[inner sep=0.7pt,circle,draw=red,fill=red] (B3) at (0.13,-0.3){};
			\node[below] at (B3) {$u_1^{(j)}$};
			\node[inner sep=0.7pt,circle,draw=red,fill=red] (C3) at (0.6,-0.3){};
			\node[below] at (C3) {$u_2^{(j)}$};
			\node[inner sep=0.7pt,circle,draw=red,fill=red] (D3) at (0.765,-0.3){};
			\node[below] at (D3) {$u_{3}^{(j)}$};
			\node[inner sep=0.7pt,circle,draw=red,fill=red] (E3) at (0.95,-0.3){};
			\node[below] at (E3) {$u_{4}^{(j)}$};

			\draw[red,->] (0.13,-0.1)--(0.13,-0.25);
			\draw[red,->] (0.6,-0.1)--(0.6,-0.25);
			\draw[red,->] (0.765,-0.1)--(0.765,-0.25);
			\draw[red,->] (0.95,-0.1)--(0.95,-0.25);

		\end{tikzpicture}
		\caption{The $j^{th}$ side-partition $(0=u_0^{(j)}<u_1^{(j)}<\ldots<u_{N_j}^{(j)}<u_{N_j+1}^{(j)}=c_j)$ of $c_j$, with $s_j=2,s_{{j+1}}=3$. An alternative way of building the $u^{(j)}[N_j,c_j]$ will be given in \Cref{fig:fusion}. Notice here that we see the contact point on $c_j$ but we don't mark it and treat it as the other points. }
		\label{fig7}
	\end{figure}
	
	\paragraph{The main strategy of the proof}is to consider for all $\ssk\in\Nkn$ the extraction mapping, which encodes a convex $\zzn$-gon by its $\ECP(\zzn)$, and its side-partitions:
	\begin{align}
		\app{\chi_{\ssk}}{\CVnij}{\Niceset}{z[n]}{\left(\ell[\kappa],u^{(1)}[N_1,c_1],\ldots,u^{(\kappa)}[N_\kappa,c_\kappa]\right)}.
	\end{align}
	where $\Niceset:=\im(\chi_{\ssk})$ is a strict subset of $(\RR^+)^{\kappa}\times\prod_{j=1}^\kappa (\RR^+)^{N_j}$ that we now discuss. Recall that we have set for all $j\in\entk,$ $N_j=s_j+s_{\wj}-1$.
	
	We need, in what follows, to see the map $\chi_{\ssk}$ as a "nice map" (a piecewise linear map, see \Cref{rmk:loclin}) with a "nice inverse" (with a computable Jacobian determinant) since we will use in the sequel this inverse to push-forward a measure of $\Niceset$ onto the Lebesgue measure on $\CVnij$.
	
	Since $\CVnij$ is a subset of $\RR^{2n}$ with non-empty interior, $\Niceset$ will be seen to be identifiable with a subset of a domain with the same dimension.
	In order to characterize $\Niceset$, it is relevant to notice that since the $\ssk$ are fixed, the $u^{(j)}[N_j,c_j]$ allow to reconstruct the vectors of the convex chains. Since these vectors have increasing slope (when turning around the $\zzn$-gon counter-clockwise), the $u^{(j)}[N_j,c_j]$ must satisfy a condition that we now detail.
	
	\paragraph{The image set of $\chi_{\ssk}$.}Set $\mathcal{L}_\kappa^*:=\{\Lj\in\mathcal{L}_\kappa \text{ s.t. for all }j\in\entk,c_j>0 \}.$ For any $\Lj\in\mathcal{L}_\kappa^*$, consider the side lengths $\Cj$ of the $\ECP$ induced by $\Lj.$ 
	For any $\kappa$-tuple of side-partitions $\left(u^{(1)}[N_1,c_1],\ldots,u^{(\kappa)}[N_\kappa,c_\kappa]\right)$ of $\Cj$, and all $j\in\entk,$ define the interpoints distances of the side-partition $u^{(j)}[N_j,c_j]$ by $\Delta u^{(j)}_i=u_i^{(j)}-u_{i-1}^{(j)}$ for all $i\in\{1,\ldots,N_j+1\}.$ Then, set the vectors
	\[v_k^{(j)}=\begin{pmatrix} 
		\Delta u^{{(j)}}_{s_{j}+k}\\ 
		\Delta u^{{(\wj)}}_k
	\end{pmatrix},\quad \forall k\in\{1,\ldots,s_{\wj}\}.\]
	In words, summing the vectors $v^{(j)}[s_{\wj}]$ allow one to join the point $(0,0)$ to $(c_j-u^{(j)}_{s_j},u^{(\wj)}_{s_{\wj}})$. When reordered by increasing slope, these vectors form the boundary of a convex polygon whose vertices form a convex chain.
	This condition on the vectors must be encoded in the side-partitions when decomposing a $\zzn$-gon through $\chi_{\ssk}$; this condition allows us to identify the image set $\Niceset.$
	
	This leads us to defining $\Sntk$, the following open subset of $\RR^{\kappa}\times\prod_{j=1}^\kappa \RR^{N_j}$ by setting
	\begin{multline*}
		\Sntk:=\bigg\{\left(\ell[\kappa],w^{(1)},\ldots,w^{(\kappa)}\right)\in\mathcal{L}_\kappa^*\times\prod_{j=1}^\kappa \RR^{N_j}
		\text{ where }w^{(j)}:=w^{(j)}[N_j,c_j]\in P[c_j,N_j],\\
		\text{ and }\underbrace{\frac{\Delta w^{(j)}_1}{\Delta w^{({j-1})}_{s_{{j-1}}+1}}< \ldots < \frac{\Delta w^{(j)}_{s_j}}{\Delta w^{({j-1})}_{s_{{j-1}}+s_j}}}_{\text{condition on the slopes order}} \text{ for all }j\in\entk\bigg\}.
	\end{multline*}
	Note that we set $\Lj$ in $\mathcal{L}_\kappa^*$ so to force the construction of any $\ECP$ possible (except those having a nonzero side) within $\Ck.$
	The increments are considered for the side-partitions as they compose the vectors in each corner in the way described by \Cref{fig:fusion} below. Recall the family of mappings $(\varphi_{j})_{j\in\entk}$ introduced in \eqref{eq:Aj} together with \Cref{fig:affine}:
	
	\begin{figure}[H]
		\begin{minipage}{0.41\textwidth}
		
			\begin{tikzpicture}[scale=2.8]
				
				\draw[blue,line width=0.25mm] (0,0)--(1.0491,0.0) node[midway,below]{$c_j$};
				\draw[blue,line width=0.25mm] (1.0491,0.0)--(1.7033,0.8202)node[midway,below,sloped]{$c_{\widebar{j+1}}$};
				\draw[blue,line width=0.25mm] (-0.6541,0.8202)--(0,0)node[midway,below,sloped]{$c_{\widebar{j-1}}$};
				
				\node[inner sep=0.7pt,circle,draw=blue,fill=blue] (R) at (1.0491,0.0){};
				\node[below] at (R) {$b_{j}$};
				\node[inner sep=0.7pt,circle,draw=blue,fill=blue] (T) at (0,0){};
				\node[below] at (T) {$b_{\widebar{j-1}}$};
				
				\node[inner sep=0.7pt,circle,draw=red,fill=red] (A) at (-0.4941,0.6202){};
				\node[inner sep=0.7pt,circle,draw=red,fill=red] (B) at (-0.041,0.2202){};
				\node[inner sep=0.7pt,circle,draw=red,fill=red] (C) at (0.6,0.){};
				\node[inner sep=0.7pt,circle,draw=red,fill=red] (D) at (0.8,0.05){};
				\node[inner sep=0.7pt,circle,draw=red,fill=red] (E) at (1.1,0.2){};
				\node[inner sep=0.7pt,circle,draw=red,fill=red] (F) at (1.45,0.5){};
				
				\draw[thick,->] (A)--(B);
				\draw[thick,->] (B)--(C);
				\draw[thick,->] (C)--(D);
				\draw[thick,->] (D)--(E);
				\draw[thick,->] (E)--(F);

			\end{tikzpicture}
		\end{minipage}
		\begin{minipage}{0.25\textwidth}
		\centering
				\begin{tikzpicture}[scale=1.2]
				\node[inner sep=0.7pt,circle,draw=red,fill=red] (Cbis) at (0.8,0.){};
				\node[inner sep=0.7pt,circle,draw=red,fill=red] (C) at (0.6,0.){};
				\draw[blue,line width=0.3mm] (0,1.0491)--(0.,0.) node[midway,below]{};
				\draw[blue,line width=0.3mm] (0.,0.0)--(C) node[midway,below,sloped]{};
				
				\draw[blue,line width=0.3mm] (Cbis)--(1.2491,0) node[midway,below,sloped]{};
				\draw[blue,line width=0.3mm] (1.2491,0)--(1.2491,1.0491)node[midway,below,sloped]{};
				
				\node[inner sep=0.7pt,circle,draw=blue,fill=blue] (R) at (1.2491,0.0){};
				\node[inner sep=0.7pt,circle,draw=blue,fill=blue] (T) at (0,0){};
				\node[left] at (T) {};
				\node[right] at (R) {};
				\node[inner sep=0.7pt,circle,draw=red,fill=red] (A) at (0.,0.7932){};
				\node[left] at (A) {};
				\node[inner sep=0.7pt,circle,draw=red,fill=red] (B) at (0.15,0.19){};
				\node[inner sep=0.7pt,circle,draw=red,fill=red] (D) at (0.95,0.08){};
				\node[inner sep=0.7pt,circle,draw=red,fill=red] (E) at (1.15,0.3){};
				\node[inner sep=0.7pt,circle,draw=red,fill=red] (F) at (1.2491,0.6408){};
				\node[right] at (F) {};
				\draw[thick,->] (A)--(B);
				\draw[thick,->] (B)--(C);
				\draw[thick,->] (Cbis)--(D);
				\draw[thick,->] (D)--(E);
				\draw[thick,->] (E)--(F);

		\draw[black,line width=0.5mm,->] (0.3,-0.5)--(1.,-0.5) node[midway,below]{\large $\varphi_{{j-1}} \text{ and } \varphi_j$};
		\end{tikzpicture}

		\end{minipage}
		\begin{minipage}{0.3\textwidth}
			\centering
			\begin{tikzpicture}[scale=3.1]
				
				\draw[blue,line width=0.3mm] (0,1.0491)--(0.,0.0) node[midway,below]{};
				\draw[blue,line width=0.3mm] (0.,0.0)--(1.0491,0.)node[midway,below,sloped]{};
				\draw[blue,line width=0.3mm] (1.0491,0.)--(1.0491,1.0491)node[midway,below,sloped]{};
				
				\node[inner sep=0.7pt,circle,draw=blue,fill=blue] (R) at (1.0491,0.0){};

				\node[inner sep=0.7pt,circle,draw=blue,fill=blue] (T) at (0,0){};
				\node[left] at (T) {\small $B'_{{j-1}}$};
				\node[right] at (R) {\small $B'_j$};
				\node[inner sep=0.7pt,circle,draw=red,fill=red] (A) at (0.,0.7932){};
				\node[left] at (A) {\small $A'_{{j-1}}$};
				\node[inner sep=0.7pt,circle,draw=red,fill=red] (B) at (0.15,0.19){};
				\node[inner sep=0.7pt,circle,draw=red,fill=red] (C) at (0.6,0.){};
				\node[inner sep=0.7pt,circle,draw=red,fill=red] (D) at (0.75,0.08){};
				\node[inner sep=0.7pt,circle,draw=red,fill=red] (E) at (0.95,0.3){};
				\node[inner sep=0.7pt,circle,draw=red,fill=red] (F) at (1.0491,0.6408){};
				\node[right] at (F) {\small $C'_j$};
				\draw[thick,->] (A)--(B);
				\draw[thick,->] (B)--(C);
				\draw[thick,->] (C)--(D);
				\draw[thick,->] (D)--(E);
				\draw[thick,->] (E)--(F);
				
				\draw[black,thick,-|] (-0.01,-0.1)--(0.,-0.1);
				\draw[black,thick,-|] (-0.01,-0.1)--(0.15,-0.1);
				\draw[black,thick,-|] (0.15,-0.1)--(0.6,-0.1);
				\draw[black,thick,-|] (0.6,-0.1)--(0.75,-0.1);
				\draw[black,thick,-|] (0.75,-0.1)--(0.95,-0.1);
				\draw[black,thick,-|] (0.95,-0.1)--(1.0591,-0.1);
				
				\draw[color=blue,decorate,decoration={brace,raise=-0.1cm}](0.6,-0.25)--(-0.01,-0.25)  node[below=0.01cm,pos=0.62] { $\pi_1(\varphi_{{j-1}}(v))$};
				\draw[color=red,decorate,decoration={brace,raise=-0.1cm}](1.0591,-0.25)--(0.6,-0.25)  node[below=0.01cm,pos=0.3] {$-\pi_2(\varphi_{j}(v))$};
				
			\end{tikzpicture}
		\end{minipage}
		\captionof{figure}{\label{fig:fusion}	The map $\varphi_j$ (resp. $\varphi_{{j-1}}$), as introduced in \Cref{fig:affine}, sends the triangle $\corner_j$ (resp. $\corner_{{j-1}}$) on the triangle $A'_jB'_jC'_j$ (resp. $A'_{{j-1}}B'_{{j-1}}C'_{{j-1}}$). If you perform one more rotation, equivalent to set $C'_{{j-1}}=A'_j$ and fix $B'_{{j-1}},A'_j,B'_j$ on the same line, we may interprete side-partitions just like it appears in the right-hand figure above.}
	\end{figure}
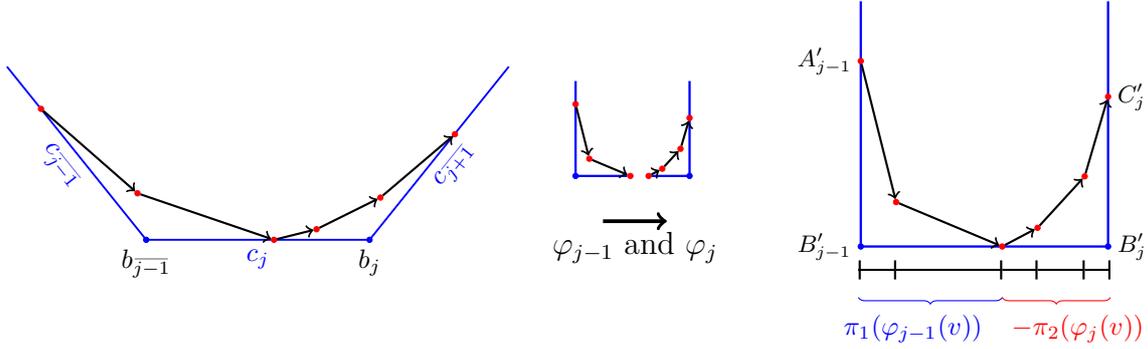
	\paragraph{A powerful diffeomorphism.}It is rather easy to see that, up to a Lebesgue null set \footnote{(we don't want to treat separately the cases in which several points $\zzn$ are parallel to the lines of $\Ck$, or more than two $z_i$ are aligned).}, $\chi_{\ssk}$ is a bijection between $\CVnij$ and $\Sntk.$ The following theorem details some even more important properties of the mapping $\chi_{\ssk}.$
	
	\begin{Theoreme}\label{thm:diffeo}
		For all $\ssk\in\Nkn$, the mapping
		\begin{align}
			\app{\diffeo}{\CVnij}{\Sntk}{z[n]}{\chi_{\ssk}(z[n])}
		\end{align} is a piecewise diffeomorphism (in the sense of \Cref{rmk:loclin}) whose Jacobian determinant is constant and equals $1/\sth^{n-\kappa}$ (hence, the Jacobian determinant does not depend on $\ssk$). \par
		In particular, the Lebesgue measure of the set of interest $\CVnij$ satisfies:
		\begin{align}\label{for:lebesgue}
			\Leb_{2n}\left(\CVnij\right) = \Leb_{2n}\left(\Sntk\right)\sth^{n-\kappa}.
		\end{align}
	\end{Theoreme}
	\begin{proof}[Proof of \Cref{thm:diffeo}]
		We need to detail how the inverse mapping of $\diffeo$ is defined to understand its (piecewise) linearity. Pick $\left(\ell[\kappa],u^{(1)}[N_1,c_1],\ldots,u^{(\kappa)}[N_\kappa,c_\kappa]\right)\in\Sntk$.
		\paragraph{Linearity.}Since the tuple $\Lj$ is in $\mathcal{L}_\kappa^*$, it defines an equiangular parallel polygon $\ECP$ inside $\Ck.$ The map which associates to the $\Lj$ the $\bb[\kappa]$ is piecewise linear: in the classical Cartesian coordinate system, for any $j\in\entk$, the coordinates of $\bb_{j}$ are linear in $\ell_j$ and $\ell_{j+1}$ since \[\bb_j=\left(r_{j-1}+\frac{\ell_j}{\tth}-\frac{\ell_{j+1}}{\sth},\ell_j\right),\] up to a rotation.
		
		Then the contact point $\cp_j$ is a translation of $\bb_{{j-1}}$ by $u_{s_j}^{(j)}$ along the $j^{th}$ side of the $\ECP.$ This means that the construction of the contact points are linear in the $\Lj$ and $u_{s_j}^{(j)},j\in\entk.$ To reconstruct the rest of the points, recall the vectors:
		\[v_k^{(j)}=\begin{pmatrix} 
			\Delta u^{{(j)}}_{s_{j}+k}\\ 
			\Delta u^{{(\wj)}}_k
		\end{pmatrix},\quad \forall k\in\{1,\ldots,s_{\wj}-1\}.
		\]
		The convexity condition imposed on the slopes in $\Sntk$ forces these vectors to appear by increasing slope, so that the map $\varphi_j$ sends these vectors in $\corner_j$ to form the boundary of a convex polygon, whose tuple of vertices is thus a convex chain. The construction of the points of this convex chain hence rewrites as
		\[z^{(j)}_2=\cp_j+A_j(\theta_\kappa)^{-1} v^{(j)}_1,\]
		where $A_j$ was introduced in \eqref{eq:Aj}, and inductively for all $k\in\{2,\ldots, s_{\wj}-1\},$
		\[z^{(j)}_{k+1}=z^{(j)}_{k}+A_j(\theta_\kappa)^{-1} v^{(j)}_k.\]
		Notice that we built only $s_{\wj}-1$ vectors, since the $s_{\wj}$$^{th}$ connects the last point $z^{(j)}_{s_j}$ to $\cp_{\wj}$ and is thus determined.
		
		{\centering
			\begin{tikzpicture}[scale=0.8][hbtp]
			
				\draw[blue] (0.8622,0.2153)--(1.93,1.51) node[pos=0.5,below,sloped]{\tiny $j^{th}$ side};
				\draw[blue] (1.93,1.51)--(1.64,2.8) node[pos=1.10,above right,sloped]{\tiny $\widebar{j+1}^{th}$ side};

				\node[inner sep=1pt,circle,draw=black,fill=black] (L1) at (1.0149,0.4){};
				\node[left,black] at (L1) {\tiny $\cp_{j}$};
				\node[inner sep=1pt,circle,draw=black,fill=black] (L) at (1.93,1.51){};
				%\node[below,black] at (L) {\tiny $\bb_j$};
				%				\node[inner sep=1pt,circle,draw=black,fill=black,right] (L) at (1.78,2.21){};
				%				\node[left,black] at (L) {\tiny $\cp_j$};
				
				\draw[black,|-|] (1.93,1.51)--(1.8743,1.7577) node[pos=0.7,right]{};
				\draw[black,-|] (1.8743,1.7577)--(1.804,2.07) node[pos=0.7,right]{};
				\draw[black,-|] (1.804,2.07)--(1.6835,2.61) node[pos=0.7,right]{};
				
				\draw[black,|-|] (1.0149,0.4) --(1.348,0.8) node[pos=0.7,right]{};
				\draw[black,-|] (1.348,0.8)--(1.635,1.15) node[pos=0.7,right]{};
				\draw[black,-|] (1.635,1.15) --(1.7593,1.3) node[pos=0.7,right]{};
				
				\draw[blue,->] (1.0149,0.4) --(1.2897,1.04) node[pos=0.7,right]{};
				\draw[blue,->] (1.2897,1.04) -- (1.5004,1.7) node[pos=0.7,right]{};
				\draw[blue,->] (1.5004,1.7) -- (1.5128,2.4) node[pos=0.7,right]{};
				
				\node[inner sep=0.7pt,circle,draw=red,fill=red] (A) at (1.2897,1.04){};	
				\node[left,red] at (A) {\tiny $z_2^{(j)}$};			
				\node[inner sep=0.7pt,circle,draw=red,fill=red] (B) at (1.5004,1.7){};	
				\node[left,red] at (B) {\tiny $z_3^{(j)}$};			
				\node[inner sep=0.7pt,circle,draw=red,fill=red] (C) at (1.5128,2.4){};
				\node[left,red] at (C) {\tiny $z_4^{(j)}$};
			\end{tikzpicture}
			
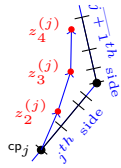
\captionof{figure}{\small Vector building}\label{fig5}}
		
		We obtain $n$ points $(z_1,\ldots,z_n)=(\underbrace{\cp_1,z^{(1)}_2,\ldots,z^{(1)}_{s_1}}_{s_1 \text{ points}},\underbrace{\cp_2,z^{(2)}_2,\ldots,z^{(2)}_{s_2}}_{s_2 \text{ points}},\ldots,\underbrace{\cp_\kappa,z^{(\kappa)}_2,\ldots,z^{(\kappa)}_{s_\kappa}}_{s_\kappa \text{ points}}).$
		In the end, the whole construction only includes maps that are piecewise linear and piecewise differentiable (\Cref{rmk:loclin}) in the data $\left(\ell[\kappa],u^{(1)}[N_1,c_1],\ldots,u^{(\kappa)}[N_\kappa,c_\kappa]\right)$, and thus so is $\diffeo.$
		
		\paragraph{Jacobian.} Let us compute the Jacobian determinant of the inverse mapping $(\diffeo)^{-1}$. This requires first the Jacobian determinant of the construction of the contact points $\cp[\kappa]$. To build a contact point, we build the vertices $\bb[\kappa]$: we fix the $y$-coordinate of $\bb_\kappa$ and $\bb_1$ as $\ell_1$. Now, rotate the figure of $\pi/2-\thk$: in this new system of coordinates, the $y$-coordinate of $\bb_1$ and $\bb_2$ is $\ell_2$. This determines the coordinates of $\bb_1$, and from one rotation to the other, this of $\bb_j$ for all $j\in\entk$. The Jacobian determinant of the whole construction of the $\bb[\kappa]$ is the determinant of a product of rotation matrices, which is thus 1.
		
		Then, as said before, the contact point $\cp_j$ is built as a translation of $u_{s_j}^{(j)}$ from $\bb_{{j-1}}$ on the $j^{th}$ side of $\ECP.$ This operation has Jacobian determinant 1 as well.
		
		For $j\in\entk$, the building of $z_k^{(j)},k\in\{2,\ldots,s_j\}$ is a translation from $z^{(j)}_{k-1}$ with the product of the matrix $A_j(\theta_\kappa)^{-1}$ with the vector $v^{(j)}_{k-1}$, for all $j\in\entk.$ So we have 
		\begin{align}\label{eq:jac}\nonumber
			\Jac\left((\diffeo)^{-1}\right)&=\left|\prod_{j=1}^{\kappa}\det\left(A_j(\theta_\kappa)^{-1}\right)^{s_j-1}\right|\\
			&=\sth^{n-\kappa}.
		\end{align}
	\end{proof}

	\subsection{Working at $\Lj$ fixed.}
	We performed a first "conditioning" based on the size-vector $\ssk$ of the vectors forming the boundary of any $\zzn$-gon. From this point, the map $\diffeo$ encodes $\zzn$ in two parts: the "coordinates" $\Lj$ of the $\ECP(\zzn)$ (in the sense that their data is equivalent) and the side-partitions $\left(u^{(1)}[N_1,c_1],\ldots,u^{(\kappa)}[N_\kappa,c_\kappa]\right)$. We may now perform a second conditioning on the coordinates $\Lj,$ by introducing the set 
	\begin{align}
		\Sntksj=\bigg\{\left(w^{(1)},\ldots,w^{(\kappa)}\right) \text{ such that }\left(\Lj,w^{(1)},\ldots,w^{(\kappa)}\right)\in\Sntk\bigg\}.
	\end{align}
	This conditioning actually reveals the mass of $\zzn$-gons contained in an $\ECP$ of coordinates $\Lj$ with a repartition $\ssk$. Indeed, we have the following Lemma:
	\begin{Lemma}\label{lem:suite}For all $\Lj\in\mathcal{L}_\kappa,\ssk\in\Nkn$:
		\begin{align}\label{eq:super1}
			\Leb_{2n}(\CVnij)&=\sth^{n-\kappa}\int_{\RR^\kappa}\mathbb{1}_{\Lj\in\mathcal{L}_\kappa}\Leb_{2n-\kappa}\left(\Sntksj\right)\mathrm{d}\Lj.
		\end{align}
	\end{Lemma}
	\begin{proof}
		We have
		\begin{align}
			\Leb_{2n}(\Snk)&=\int_{\RR^\kappa}\mathbb{1}_{\Lj\in\mathcal{L}_\kappa}\underbrace{\int_{\RR^{2n-\kappa}} \mathbb{1}_{\left(u^{(1)}[N_1,c_1],\ldots,u^{(\kappa)}[N_\kappa,c_\kappa]\right)\in\Sntksj}  \mathrm{d}u[2n-\kappa]}_{\Leb_{2n-\kappa}\left(\Sntksj\right)}\mathrm{d}\Lj,
		\end{align}
		where we denoted $\mathrm{d}u[2n-\kappa]=\prod_{j=1}^\kappa \mathrm{d}u^{(j)}[N_j,c_j]$ not to burden our equations. Hence, (\ref{for:lebesgue}) allows to conclude.
	\end{proof}
	
	This Lemma encodes a $n$-tuple $\zzn$ in convex position by a new geometrical description embodied in the coordinates $\left(\ell[\kappa],u^{(1)}[N_1,c_1],\ldots,u^{(\kappa)}[N_\kappa,c_\kappa]\right)$. This change of variables comes at the price of the Jacobian computed in \Cref{thm:diffeo}. The next step, as suggested by \Cref{lem:suite}, is to compute at $(\Lj,\ssk)$ fixed the Lebesgue measure of the set $\Sntksj.$
	
	\paragraph{The Lebesgue measure of $\Sntksj$.}Pick $\Lj\in\mathcal{L}_\kappa$, and $\left(u^{(1)},\ldots,u^{(\kappa)}\right)\in\Sntksj$. This tuple of side-partitions
	$\left(u^{(1)},\ldots,u^{(\kappa)}\right)$ can be seen as an element of the set $\prod_{j=1}^\kappa P[c_j,N_j]$. Indeed, a side-partition $u^{(j)}:=u^{(j)}[N_j,c_j]$ marks $N_j$ points on the segment $[0,c_j]$. Nonetheless, just like we did after \eqref{eq:P} and \eqref{eq:I}, we may rather consider the tuples of distances between points, and reorder each $u^{(j)}$ into increasing increments so to form $(\Delta\widetilde{u}^{(1)}[N_1+1],\ldots,\Delta\widetilde{u}^{(\kappa)}[N_\kappa+1])$, which is thus an element of $\prod_{j=1}^\kappa I[c_j,N_j+1]$. Considering the elements of $\prod_{j=1}^\kappa I[c_j,N_j+1]$ rather than those of $\prod_{j=1}^\kappa P[c_j,N_j]$ prevents us from forming twice the same convex chain. We define then
	\begin{align*}
		\app{\Order}{\Sntksj}{\prod_{j=1}^\kappa I[c_j,N_j+1]}{\left(u^{(1)},\ldots,u^{(\kappa)}\right)}{(\Delta\widetilde{u}^{(1)}[N_1+1],\ldots,\Delta\widetilde{u}^{(\kappa)}[N_\kappa+1])},
	\end{align*}
	a piecewise linear mapping. Given $(\Delta\widetilde{u}^{(1)}[N_1+1],\ldots,\Delta\widetilde{u}^{(\kappa)}[N_\kappa+1])\in\prod_{j=1}^\kappa I[c_j,N_j+1]$, how many distinct $n$-tuples $\left(u^{(1)},\ldots,u^{(\kappa)}\right)\in\Sntksj$ can we build out of this object? 
	We solve this matter in the following lemma :
	\begin{Lemma}\label{lem:cnt}
		Let $\ssk\in\Nkn$, $\Lj\in\mathcal{L}_\kappa$ and the corresponding $\Cj$, and consider a tuple $(\Delta\widetilde{u}^{(1)}[N_1+1],\ldots,\Delta\widetilde{u}^{(\kappa)}[N_\kappa+1])\in\prod_{j=1}^\kappa I[c_j,N_j+1].$ Then,
		\begin{align}
			\#\Order^{-1}\left(\Delta\widetilde{u}^{(1)}[N_1+1],\ldots,\Delta\widetilde{u}^{(\kappa)}[N_\kappa+1]\right)=\prod_{j=1}^{\kappa}{s_j+s_{\wj} \choose s_j}s_j!
		\end{align}
	\end{Lemma}
	
	\begin{proof}
		We need to build $\kappa$ sets of vectors, the $j^{th}$ being devoted to the construction of the convex chain in the $j^{th}$ corner of the $\ECP.$ To form the $s_j$ vectors in the $j^{th}$ corner	we select $s_j$ pieces among $\Delta\widetilde{u}^{(j)}[N_j+1]$ that will account for $x$-contributions of vectors and we select $s_j$ pieces (or complementarily $s_{\wj}$ pieces) in $\Delta\widetilde{u}^{(\wj)}[N_{\wj}+1]$ that will account for $y$-contributions. There are $\prod_{j=1}^{\kappa}{s_j+s_{\wj} \choose s_j}$ ways of choosing these pieces, and  $\prod_{j=1}^{\kappa}s_j!$ ways to pair these elements to form the $s_j$ vectors in each corner (see \Cref{fig:pair} for an example of construction).
		
		There exists a unique order that sorts these vectors in convex order in each corner, so that, put together, these pieces form a convex polygon whose set of vertices is a "distinct" $n$-tuple $\zzn\in\CVnij$ with $\Lj(\zzn)=\Lj$. Now, consider $\diffeo(\zzn)=\left(\Lj,u^{(1)}[N_1,c_1],\ldots,u^{(\kappa)}[N_\kappa,c_\kappa]\right)$: the last entries  $\left(u^{(1)},\ldots,u^{(\kappa)}\right):=\left(u^{(1)}[N_1,c_1],\ldots,u^{(\kappa)}[N_\kappa,c_\kappa]\right)$ of this tuple form a new distinct element (since $\zzn$ is one as well) of $\Sntksj.$
	\end{proof}      
			\begin{figure}[H]
			\begin{minipage}{0.45\textwidth}
				\centering
				\begin{tikzpicture}[scale=2.5]
				
				\node[inner sep=1pt,circle,draw=blue,fill=blue] (R1) at (1.1491,0.0){};
				\node[below,sloped] at (R1) {$0$};
				\node[inner sep=1pt,circle,draw=blue,fill=blue] (R2) at (1.8033,0.8202){};
				\node[right,sloped] at (R2) {$c_{j+1}$};
				\draw[blue] (R1)--(R2)node[midway,below]{};
				\node[inner sep=1pt,circle,draw=red,fill=red] (B4) at (1.205,0.07){};
				
				\node[inner sep=1pt,circle,draw=red,fill=red] (C4) at (1.295,0.186){};
				\node[below] at (C4) {};
				\node[inner sep=1pt,circle,draw=red,fill=red] (D4) at (1.415,0.331){};
				
				%\node[inner sep=1pt,circle,draw=red,fill=red] (E4) at (1.505,0.442){};
				%\node[below] at (E4) {};
				\node[inner sep=1pt,circle,draw=red,fill=red] (F4) at (1.595,0.561){};
				
				\draw[yellow,line width=0.25mm] (C4)--(B4)node[midway,below]{};
				\draw[violet,line width=0.25mm] (D4)--(F4)node[midway,below]{};
				\draw[olive,line width=0.25mm] (F4)--(R2)node[midway,below]{};
				
				\node[inner sep=1pt,circle,draw=blue,fill=blue] (O1) at (0,-0.1){};
				\node[below] at (O1) {$0$};
				\node[inner sep=1pt,circle,draw=blue,fill=blue] (O2) at (1.0491,-0.1){};
				\node[below,sloped] at (O2) {$c_j$};
				\draw[blue] (O1)--(O2) node[midway,below]{};
				\node[inner sep=1pt,circle,draw=red,fill=red] (B3) at (0.07,-0.1){};
				%					\node[below] at (B3) {$u_1^{(j)}$};
				\node[inner sep=1pt,circle,draw=red,fill=red] (C3) at (0.18,-0.1){};
				%					\node[below] at (C3) {$\cdots$};
				\node[inner sep=1pt,circle,draw=red,fill=red] (D3) at (0.34,-0.1){};
				\node[below] at (D3) {};
				\node[inner sep=1pt,circle,draw=red,fill=red] (E3) at (0.6,-0.1){};
				%					\node[below] at (E3) {$u_{4}^{(j)}$};
				\draw[violet,line width=0.4mm] (O1)--(B3)node[midway,below]{};
				\draw[olive,line width=0.4mm] (D3)--(C3)node[midway,below]{};
				\draw[yellow,line width=0.4mm] (D3)--(E3)node[midway,below]{};
				
				\node[inner sep=1pt] (A) at (0.035,-0.095){};
				\node[inner sep=1pt] (B) at (1.494,0.4315 ){};
				\draw[<->,>=latex,thick,violet] (A)to [bend left](B);
				
				\node[inner sep=1pt] (C) at (1.70665,0.7356){};
				\node[inner sep=1pt] (D) at (0.25,-0.095){};
				\draw[<->,>=latex,thick,olive] (D)to [bend left](C);
				
				\node[inner sep=1pt] (E) at (0.47,-0.095){};
				\node[inner sep=1pt] (F) at (1.25,0.15 ){};
				\draw[<->,>=latex,thick,yellow] (E)to [bend left](F);
				\end{tikzpicture}
			\end{minipage}
			$\underset{\text{reordering}}{\longrightarrow}$
			\begin{minipage}{0.45\textwidth}
				\centering
				\begin{tikzpicture}[scale=2.5]

				\node[inner sep=1pt,circle,draw=blue,fill=blue] (O1) at (0,-0.1){};
				\node[below] at (O1) {$\bb_{j-1}$};
				\node[inner sep=1pt,circle,draw=blue,fill=blue] (O2) at (1.0491,-0.1){};
				\node[below,sloped] at (O2) { $\bb_j$};
				\node[inner sep=1pt,circle,draw=blue,fill=blue] (O3) at (1.7033,0.7202){};
				\node[below,right,sloped] at (O3) { $\bb_{j+1}$};
				
				\draw[blue] (O1)--(O2)--(O3);
				
				\node[inner sep=1pt,circle,draw=blue,fill=blue] (R1) at (0.0,-0.1){};
				\node[inner sep=1pt,circle,draw=red,fill=red] (R1bis) at (0.54,-0.1){};
				\node[below] at (R1bis) { $\cp_{j}$};
				\node[inner sep=1pt,circle,draw=red,fill=red] (R2) at (0.8583,-0.0308){};
				\node[inner sep=1pt,circle,draw=red,fill=red] (R3) at (1.1463,0.1052){};
				\node[inner sep=1pt,circle,draw=red,fill=red] (R4) at (1.40,0.338){};
				\node[right] at (R4) { $\cp_{j+1}$};
				\draw[->,line width=0.4mm,yellow] (R1bis)to(R2);
				\draw[->,line width=0.4mm,olive] (R2)to(R3);
				\draw[->,line width=0.4mm,violet] (R3)to(R4);

				\end{tikzpicture}
			\end{minipage}
			\captionof{figure}{\label{fig:pair} In the first drawing, given a partition in $I[c_j,s_{{j-1}}+s_j]$ and a partition in $I[c_{\wj},s_j+s_{\wj}]$, we randomly pair $s_j$ pieces of $c_j$ with $s_j$ pieces of $c_{{j+1}}$ to form the vectors in the $j^{th}$ corner. Note that an affine transformation is hiding in the construction of these vectors. In the second drawing, vectors were reordered by increasing slope. The points $\cp_j,\cp_{{j+1}}$ naturally appear as the edges of the convex chain formed by those vectors. In these particular drawings, we took $s_j=3,s_{{j-1}}=2,s_{{j+1}}=2$.}
		\end{figure}
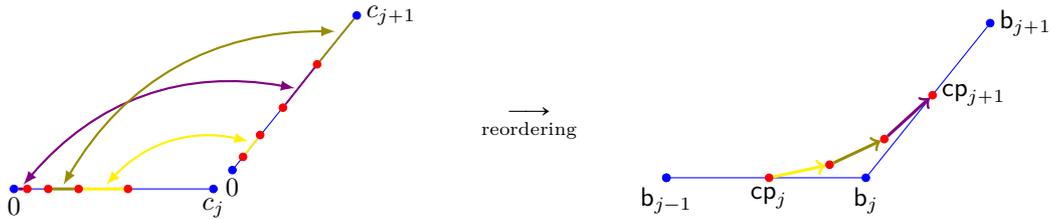
	
	This allows us to compute the Lebesgue measure of $\Sntksj.$ Indeed, the map $\Order$ carries the Lebesgue measure of $\Sntksj$ onto that of $\prod_{j=1}^\kappa I[c_j,N_j+1]$. 
	\begin{Corollaire}\label{cor:Lebes}
		For $\ssk\in\Nkn$, and $\Lj\in\mathcal{L}_\kappa$ fixed, we have 
		\begin{align}\label{equ:lebes}
			\Leb_{2n-\kappa}\left(\Sntksj\right)=\prod_{j=1}^{\kappa}\frac{c_j^{s_j+s_{\wj}-1}}{s_j!(s_j+s_{\wj}-1)!}.
		\end{align}
	\end{Corollaire}
	\begin{proof}
		Indeed, by the previous Lemmas, we obtain
		\begin{align}
			\Leb_{2n-\kappa}\left(\Sntksj\right)&=\prod_{j=1}^{\kappa}{s_j+s_{\wj} \choose s_j}s_j! \cdot\Leb_{N_j}\left(I[c_j,N_j+1]\right),
		\end{align}
		and we conclude by \eqref{equ:simplex}.
	\end{proof}
	
	\subsection{The joint distribution of the pair $(\lk,\sk)$}
	
	\Cref{thm:diffeo} concretizes our understanding of this new equivalent geometric description of the set $\CVnij$ in terms of the $\ECP$. Let $\zn$ has distribution $\Dtn$, and set $\lk=\Lj(\zn),\sk=\ssk(\zn).$ By computing the Lebesgue measure of the set $\Sntksj,$ we managed to understand the weight of all $\zzn$-gons contained in any $(\Lj,\ssk)$-fibration, which is the key to the computation of the joint distribution of the pair $(\lk,\sk)$.
	
	\begin{Theoreme}\label{thm:distri}
		Let $\zn$ has distribution $\Dtn$, and consider the random variables $\lk=\Lj(\zn)$, $\sk=\ssk(\zn).$ Then for a given $\ssk\in\NN^\kappa$, the pair $(\lk,\sk)$ has the joint distribution 
		\begin{align}\label{equmonstre}
			\mathbb{P}\left(\lk\in\mathrm{d}\Lj,\sk=\ssk\right)=f^{(n)}_{\kappa}\left(\ell[\kappa],\ssk\right)\mathrm{d}\ell[\kappa],
		\end{align}where  
		\begin{align}
			f^{(n)}_{\kappa}\left(\ell[\kappa],\ssk\right)=\frac{n!\sth^{n-\kappa}}{\ptk} \mathbb{1}_{\ssk\in\NN_\kappa(n)} \mathbb{1}_{\ell[\kappa]\in\mathcal{L}_\kappa}\prod_{j=1}^{\kappa}\frac{c_j^{s_{{j-1}}+s_j-1}}{s_j!(s_{{j-1}}+s_j-1)!}.
		\end{align}
	\end{Theoreme}
	
	\begin{proof}
		Write
		\begin{align}
			\mathrm{d}\Dtn(\zzn)&=\frac{n!}{\ptk}\mathbb{1}_{z[n]\in\CVn}\mathbb{1}_{\ssk(z[n])\in\Nkn}\mathrm{d}\zzn\\
			&=\frac{n!}{\ptk}\sum_{\ssk\in\Nkn}\mathbb{1}_{\zzn\in\CVnij}\mathrm{d}\zzn.
		\end{align}
		For any continuous bounded test function $\eta:\RR^{\kappa}\times\NN^{\kappa}\to\RR$, we have
		\begin{align}\nonumber
			\mathbb{E}\bigg[\eta\left(\lk,\sk\right)\bigg]
			=\int_{(\RR^2)^n}\eta(\left(\Lj(\zzn),\ssk(\zzn)\right)\mathrm{d}\Dtn(\zzn)
		\end{align}
		which, after the change of variables $\diffeo(\zzn)=\left(\Lj,u^{(1)}[N_1,c_1],\ldots,u^{(\kappa)}[N_\kappa,c_\kappa]\right)$, done at $\Lj,\ssk$ fixed, gives
		\begin{multline}\label{eq:pourrave}
			\mathbb{E}\bigg[\eta\left(\lk,\sk\right)\bigg]=\frac{n!}{\ptk}\sum_{\ssk\in\Nkn}\Jac\left((\chi_{\ssk})^{-1}\right)\int_{\RR^{\kappa}}\eta(\Lj,\ssk)\\\left[\int_{\RR^{2n-\kappa}}\mathbb{1}_{\left(\Lj,u^{(1)}[N_1,c_1],\ldots,u^{(\kappa)}[N_\kappa,c_\kappa]\right)\in\Sntk}\mathrm{d}u[2n-\kappa]\right]\mathrm{d}\Lj
		\end{multline}
		Now, $\Jac\left((\chi_{\ssk})^{-1}\right)=\sth^{n-\kappa}$ for all $\ssk\in\Nkn$ by \eqref{eq:jac}, and the last bracket in \eqref{eq:pourrave} is nothing but the Lebesgue measure of $\Sntksj$ that we computed in \Cref{cor:Lebes} !
		Hence, inserting \eqref{equ:lebes} in \eqref{eq:pourrave} gives \Cref{thm:distri}.
	\end{proof}
	In the next section, we are going to exploit the asymptotical stochastic behavior of the pair $(\lk,\sk)$ to deduce an equivalent of $\ptk.$ However, in the particular cases $\kappa\in\{3,4\}$, we have $\Dtn=\Qn{\kappa}$, and the set $\mathcal{L}_\kappa$ is easily computable ! Hence we may immediately compute the exact value of $\pk$ out of $\Qn{\kappa}$. We propose in \Cref{sec:valtr} to have a look at these computations to recover Valtr' famous results in the triangle and the parallelogram.

	\section{An asymptotic result for convex regular polygons}
	\label{sec4}
	
	Let $\zn$ has distribution $\Dtn$ and consider $\lk=\Lj(\zn),\sk=\ssk(\zn)$.
	By symmetry, we have $\snj\overset{(d)}{=}\mathbf{s}^{(n)}_{1}$ for all $j\in\entk$, and since $\sum_{j=1}^{\kappa}\snj=n,$ the expectation of $\snj$ is given by $\mathbb{E}[\snj]=n/\kappa.$ In the sequel we will set $\mathbf{s}^{(n)}_{\kappa}=n-\sum_{j=1}^{\kappa-1}\snj$, and will describe $\mathbf{s}^{(n)}[\kappa-1]$ since the last value is determined by the other ones. What we are interested in from here, are the fluctuations of $\mathbf{s}^{(n)}[\kappa-1]$ around its expectation, and the asymptotic behaviour of the variables $\lk$ as $n$ grows. This is all contained in the following theorem :
	\begin{Theoreme}\label{thmonstre}
		Let $\zn$ has distribution $\Dtn$, and consider $\lk=\Lj(\zn),\sk=\ssk(\zn)$. 
		We introduce the random variables $\widebar{\Bell}^{(n)}[\kappa]=n\lk$ and $\mathbf{x}_j^{(n)}=\frac{\snj-n/\kappa}{\sqrt{n/\kappa}}$, for all $j\in\entk$.
		The following convergence in distribution holds in $\RR^{2\kappa-1}$:
		\[\left(\widebar{\Bell}_1^{(n)},\ldots,\widebar{\Bell}_\kappa^{(n)},\mathbf{x}_1^{(n)},\ldots,\mathbf{x}_{\kappa-1}^{(n)}\right)\dd\left(\widebar{\Bell}_1,\ldots,\widebar{\Bell}_\kappa,\mathbf{x}_1,\ldots,\mathbf{x}_{\kappa-1}\right),\]
		where the variables $\widebar{\Bell}[\kappa]$ are independent from the $\mathbf{x}[\kappa-1]$, the $\widebar{\Bell}[\kappa]$ are $\kappa$ random variables exponentially distributed with rate $\displaystyle\frac{2m_\kappa}{\kappa r_\kappa}$, where $\displaystyle m_\kappa=\frac{1+\cth}{\sth}$, and $\mathbf{x}=\mathbf{x}[\kappa-1]$ is a centered Gaussian random vector whose inverse covariance matrix $\Sigma_\kappa^{-1}$ of size $(\kappa-1)\times(\kappa-1)$ is given by :
		
		\[\Sigma_3^{-1}=\frac{1}{2}\begin{pmatrix}
			6&3\\
			3&6
		\end{pmatrix},\quad\Sigma_4^{-1}=\frac{1}{2}\begin{pmatrix}
			6&4&2\\
			4&8&4\\
			2&4&6
		\end{pmatrix},\quad\Sigma_5^{-1}=\frac{1}{2}\begin{pNiceMatrix}[nullify-dots,xdots/line-style=loosely dotted]
		6&4&3&2\\
		4&8&5&3\\
		3&5&8&4\\
		2&3&4&6\\
		\end{pNiceMatrix}\]
		and more generally
		\[\Sigma_\kappa^{-1}=\frac{1}{2}\begin{pNiceMatrix}[nullify-dots,xdots/line-style=loosely dotted]
			6&4&3&\Cdots&\Cdots&3&2\\
			4&8&5&4&\Cdots&4&3\\
			3&5&\Ddots&\Ddots&\Ddots&\Vdots&\Vdots\\
			\Vdots&4&\Ddots&\Ddots&\Ddots&4&\Vdots\\
			\Vdots&\Vdots&\Ddots&\Ddots&\Ddots&5&3\\
			3&4&\Cdots&4&5&8&4\\
			2&3&\Cdots&\Cdots&3&4&6\\
		\end{pNiceMatrix}\text{ for }\kappa \geq 6,\]
		whose determinant $\mathbb{d}_\kappa=\det\left(\Sigma_\kappa^{-1}\right)$ has already been mentioned in \Cref{thm-1}. The computation of this determinant, \ie
		\[\mathbb{d}_\kappa=\frac{\kappa}{3\cdot2^\kappa}\left(2(-1)^{\kappa-1}+(2-\sqrt{3})^{\kappa}+(2+\sqrt{3})^{\kappa}\right),\]
		announced in \eqref{eq:det}, is done in Appendix \ref{sec:det}.
	\end{Theoreme}
	We first state two important intermediary lemmas which will allow us to prove \Cref{thmonstre}.
	\begin{Lemma}[{\bf Local limit theorem for Poisson variables}]\label{cor1}
		Let $\kappa$ be any integer $\geq1$ and $Y_n$ a Poisson variable of mean $\frac{n}{\kappa}.$ We have
		\[\sup_{y}\left\vert \sqrt{\frac{n}{\kappa}}\mathbb{P}\left(Y_n=\left\lfloor\frac{n}{\kappa}+y\sqrt{\frac{n}{\kappa}}\right\rfloor\right)-\frac{e^{-y^2/2}}{\sqrt{2\pi}}\right\vert\cvg 0.\]
	\end{Lemma}
	
	\begin{proof}
		Pick $n$ i.i.d. random Poisson variables $X_1,\ldots,X_n$ of mean $1/\kappa$ and apply the local limit theorem \cite[Theorem VII.1.1]{petrov1975sums} to $\widetilde{X}_i=\sqrt{\kappa}(X_i-\frac{1}{\kappa}).$ The support of $\widetilde{X}_1$ is included in $\sqrt{\kappa}\mathbb{Z}-1/\sqrt{\kappa}$ and $X_1+\cdots+X_n$ is a Poisson variable of mean $n/\kappa.$
	\end{proof}
	\begin{Lemma}\label{lemdebase}
		Let $(g_n)_{n\in\NN}$ be a sequence of nonnegative measurable functions on $\RR^d$. Assume that for all $\varepsilon>0$, there exists a compact set $K_\varepsilon$ such that for all $n$ large enough, $\int_{K_\varepsilon^c}g_n<\varepsilon$ (where $K_\varepsilon^c$ is the complement of $K_\varepsilon$ in $\RR^d$), and that $g_n$ uniformly converges on all compact sets of $\RR^d$ towards a density $g$ (with respect to the Lebesgue measure on $\RR^d$). Then, there exists a sequence $(\alpha_n)_{n\in\NN}$ such that for $n$ large enough (for small values of $n$, $g_n$ could be zero), $\frac{1}{\alpha_n}g_n$ is a density and $\alpha_n\underset{n\to+\infty}{\longrightarrow}1.$
	\end{Lemma}
	\begin{proof}
		Take $\varepsilon>0$, and $K$ such that for $n$ large enough, $\int_{K^c}g_n<\varepsilon$. Since $g$ is a density, there exists a compact set $H$ such that $\int_{H}g\geq 1-\varepsilon.$ Let $S=K\cup H$. By the uniform convergence, there exists $m\in\NN$ such that for all $n\geq m$, we have $\int_S \vert g_n-g\vert\leq \varepsilon.$
		Then, the triangular inequality gives
		\begin{align}\label{monlabel1}
			\int_S g_n \geq \int_S g -\int_S\vert g_n-g\vert\geq 1-2\varepsilon,\quad\text{ and }\quad\int_{\RR^d} g_n \leq \int_S g +\int_S\vert g_n-g\vert+\int_{S^c}g_n\leq 1+2\varepsilon.
		\end{align}
		This shows $\alpha_n=\int_{\RR^d}g_n$ is finite, well-defined and nonzero for $n$ large enough, and $\frac{1}{\alpha_n}g_n$ is a density on $\RR^d$. From \eqref{monlabel1}, we have
		$\alpha_n\underset{n\to+\infty}{\longrightarrow}1$ and this concludes the proof.
	\end{proof}
	
	\paragraph{Notation.} Recall that by \Cref{prop1}, if $z[n]\in\CVn$, we can express the side lengths $c[\kappa]$ as a function of the boundary distances $\ell[\kappa]$ of the $\ECP(\zzn)$:
	\[c_j=r_\kappa-\cl_j(\ell[\kappa]),\text{ for all }j\in\entk,\]with $\cl_j(\ell[\kappa])=(\ell_{{j-1}}+\ell_{\wj}+2\ell_j\cth)/\sth.$
	
	%	We may then define $\iota e$ (standing for internal-external) the mapping  :
	%	\begin{align}
		%		\app{\iota e}{\mathcal{L}_\kappa}{(\RR_+)^\kappa}{\Lj}{\left(r_\kappa-\cl_j(\ell[\kappa])\right)_{j\entk}}.
		%	\end{align}

	\begin{proof}[Proof of \Cref{thmonstre}]
		
		Note that the joint density of $\left(\left(\widebar{\Bell}_1,\ldots,\widebar{\Bell}_\kappa\right),\left(\mathbf{x}_1,\ldots,\mathbf{x}_{\kappa-1}\right)\right)$ on $\RR_{+}^{\kappa}\times\RR^{\kappa-1}$ is given by
		\[g_{\kappa}(\widebar{\ell}[\kappa],x[\kappa-1])=\left(\sqrt{\frac{\mathbb{d}_\kappa}{(2\pi)^{\kappa-1}}}\exp\left(-\frac{1}{2}x^{t}\Sigma_\kappa^{-1}x\right)\right) \left(\frac{2m_\kappa}{\kappa r_\kappa}\right)^\kappa\exp\left(-\frac{2m_\kappa}{\kappa r_\kappa}\sum_{j=1}^{\kappa}\widebar{\ell}_j\right).\]
		
		This proof of \Cref{thmonstre} is carried out in two steps : 
		\begin{enumerate}
			\item We show the uniform convergence on compact sets of the "density" of the pair $\left(\widebar{\Bell}^{(n)}[\kappa],\mathbf{x}^{(n)}[\kappa-1]\right)$. More exactly, we show the uniform convergence on compact sets of a density $g^{(n)}_{\kappa}$ introduced in (\ref{gn}) associated to these random variables.
			\item We give an argument of uniform integrability for this limit to apply Lemma \ref{lemdebase} and conclude.
		\end{enumerate}
		
		{\bf Step 1:} Let $\psi:\RR_{+}^\kappa\times\RR^{\kappa-1}\to\RR$ a bounded continuous function and let us pass to the limit in the expectation 
		\begin{align}\label{eq:exp}
			\mathbb{E}\left[\psi\left(\widebar{\Bell}^{(n)}[\kappa],\mathbf{x}^{(n)}[\kappa-1]\right)\right]&=\sum_{\ssk\in\NN_\kappa(n)}\int_{\RR^\kappa}\psi\left(n\ell[\kappa],\frac{s[\kappa-1]-n/\kappa}{\sqrt{n/\kappa}}\right)f^{(n)}_{\kappa}\left(\ell[\kappa],\ssk\right)\mathrm{d}\ell[\kappa],
		\end{align}
		where the joint distribution $f^{(n)}_{\kappa}$ of the couple $(\lk,\sk)$ is given in Theorem \ref{thm:distri}.
		
		We perform both substitutions $\widebar{\ell}_j=n\ell_j$ and $x_j=(s_j-n/\kappa)/\sqrt{n/\kappa}$ in the right-hand side of \eqref{eq:exp}. We turn our sum over $\ssk\in\NN_\kappa(n)$ into an integral, in the following way :
		
		\[\frac{n!\sth^{n-\kappa}}{\ptk}\int_{\bNkn}\int_{\mathcal{L}_\kappa}\psi\left(n\ell[\kappa],\frac{\lfloor q[\kappa-1]\rfloor-n/\kappa}{\sqrt{n/\kappa}}\right)\prod_{j=1}^{\kappa}\frac{ \left(r_\kappa-\cl_j(\ell[\kappa])\right)^{\lfloor q_j\rfloor+\lfloor q_{\wj}\rfloor-1}}{\lfloor q_j\rfloor !(\lfloor q_j\rfloor+\lfloor q_{\wj}\rfloor-1)!}\cdot\mathrm{d}\ell[\kappa]\mathrm{d}q[\kappa-1],\] where $\flr{q_\kappa}$ is set to satisfy $\flr{q_\kappa}=n-\sum_{j=1}^{\kappa-1} \flr{q_j},$ (notice that there is no integration with respect to $q_\kappa$) and the integration is now done on the region $\bNkn:=\left\{q[\kappa-1],\text{ with }q_j>0\text{ and }\sum_{j=1}^{\kappa-1} q_j\leq n\right\}.$
		
		Let us rather consider the term
		\[\frac{n!\sth^{n-\kappa}}{\ptk}\int_{\bNkn}\int_{\mathcal{L}_\kappa}\psi\left(n\ell[\kappa],\frac{ q[\kappa-1]-n/\kappa}{\sqrt{n/\kappa}}\right)\prod_{j=1}^{\kappa}\frac{ \left(r_\kappa-\cl_j(\ell[\kappa])\right)^{\lfloor q_j\rfloor+\lfloor q_{\wj}\rfloor-1}}{\lfloor q_j\rfloor !(\lfloor q_j\rfloor+\lfloor q_{\wj}\rfloor-1)!}\cdot\mathrm{d}\ell[\kappa]\mathrm{d}q[\kappa-1]\] (we withdrew the floor function in $\psi$). This quantity turns out to be 
		$\mathbb{E}\left[\psi\left(\widebar{\Bell}^{(n)}[\kappa],\widebar{\mathbf{x}}^{(n)}[\kappa-1]\right)\right]$, where for all $j\in\{1,\ldots,\kappa-1\}$, we set $\widebar{\mathbf{x}}^{(n)}_j:=\mathbf{x}^{(n)}_j+\mathbf{U}_j/\sqrt{n/\kappa}$, with $\mathbf{U}_j$ a r.v. uniformly distributed in $\zerun$.
		
		We replaced a sum by an integral and this amounts to representing a discrete random variable by a continuous one, \ie if $X$ has a discrete law, $\mathbb{P}(X=k)=p_k,k\in\ZZ$, then $X\sur{=}{(d)}\lfloor X+U\rfloor$, where $U$ is uniform in $\zerun$. Then, 
		\[\sum_{k\in\ZZ} p_kf(k)=\int_\RR f(\lfloor x\rfloor) p_{\lfloor x\rfloor}\mathrm{d}x.\]
		We are going to prove first that $\mathbb{E}\left[\psi\left(\widebar{\Bell}^{(n)}[\kappa],\widebar{\mathbf{x}}^{(n)}[\kappa-1]\right)\right]$ converges to deduce that its fellow counterpart $\mathbb{E}\left[\psi\left(\widebar{\Bell}^{(n)}[\kappa],{\mathbf{x}}^{(n)}[\kappa-1]\right)\right]$ converges as well, to the same limit.
		
		We obtain after substitution : 
		\begin{align}\label{jpp1}
			\mathbb{E}\left[\psi\left(\widebar{\Bell}^{(n)}[\kappa],\widebar{\mathbf{x}}^{(n)}[\kappa-1]\right)\right]=\int_{\RR^{\kappa-1}}\int_{\RR_+^\kappa}\psi(\widebar{\ell}[\kappa],x[\kappa-1])g^{(n)}_{\kappa}(\widebar{\ell}[\kappa],x[\kappa-1])\mathrm{d}\widebar{\ell}[\kappa]\mathrm{d}x[\kappa-1]
		\end{align}	
		where for all $n\geq 3,$ $g^{(n)}_{\kappa}$ stands for the joint distribution of a pair $\left(\widebar{\Bell}^{(n)}[\kappa],\widebar{\mathbf{x}}^{(n)}[\kappa-1]\right)$. With convention $x_\kappa=-\sum_{i=1}^{\kappa-1}x_i$, the function $g^{(n)}_{\kappa}$ can be decomposed as follows
		\begin{align}\label{gn}
			g^{(n)}_{\kappa}\left(\widebar{\ell}[\kappa],x[\kappa-1]\right):=\omega(n,\kappa)~h^{(1)}_n(\widebar{\ell}[\kappa],x[\kappa-1])~h^{(2)}_n(x[\kappa-1]),
		\end{align}
		with 
		\begin{align}\label{equ7}
			\omega(n,\kappa)=\frac{n!\sth^{n-\kappa}}{\ptk}\left[\frac{1}{\sqrt{2\pi^{\kappa+1}\mathbb{d}_\kappa}}\frac{\kappa^{3n}e^{3n}}{4^n n^{3n}}\right]\left[\bigg(\frac{\kappa r_\kappa}{2m_\kappa}\bigg)^\kappa r_\kappa^{2n-\kappa}\right]\frac{1}{n^\kappa}\sqrt{\frac{n}{\kappa}}^{\kappa-1},
		\end{align} and 
		\begin{align}
			h^{(1)}_n\left(\widebar{\ell}[\kappa],x[\kappa-1]\right)&=\frac{1}{r_\kappa^{2n-\kappa}}\bigg(\frac{2m_\kappa}{\kappa r_\kappa}\bigg)^\kappa~\mathbb{1}_{\widebar{\ell}[\kappa] \in n\mathcal{L}_\kappa}~\prod_{j=1}^{\kappa}\bigg( r_\kappa-\frac{1}{n}\cl_j(\widebar{\ell}[\kappa])\bigg)^{d^{(2)}_j(x[\kappa-1])},\\
			h^{(2)}_n\left(x[\kappa-1]\right)&= \sqrt{2\pi^{\kappa+1}\mathbb{d}_\kappa}\frac{4^n n^{3n}}{\kappa^{3n}e^{3n}}\prod_{j=1}^{\kappa}\frac{1}{d^{(1)}_j(x[\kappa-1])!\cdot d^{(2)}_j(x[\kappa-1])!},\\[2pt]
			\intertext{where}
			d^{(1)}_j\left(x[\kappa-1]\right)&=\lfloor n/\kappa + \sqrt{n/\kappa}x_j \rfloor,\text{ for all }j\in\entk,\\[10pt]
			d^{(2)}_j\left(x[\kappa-1]\right)&=\lfloor 2n/\kappa-1 +\sqrt{n/\kappa}(x_j+x_{\wj})\rfloor,\text{ for all }j\in\entk.
		\end{align}
		We have arranged the factors so that, as we will see, $h^{(1)}_n$ and $h^{(2)}_n$ converge to some probability densities.
		
		Note first that there exists $\eta>0$ such that $[0,\eta]^\kappa\subset \mathcal{L}_\kappa$ and thus, we have $n\mathcal{L}_\kappa\cvg \RR_{+}^\kappa$. Then for every compact $K\subset \RR_{+}^\kappa$, and for all $\varepsilon>0,$ there exists $n_0\in\NN$ such that for all $n\geq n_0,$ $K\subset n\mathcal{L}_\kappa$, \ie $\norm{\mathbb{1}_{\widebar{\ell}[\kappa]\in n\mathcal{L}_\kappa}-1}=0<\varepsilon$, so that the map $\widebar{\ell}[\kappa]\mapsto\mathbb{1}_{\widebar{\ell}[\kappa]\in n\mathcal{L}_\kappa}$ converges uniformly to the constant function $1$ on every compact set of $\RR_+^\kappa$. Now by the standard approximation $\left(1-\frac{a}{n}\right)^{nb}\underset{n\to+\infty}{\longrightarrow}e^{-ab}$ uniformly for $(a,b)$ on every compact set, we get that $h^{(1)}_n$ 
		converges uniformly on every compact set of $\RR_{+}^\kappa\times\RR^{\kappa-1}$ towards $h^{(1)}$ with 
		\begin{align}\label{equ8}
			h^{(1)}\left(\widebar{\ell}[\kappa]\right)=\bigg(\frac{2m_\kappa}{\kappa r_\kappa}\bigg)^\kappa\prod_{j=1}^{\kappa}\exp\left(-\frac{2m_\kappa}{\kappa r_\kappa}\widebar{\ell}_j\right).
		\end{align}
		
		Now thanks to \Cref{cor1}, we have for $x[\kappa-1]$ fixed in $\RR^{\kappa-1},$
		\begin{align}
			h_n^{(2)}(x[\kappa-1])&\underset{n\to +\infty}{\sim}\sqrt{2\pi^{\kappa+1}\mathbb{d}_\kappa}\frac{4^n n^{3n}}{\kappa^{3n}e^{3n}}\bigg(\frac{\kappa}{n}\bigg)^n\bigg(\frac{\kappa}{2n-\kappa}\bigg)^{2n-\kappa}e^{3n-\kappa}\prod_{j=1}^{\kappa}\frac{\sqrt{\kappa}e^{-x_j^2/2}}{\sqrt{2\pi n}}\frac{\sqrt{\kappa}e^{-(x_j+x_{\wj})^2/4}}{\sqrt{2\pi (2n-\kappa)}}. \nonumber
		\end{align}
		After simplifications, this actually rewrites as the convergence on every compact set of $\RR^{\kappa-1}$ of $h_n^{(2)}$ towards $h^{(2)}$ where
		\begin{align}
			h^{(2)}\left(x[\kappa-1]\right)=\sqrt{\frac{\mathbb{d}_\kappa}{(2\pi)^{\kappa-1}}}\exp\left(-\frac{1}{2}x^{t}\Sigma_\kappa^{-1}x\right).
		\end{align}
		We have established the following uniform convergence on every compact set of $\RR_+^{\kappa}\times\RR^{\kappa-1}$ :
		\begin{align}\label{equ2}
			\frac{1}{\omega(n,\kappa)}g^{(n)}_{\kappa}(\widebar{\ell}[\kappa],x[\kappa-1])\underset{n\to +\infty}{\longrightarrow} g_{\kappa}(\widebar{\ell}[\kappa],x[\kappa-1]).
		\end{align}
		This concludes the step 1. of our proof. \\
		{\bf Step 2:} We will apply \Cref{lemdebase} to the sequence of functions $g_n = \frac{1}{\omega(n,\kappa)}g^{(n)}_{\kappa}, $ and to $g=g_{\kappa}$ which is already known to be a density. We therefore need to check that we control the mass of $\frac{1}{\omega(n,\kappa)}g^{(n)}_{\kappa}$ out of a certain compact set. 
		
		For any compact $K'$ of $\RR_+^\kappa$, we have $\mathbb{1}_{\widebar{\ell}[\kappa] \in n\mathcal{L}_\kappa\cap (K')^c}\leq\mathbb{1}_{\widebar{\ell}[\kappa] \in n\mathcal{L}_\kappa}$, and with $n\mathcal{L}_\kappa$ being a compact set of $\RR_+^\kappa$, there exists some $N\in\NN$ such that when $n\geq N$, we have for all $(\widebar{\ell}[\kappa],x[\kappa-1])\in (K')^c\times \RR^{\kappa-1}$:
		\begin{align}\label{hn1}
			h_n^{(1)}(\widebar{\ell}[\kappa],x[\kappa-1])&\leq2h^{(1)}(\widebar{\ell}[\kappa]).
		\end{align}
		
		Let now $\varepsilon>0$ be fixed for the rest of this proof and let us build the compact set $K_\varepsilon$ out of which we control the mass of $h_n^{(2)}.$ We can reinterpret the map $h_n^{(2)}$ as follows : if $M_1,M_2$ are two independent multinomial variables, with $M_1\sim\mathcal{M}(n;\frac{1}{\kappa},\ldots,\frac{1}{\kappa})$ and $M_2\sim\mathcal{M}(2n-\kappa;\frac{1}{\kappa},\ldots,\frac{1}{\kappa})$, for \[\mathbf{P}\left(x[\kappa-1]\right)=\mathbb{P}\left(M_1=d^{(1)}(x[\kappa-1])[\kappa],M_2=d^{(2)}(x[\kappa-1])[\kappa]\right),\]we have
		\begin{align}\label{hn}\nonumber
			h_n^{(2)}(x[\kappa-1])&=\frac{\sqrt{2\pi^{\kappa+1}\mathbb{d}_\kappa}}{\kappa^\kappa}\frac{4^n n^{3n}}{e^{3n}n!(2n-\kappa)!}\mathbf{P}\left(x[\kappa-1]\right),\\
			&\leq \frac{2^\kappa\sqrt{\pi}^{\kappa-1}\sqrt{\mathbb{d}_\kappa}}{B\kappa^\kappa}n^{\kappa-1}\mathbf{P}\left(x[\kappa-1]\right),
		\end{align}
		for $n$ large enough we both have $n!\geq 2^{-1/2}\sqrt{\pi}n^{n+1/2}e^{-n}$ and the existance of a constant $\alpha$ such that for all $n$, $(2n-\kappa)^{2n-\kappa+1/2}\geq \alpha(2n)^{2n-\kappa+1/2}$, we can set $B:=2\pi e^\kappa$ so that
		\[n!(2n-\kappa)!\geq B\frac{4^n n^{3n}}{e^{3n}n^{\kappa-1}}, \text{ for all }n\geq1.\]
		Let $M_k(i),k\in\{1,2\},i\in\entk$ be the $i^{th}$ entry of the multinomial r.v. $M_k$. Recall that the entry $M_1(i)$ is a binomial r.v. $\mathcal{B}(n,\frac{1}{\kappa})$, as well as, for $i\neq j,$ the law of $M_1(i)$ conditioned on $M_1(j)=k_j$ is a binomial distribution $\mathcal{B}(n-k_j,\frac{1}{\kappa-1})$. The same analogous results hold for $M_2(i).$ Now, since these marginals are binomial random variables, they are concentrated around their mean. We will design a compact $K_\varepsilon$ such that $K_\varepsilon^c$ contains the elements $x[\kappa-1]$ whose $i^{th}$ entry (for at least one $i$) is far from its expected value (which will give us exponential small bounds). 
		
		Let us rewrite by presenting the multinomial r.v. as sort of Markov Chains. We have 
		\begin{multline}
			\mathbf{P}\left(x[\kappa-1]\right)=\\\prod_{i=1}^{\kappa-1}\mathbb{P}\bigg(M_1(i)=\flr{\frac{n}{\kappa}+x_i\sqrt{\frac{n}{\kappa}}},M_2(i)=\flr{\frac{2n-\kappa}{\kappa}+(x_i+x_{{i-1}})\sqrt{\frac{n}{\kappa}}}~\bigg\vert G_1(x,i),G_2(x,i)\bigg)
		\end{multline}
		where \[G_1(x,i):=\bigcap_{j=1}^{i-1}\left\{M_1(j)=\flr{\frac{n}{\kappa}+x_j\sqrt{\frac{n}{\kappa}}}\right\} \text{ and } G_2(x,i):=\bigcap_{j=1}^{i-1}\left\{M_2(j)=\flr{\frac{2n-\kappa}{\kappa}+(x_j+x_{{j-1}})\sqrt{\frac{n}{\kappa}}}\right\}.\]
		For all $x[\kappa-1]\in\RR^{\kappa-1}$ and  $i\in\{1,\ldots,\kappa-1\},$ let $Y_i(x)$ be a binomial r.v. with the same law as $M_1(i)\vert G_1(x,i),$ \ie \[Y_i(x)\sim\mathcal{B}\left(\flr{\frac{n}{\kappa}(\kappa-i+1)-\sqrt{\frac{n}{\kappa}}\left(x_1+\ldots+x_{i-1}\right)},\frac{1}{\kappa-i+1}\right).\]
		A standard inequality for binomial distributions $\mathbf{x}\sim\mathcal{B}(m,q)$ with $q\in[a,b],0<a<b<1,$ \cite[III.5.2]{petrov1975sums} gives the existence of a constant $C>0$ such that for all $x$,
		\begin{align}\label{binom}
			\mathbb{P}(\mathbf{x}=x)\leq\frac{1}{C\sqrt{m}}.
		\end{align}
		For $n$ large enough, this implies the existence of a constant $C_\kappa>0$ such that
		\begin{align}\label{eq:swag2}
			\mathbf{P}\left(x[\kappa-1]\right)\leq \frac{ \prod_{i=1}^{\kappa-1}\mathbb{P}\bigg(M_1(i)=\flr{\frac{n}{\kappa}+x_i\sqrt{\frac{n}{\kappa}}}~\big\vert~
				G_1(x,i)\bigg) }{C_\kappa\sqrt{n}^{\kappa-1}}.
		\end{align}
		
		Controlling the map $\mathbf{P}$ allows one to control the map $h_n^{(2)}$ (recall \eqref{hn}), and thus $g_n$ (recall \eqref{gn}). 
		Define the sequence $e[\kappa]$ as $e_1=\sqrt{\kappa}$ and for all $j\in\{2,\ldots,\kappa-1\}$,
		\begin{align}\label{eq:swag}
			e_j=\sqrt{\kappa}+\frac{e_1+\ldots+e_{j-1}}{\kappa-j+1}.
		\end{align}
		We will use $e[\kappa]$ to define an event of $\mathbf{P}$ that has small probability. Let $M>0$, and $t\in\{1,\ldots,\kappa-1\}$.
		We define the set
		\begin{align}
			\widehat{B}_M(t)=\left\{w[\kappa-1]\in\RR^{\kappa-1}, \text{ such that } \abso{w_t}>Me_t \text{ and } \abso{w_j}\leq Me_j,1\leq j\leq t-1\right\}.
		\end{align}
		Intuitively, forcing the multinomial variable $M_k,k\in\{1,2\}$ to be in $\widehat{B}_M(t)$ is a huge condition for $M$ large since this imposes to $M_k$ to have a coordinate far from its mean.
		Since for all $x[\kappa-1]\in \widehat{B}_M(t)$, by \eqref{eq:swag}, we have 
		\[\sqrt{\frac{n}{\kappa}}\left\vert x_t+\frac{x_1+\cdots+x_{t-1}}{\kappa-t+1}\right\vert>\sqrt{\frac{n}{\kappa}}M\sqrt{\kappa}=M\sqrt{n},\] 
		we may write, (with a small change of variables $t_i =\frac{n}{\kappa}+x_i\sqrt{\frac{n}{\kappa}}$  and then $\sup$ on $\widehat{B}_M(t)$), by \eqref{eq:swag2}:
		\begin{multline}
			\int_{ \widehat{B}_M(t)}\prod_{i=1}^{\kappa-1}\mathbb{P}\bigg(M_1(i)=\flr{\frac{n}{\kappa}+x_i\sqrt{\frac{n}{\kappa}}}~\bigg\vert~
			G_1(x,i)\bigg)\mathrm{d}x
			\leq\sup_{x\in\widehat{B}_M(t)}\left[\prod_{i=1}^{t-1}\mathbb{P}\left(\abso{Y_i(x)-\mathbb{E}[Y_i(x)]}\leq M\sqrt{n}\right)\right]\\
			\times\mathbb{P}\left(\abso{Y_t(x)-\mathbb{E}[Y_t(x)]}>M\sqrt{n}\right)\left[\prod_{i=t+1}^{\kappa-1}\mathbb{P}\left(\abso{Y_i(x)-\mathbb{E}[Y_i(x)]}\in\RR\right)\right]\frac{1}{\sqrt{n}^{\kappa-1}}.
		\end{multline}

		We bound the terms different from $t$ in the product by $1$, and handle the term in $t$ by Hoeffding's inequality \cite[III.5.8]{petrov1975sums}, \ie 
		\[\mathbb{P}\left(\abso{Y_t(x)-\mathbb{E}[Y_t(x)]}>M\sqrt{n}\right)\leq 2\exp(-2M^2),\text{ for all }x[\kappa-1]\in\widehat{B}_M(t).\]
		Let us check that there exists $M:=M(\varepsilon)$ large enough such that the integral of the map $\mathbf{P}$ out of $K_\varepsilon:=[-M\sqrt{\kappa},M\sqrt{\kappa}]^{\kappa-1}$ is controlled.
		
		With the decomposition $A_M:=\RR^{\kappa-1}\backslash[-M\sqrt{\kappa},M\sqrt{\kappa}]^{\kappa-1}=\bigsqcup_{t=1}^{\kappa-1}\widehat{B}_M(t),$ we may write
		\begin{align}
			\int_{A_M}\mathbf{P}(x[\kappa-1])\mathrm{d}x=\sum_{t=1}^{\kappa-1}\int_{\widehat{B}_M(t)}\mathbf{P}(x[\kappa-1])\mathrm{d}x\leq\frac{2(\kappa-1)}{C_\kappa n^{\kappa-1}}\exp(-2M^2).
		\end{align}
		where we set $\mathrm{d}x=\prod_{i=1}^{\kappa-1}\mathrm{d}x_i.$
		With $\varepsilon$ being fixed, we may now choose $M:=M(\varepsilon)>0$ sufficiently large so that 
		\begin{align}
			\frac{2^\kappa\sqrt{\pi}^{\kappa-1}\sqrt{\mathbb{d}_\kappa}}{B\kappa^\kappa}\frac{2(\kappa-1)}{C_\kappa}\exp\left(-2M^2\right)<\frac{1}{2}\varepsilon.
		\end{align} 
		For such a $M$, we put $K_\varepsilon=[-M\sqrt{\kappa},M\sqrt{\kappa}]^{\kappa-1}$. In this case, we have indeed
		\begin{align}\label{bound2}\nonumber
			\int_{K_\varepsilon^c}h_n^{(2)}&\leq\frac{2^\kappa\sqrt{\pi}^{\kappa-1}\sqrt{\mathbb{d}_\kappa}}{B\kappa^\kappa}n^{\kappa-1}\int_{A_M}\mathbf{P}(x[\kappa-1])\mathrm{d}x\\
			&\leq\frac{2^\kappa\sqrt{\pi}^{\kappa-1}\sqrt{\mathbb{d}_\kappa}}{B\kappa^\kappa}\frac{2(\kappa-1)}{C_\kappa}\exp(-2M^2)<\frac{1}{2}\varepsilon.
		\end{align}	
		This ends the proof of Step 2 ! Indeed, if we sum up, we have uniform convergence on every compact set of $\frac{1}{\omega(n,\kappa)}g^{(n)}_{\kappa}$ to $g_{\kappa}$, and we built two compact sets $K'\subset \RR_+^{\kappa}$ and $K_\varepsilon\subset \RR^{\kappa-1}$ such that for all $n\geq N$,
		\begin{align}
			\int_{(K')^c\times K_\varepsilon^c}\frac{1}{\omega(n,\kappa)}g^{(n)}_{\kappa}&=\frac{1}{\omega(n,\kappa)}\int_{(K')^c\times K_\varepsilon^c}h^{(1)}_n(\widebar{\ell}[\kappa],x[\kappa-1])~h^{(2)}_n(x[\kappa-1])\mathrm{d}\widebar{\ell}[\kappa]\mathrm{d}x[\kappa-1]\\
			&\leq 2\underbrace{\int_{(K')^c}h^{(1)}}_{\leq 1}\int_{K_\varepsilon^c}h_n^{(2)}< \varepsilon,
		\end{align}
		where the first line comes from \eqref{gn} and the second line comes out of the bounds we gave in \eqref{hn1} and \eqref{bound2}.
		
		We may now conclude:
		we apply Lemma \ref{lemdebase}, so that there exists a (unique !) sequence normalizing $\frac{1}{\omega(n,\kappa)}g^{(n)}_{\kappa}$ into a density. Since $g^{(n)}_{\kappa}$ is already a density, this sequence is nothing but $\omega(n,\kappa)$, hence we have 
		\begin{align}\label{wn}
			\omega(n,\kappa)\sous{\longrightarrow}{n\to\infty}1
		\end{align}
		This also proves that $g^{(n)}_{\kappa}$ converges pointwise to $g_{\kappa}$, or by definition, that $(\widebar{\Bell}^{(n)}[\kappa],\widebar{\mathbf{x}}^{(n)}[\kappa-1])\dd (\widebar{\Bell}[\kappa],\mathbf{x}[\kappa-1])$. Now, since $\abso{\widebar{\mathbf{x}}^{(n)}[\kappa-1]-\mathbf{x}^{(n)}[\kappa-1]}\proba0$, then by Slutsky's Lemma, \[(\widebar{\Bell}^{(n)}[\kappa],{\mathbf{x}}^{(n)}[\kappa-1])\dd (\widebar{\Bell}[\kappa],\mathbf{x}[\kappa-1]).\]
		This ends the proof.

	\end{proof}
	
	\Cref{thm-1} actually turns out to be a nice corollary of \Cref{thmonstre}:
	
	\begin{proof}[Proof of \Cref{thm-1}]
		We saw in \eqref{wn} that the sequence $\omega(n,\kappa)$ introduced in \eqref{equ7} satisfies $\omega(n,\kappa)\underset{n\to +\infty}{\longrightarrow}1$. This allows us to determine $\ptk$ ! Indeed, Stirling's formula yields to
		\[\ptk \underset{n\to +\infty}{\sim} \frac{1}{\pi^{\kappa/2}\sqrt{\mathbb{d}_\kappa}}\frac{\sqrt{\kappa}^{\kappa+1}}{4^\kappa(1+\cth)^{\kappa}}\frac{e^{2n}\kappa^{3n}r_\kappa^{2n}\sth^n}{4^n n^{2n+\kappa/2}}.\]
		
		Having $\ptk\underset{n\to +\infty}{\sim}\pk$ by \Cref{lem1}, we obtain the expected result.
	\end{proof}

	\section{Fluctuations around the limit shape}
	\label{sec:FLS}
	\subsection{Basic bricks fluctuations}
	\paragraph{Notation.}We say that a point $z$ is the $\alpha$-barycenter of $(a,b)$ if $z= \alpha a+(1-\alpha) b$, and $\alpha$ is called the barycenter parameter. 
	\paragraph{Basic building bricks of a $\zzn$-gon.}For any $\zn$ with distribution $\Dtn,$ we provided an alternative geometrical description of the $\zn$-gon in terms of its $\ECP(\zn),$ and more specifically in terms of:\\
	$\bullet$ the boundary distances $\lk$ of this $\ECP,$\\
	$\bullet$ the tuple $\sk$ counting the vectors in the corners of the $\ECP$,\\
	$\bullet$ the fragmentation of the sides $\ck$ into side-partitions $\mathbf{u}^{(1)}[\mathbf{N}^{(n)}_1],\ldots,\mathbf{u}^{(\kappa)}[\mathbf{N}^{(n)}_\kappa],$ where $\mathbf{N}^{(n)}_j=\snj+\mathbf{s}^{(n)}_{\wj}-1$ for all $j\in\entk.$	
	
	Notice that the contact points $\Bcp^{(n)}[\kappa]$ and the vertices $\BBb^{(n)}[\kappa]$ of the $\ECP(\zn)$ can be recovered using these three data.
	Indeed, $\lk$ determines the $\ECP$ and thus its vertices $\BBb^{(n)}[\kappa]$. 
	
	In \eqref{eq:pourrave}, we see that, conditional on the $j^{th}$ side length $\mathbf{c}^{(n)}_j=c_j$ and on the size-vector $\sk$, the side-partition $\left(u_1^{(j)}<\ldots<u_{\mathbf{N}_j}^{(j)}\right)$ has the law of a reordered $\mathbf{N}_j$-tuple of i.i.d. uniform r.v. drawn in $[0,c_j].$
	The contact point $\Bcp^{(n)}_j$ is placed on the $j^{th}$ side of the $\ECP$ (on the segment $[\BBb^{(n)}_{{j-1}},\BBb^{(n)}_j]$), at the coordinate $\mathbf{u}^{(j)}_{\snj}$ for all $j\in\entk$.
	This means that conditional on $(\lk,\sk)$, the tuple $\Bcp^{(n)}[\kappa]$ has independent entries, and $\Bcp^{(n)}_j$ is a $\Bbeta_j^{(n)}$-barycenter of $\left(\BBb^{(n)}_{{j-1}},\BBb^{(n)}_j\right)$, where $\Bbeta_j^{(n)}$ is $\beta$-distributed with parameters $\left(\snj,\mathbf{s}^{(n)}_{\wj}\right)$.
	
	For $j\in\entk$, the r.v.\begin{align}\label{eq:beta}
		\Bdeta^{(n)}_j:=\sqrt{\frac{n}{\kappa}}\left(\Bbeta_j^{(n)}-1/2\right),\text{ for all }j\in\entk
	\end{align}
	provides the fluctuations of the barycenter parameter and thus encodes the fluctuations of the contact point $\Bcp_{j}^{(n)}$ on the $j^{th}$ side $\mathbf{c}^{(n)}_{j}$ of $\ECP(\zn).$
	
	\paragraph{Fluctuations of these basic bricks.}
	We proved that the boundary distances $\lk$ behaved as $\frac{1}{n}$ times an exponential distribution, and we will prove in the sequel that the contact points $\Bcp^{(n)}[\kappa]$ are typically at distance $\frac{1}{\sqrt{n}}$ around their limit (this is visible in \eqref{eq:beta}). So in order to describe the fluctuations of the $\zn$-gon around its limit, we will state a theorem describing the joint distribution of all basic bricks taken together, rather than providing the fluctuations in $\frac{1}{\sqrt{n}}$ of a complicated object for the whole process, that would crush the behaviour of some of the bricks. A comprehensive picture of the fluctuations would rely on the concatenation of all the corners' fluctuations, adjusted to take into account the contact points: we believe that doing such a thing does not bring any new insight, and then, we let this to the interested reader as an exercice.
	
	\paragraph{The $\lt$-convex chains.}Recall \Cref{lem:rect} and its notation. Consider for all $j\in\entk$ the convex chain $\CC^{(n)}_j=\CC_j^{(n)}(\zn)$ lying in the $j^{th}$ corner of $\ECP(\zn),$ defined as an element of $\chain_{\snj}(\corner_j(\zn))$ (this accounts for the decomposition of $\zn$ into convex chains).
	In order to understand the fluctuations of the $j^{th}$ corner, we will use the normalized version of each of these convex chains: 
	\begin{align}
		\nCC^{(n)}_j=\Aff_j(\CC^{(n)}_j) \text{ where }\Aff_j=\Aff_{\corner_j(\zn)},
	\end{align}
	where, for a given nonflat triangle $ABC$, the mapping $\Aff_{ABC}$ is the unique map that sends $A,B,C$ on $(0,0),(1,0),(1,1)$ as introduced in \Cref{lem:aff}. The law of $\nCC^{(n)}_j$ only depends on $\snj,$ but, the affine mapping $\Aff_j$ depends on the coordinates of the $\ECP(\zn).$ However, determining the fluctuations of $\CC^{(n)}_j$ amounts to looking at those of $\nCC^{(n)}_j$.
	\par Recall that a generic $\lt$-normalized convex chain $\nCC_m$ of size $m$ is a random variable whose law is that of a convex chain taken uniformly in $\chain_{m}(\lt)$.
	Therefore, a consequence of \Cref{lem:rect} and \Cref{lem:aff} is the following Lemma:
	\begin{Lemma}\label{lem:indcvx} Conditional on $(\lk,\sk,\Bdeta^{(n)}[\kappa])$,
		the convex chains $\nCC^{(n)}[\kappa]$ are independent. Furthermore, for all $j\in\entk$, the distribution of $\nCC_j^{(n)}$ is that of a generic $\lt$-normalized convex chain of size $\snj$, \ie $\nCC_j^{(n)}\eqd\nCC_{\snj}$.
	\end{Lemma}
	
	Thanks to this Lemma, it should be clear that we can work on each convex chain, separately, when $\sk$ is fixed. We introduce some processes in order to describe these fluctuations.

	\subsection{A parametrization of $\nCC_m$}%a $\lowerrighttriangle$ convex chain and its limit}

Let $m\geq0$, and let $\nCC_m=((0,0),\mathbf{z}_1,\ldots,\mathbf{z}_{m-1},(1,1))$ a generic $\lt$-normalized convex chain of size $m$. Rather than considering the tuple of points $((0,0),\mathbf{z}_1,\ldots,\mathbf{z}_{m-1},(1,1))$, we consider the $m$ vectors composing the convex chain. Recall that these vectors are obtained by forming the tuples $\mathbf{u}[m],\mathbf{v}[m]$ of increments of two elements taken uniformly (and independently) in the simplex $P[1,m-1]$. Then the vectors $(\mathbf{u}_{i},\mathbf{v}_{i}),i\in\{1,\ldots,m\}$ are reordered by increasing slope to form this chain.

In order to use the toolbox of stochastic processes, it is convenient to see $\nCC_m$ as a linear process. For this we will need a suitable parametrization (several choices are possible). For technical reasons, we chose the local slope of $\nCC_m$ as the time parameter, and apply the $\arctan$ function in order to remain in a compact set.

Consider the point $\left(x_{m}(u),y_{m}(u)\right)$ for $u\in\zerun$, corresponding to the contributions of the previous vectors whose slope are smaller than $\tnorm{u}$, and the process $\fcs_m$ defined as $\fcs_m(u):=\left(x_{m}(u),y_{m}(u)\right)$ for $u\in\zerun$. Hence, we have
\begin{align}\label{eq:defprocess}
	x_m(u)=\sum_{i=1}^m \mathbf{u}_i\mathbb{1}_{\frac{\mathbf{v}_i}{\mathbf{u}_i}\leq \tnorm{u}}\quad\text{ and }\quad y_m(u)=\sum_{i=1}^m \mathbf{v}_i\mathbb{1}_{\frac{\mathbf{v}_i}{\mathbf{u}_i}\leq \tnorm{u}}.
\end{align}
where we set $\tnorm{\cdot1}=+\infty$ so that $x_m(1)=y_m(1)=1.$ 
{\bf Note }that the tuple $\nCC_m$ (seen as a set) coincides with the (finite) set $\{\fcs_m(u),u\in\zerun\}$.

Define further the curve
$\fcs_\infty(u):=\left(x(u),y(u)\right)$ for all $u\in\zerun$, where
\begin{align}\label{eq:baranylimit}
	x(u):=1-\frac{1}{\left(1+\tnorm{u}\right)^2}\quad\text{ and }\quad y(u):=\frac{\tnorm{u}^2}{\left(1+\tnorm{u}\right)^2}.
\end{align}
We have once more $x(1)=y(1)=1$.
This curve is actually the parametrization of the parabola arc lying in $\lowerrighttriangle,$ tangent at $(0,0)$ and $(1,1)$. It is not surprising to encounter $\fcs_\infty$ here, since the convergence in distribution of $\fcs_m$ to $\fcs_{\infty}$ that will be stated in \Cref{lem:cvYm} can be seen as a consequence of Bárány's work on affine perimeters (not a direct consequence, however). 

For all $j\in\entk$, we let $\fcs^{(n)}_{j}$ be the parametrization in terms of the slope of the $j^{th}$ $\lt$-normalized convex chain $\nCC_j^{(n)}.$ So now by \Cref{lem:indcvx}, we have 
\[{\cal L}\left(\fcs_j^{(n)}, 1\leq j \leq \kappa~|~\sk\right)={\cal L}\left(
\fcs_{\snj}, 1\leq j \leq \kappa\right),\]
and conditional on $\sk$, the $\fcs^{(n)}_{j},j\in\entk$ are independent.

\Cref{lem:cvYm} will also include the convergence in distribution of the fluctuation process $\fys_m:=\sqrt{m}(\fcs_m-\fcs_{\infty})$ to a Gaussian process. This convergence is actually the key that will help us understand the fluctuations of the convex chain $\nCC_j^{(n)}$ around its limit $\nCC_\infty$. Indeed, conditional on $(\lk,\sk),$ the processes 
\begin{align}\label{Ym}
	\fys_j^{(n)}:=\sqrt{\frac{n}{\kappa}}\left(\mathsf{C}^{(n)}_{j}-\mathsf{C}_\infty\right),~~~ 1\leq j \leq \kappa\
\end{align}
describe the successive fluctuations of the $\lt$-convex chain in their corners.
We are now able to state the most important result of this section (where we borrowed the notation of \Cref{thmonstre}):

\begin{Theoreme}\label{thm:globfluc}
	0. Conditional on $(\lk,\sk)$, the processes $\fcs^{(n)}[\kappa]$ are independent and have the same distribution. For all $j\in\entk,$ the process $\fcs^{(n)}_{j}$ converges in $D(\zerun,\RR)$, that we endow with the Skorokhod topology  for the rest of this paper, to the deterministic process $\fcs_\infty$ introduced in \eqref{eq:baranylimit}.
	\par Furthermore, the "fluctuation" tuple
	$\left(\widebar{\Bell}^{(n)}[\kappa],\mathbf{x}^{(n)}[\kappa-1],\Bdeta^{(n)}[\kappa],\fys^{(n)}[\kappa]\right)$ converges in distribution in $\RR^\kappa\times\RR^{\kappa-1}\times\RR^{\kappa}\times D(\zerun,\RR)^\kappa$ (equipped with the corresponding product topology) to a tuple $\left(\widebar{\Bell}[\kappa],\mathbf{x}[\kappa-1],\Bdeta[\kappa],\fys[\kappa]\right)$
	where 
	\begin{enumerate}
		\item $\left(\widebar{\Bell}[\kappa], \mathbf{x}[\kappa-1]\right)$ and their fluctuations have already been described in \Cref{thmonstre}; 
		\item Conditional on $\left(\widebar{\Bell}[\kappa], \mathbf{x}[\kappa-1]\right)$, the variables $\Bdeta_1,\ldots,\Bdeta_\kappa$ are independent, and for all $j\in\entk,$ $\Bdeta_j$ is a normal random variable with mean $(\mathbf{x}_j-\mathbf{x}_{\wj})/4$ and variance $1/8$.
		\item The tuple $\fys^{(n)}[\kappa]$ converges in distribution in $D(\zerun)$ to the tuple $\fys[\kappa]$ where the processes $\fys_1,\ldots,\fys_\kappa$ are i.i.d., and their common distribution is that of $\fys_\infty$, a Gaussian process whose law will be detailed in \Cref{lem:cvYm}.
	\end{enumerate}
	This convergence theorem contains the fluctuations of every basic brick of the $\ECP(\zn)$.
\end{Theoreme}

\begin{proof}
	We start by the proof of {\it point 2.} Recall \eqref{eq:beta}. Conditional on $(\lk,\sk)$, the $\Bcp^{(n)}[\kappa]$ are independent, which provides the independance of the $\Bdeta^{(n)}[\kappa]$ (given  $(\lk,\sk)$). It suffices then, to prove the convergence of the marginals: by symmetry, we will prove only the convergence of $\Bdeta_j^{(n)}=	\sqrt{\frac{n}{\kappa}}\left(\Bbeta_j(n)-\frac{1}{2}\right)$.\par
	According to the Skorokhod representation theorem, (up to a probability space changing), we may assume that $\frac{\snj-\frac{n}{\kappa}}{\sqrt{n/\kappa}}\as \mathbf{x}_j$ for all $j\in\entk$. We may then assume that $\snj=\frac{n}{\kappa}+\mathbf{x}_j\sqrt{\frac{n}{\kappa}}+o(\sqrt{n})$. In the sequel we drop the $o(\sqrt{n})$ since it provides only neglectible contributions.
	
	Since $\beta_j^{(n)}$ is $\beta$-distributed with parameters $(\snj,\mathbf{s}^{(n)}_{\wj}),$ we have $\beta_j^{(n)}\eqd\frac{\mathbf{T}_j}{\mathbf{T}_j+\mathbf{T}_{\wj}}$ where  $\mathbf{T}_j$ (resp. $\mathbf{T}_{\wj}$) be a r.v Gamma distributed with parameter $\snj$ (resp. $\mathbf{s}^{(n)}_{\wj}$), these variables being independent. 
	
	By the Central Limit theorem, we have 
	\begin{align*}
		\left(\frac{\mathbf{T}_j-\frac{n}{\kappa}}{\sqrt{n/\kappa}},\frac{\mathbf{T}_{\wj}-\frac{n}{\kappa}}{\sqrt{n/\kappa}}\right)\dd\left(\mathbf{x}_j+\mathbf{q}_1,\mathbf{x}_{\wj}+\mathbf{q}_2\right).
	\end{align*}for $\mathbf{q}_1,\mathbf{q}_2$ two i.i.d. standard Gaussian r.v. $\mathcal{N}(0,1)$.
	Hence%	We may rewrite the fluctuations of $\cp^{(n)}_j$ as 
	\begin{align}
		\sqrt{\frac{n}{\kappa}}\left(\Bbeta_j^{(n)}-\frac{1}{2}\right)&\eqd\sqrt{\frac{n}{\kappa}}\left(\frac{\mathbf{T}_j}{\mathbf{T}_j+\mathbf{T}_{\wj}}-\frac{1}{2}\right)\\
		&\eqd\frac{n/\kappa}{2(\mathbf{T}_j+\mathbf{T}_{\wj})}\left(\frac{\mathbf{T}_j-\frac{n}{\kappa}}{\sqrt{n/\kappa}}-\frac{\mathbf{T}_{\wj}-\frac{n}{\kappa}}{\sqrt{n/\kappa}}\right)\\
		&\dd \frac{1}{4}\left(\mathbf{x}_j-\mathbf{x}_{\wj}+\mathbf{q}_1-\mathbf{q}_2\right).
	\end{align}
	This gives the expected result.\\
	
	\par Notice then that {\it 3. $\implies$ 0.} Indeed, the weak convergence of the processes $\fcs^{(n)}_1,\ldots,\fcs^{(n)}_\kappa$ to $\fcs_\infty$ is a consequence of that of $\fys^{(n)}_1,\ldots,\fys^{(n)}_\kappa$ to $\fys_\infty$. This latter convergence will be the main object of the rest of this section. This is by far the most complicated proof, and we will need to introduce quite a lot of tools to achieve it. We will come back to this proof later.\\
	
\end{proof}
\subsection{Convergence of the $\lowerrighttriangle$-normalized convex chain}\label{sec:nmz}

In order to prove the convergence of $\fys^{(n)}[\kappa]$, a new parametrization is required. 

\paragraph{Notation.}Set the maps $g:t\in\zerun\mapsto\tnorm{t},$ and $\displaystyle h:t\in\zerun\mapsto\frac1{1+g(t)}.$ Let us introduce the following mappings for all $(s,t)\in\zerun^2$:
\begin{align}\label{eq:super}
	f(s,t)&=h(s)-h(t),\\
	e_1(s,t)&=h(s)^2-h(t)^2,\\
	e_2(s,t)&=\left(g(t)h(t)\right)^2-\left(g(s)h(s)\right)^2,\\
	v_1(s,t)&=2\left(h(s)^3-h(t)^3\right)-e_1(s,t)^2,\\
	v_2(s,t)&=2\left(\left(g(t)h(t)\right)^3-\left(g(s)h(s)\right)^3\right)-e_2(s,t)^2.\label{eq:super2}
\end{align}
In the following, we will denote by $\left(\mathsf{Z}^{(1)},\mathsf{Z}^{(2)}\right)$ the coordinates of any process $\mathsf{Z}$ taking values in $\mathbb{R}^2$.
\par Recall the definition of the process $\fys_m$ given in \eqref{Ym}.
\begin{Theoreme}\label{lem:cvYm} 
	The following convergence in distribution
	\begin{align}\label{eq:conv}
		\fcs_m\dd \fcs_\infty,
	\end{align}
	holds in $D(\zerun)^2$.

	The sequence of processes $(\fys_m)$ converges in distribution to a centered Gaussian process $\fys_\infty$ in $D(\zerun)^2$, whose coordinates can be represented as follows:% can be decomposed in three independent processes:
	\begin{align}
		\fys_\infty^{(1)}(t)&=\widebar{\fys}^{(1)}(t)+\mathbf{g}_1\cdot x(t)+(\mathbf{g}_1-\mathbf{g}_2)\cdot \frac2{\pi}\frac{\tnorm{t}}{1+\tnorm{t}^2}\cdot x'(t),\\
		\fys_\infty^{(2)}(t)&=\widebar{\fys}^{(2)}(t)+\mathbf{g}_1\cdot y(t)+(\mathbf{g}_1-\mathbf{g}_2)\cdot \frac2{\pi}\frac{\tnorm{t}}{1+\tnorm{t}^2}\cdot y'(t), 
	\end{align}
	and where:
	\bir
	\itr The mappings $x',y'$ are the derivatives of $t\mapsto x(t)=\fcs_\infty^{(1)}(t),t\mapsto y(t)=\fcs_\infty^{(2)}(t)$ (they are deterministic processes),
	\itr The random variables $\mathbf{g}_1,\mathbf{g}_2$ are independent standard Gaussian variables ${\cal N}(0,1)$,
	\itr The process $\widebar{\fys}$ is a centered Gaussian process with variance function
	\begin{align}
		\mathbb{V}\left[\widebar{\fys}^{(p)}(t)\right]=f(0,t)v_p(0,t)+f(0,t)(1-f(0,t))e_p(0,t)^2\quad \forall p\in\{1,2\},
	\end{align}
	and with covariance function determined by, for $(s< t)\in\zerun^2,$
	\begin{align}
		\cov\left(\widebar{\fys}^{(p)}(s),\widebar{\fys}^{(q)}(t)-\widebar{\fys}^{(q)}(s)\right)=-f(0,s)f(s,t)e_p(0,s)e_q(s,t)\quad \forall (p,q)\in\{1,2\}^2.
	\end{align}
	\itr The processes $\widebar{\fys}^{(1)},\widebar{\fys}^{(2)}$ are independent from $\mathbf{g}_1$ and $\mathbf{g}_2$.
	\eir
\end{Theoreme}
The reason why the functions $f,e_p,v_p$ appear will be revealed in \Cref{prop-1}. Notice that $\lim_{t\to1}\frac2{\pi}\frac{\tnorm{t}}{1+\tnorm{t}^2}\cdot x'(t)$ and $\lim_{t\to1}\frac2{\pi}\frac{\tnorm{t}}{1+\tnorm{t}^2}\cdot y'(t)$ are finite so the process $\fys_\infty$ is also well-defined at $t=1$.

Once more, the first assertion \eqref{eq:conv} is a consequence of the second: we will only prove the second one.
\begin{Remarque}[Back to \Cref{thm:globfluc}]
	The proof of point 4. of \Cref{thm:globfluc} is an immediate consequence of \Cref{lem:cvYm} ! Indeed we have for all $j\in\entk$:
	\begin{align}
		\fys_j^{(n)}=\sqrt{\frac{n/\kappa}{\snj}}\sqrt{\snj}\left(\fcs_{\snj}-\fcs_\infty\right).
	\end{align}
	By \Cref{lem:cvYm}, the term $\sqrt{\snj}\left(\fcs_{\snj}-\fcs_\infty\right)$ converges in distribution to $\fys_\infty$, and we know furthermore that $\frac{\snj}{n/\kappa}\dd1$. Slutsky's Lemma allows to conclude.
\end{Remarque}

\begin{Remarque}
	An immediate consequence of \Cref{lem:cvYm} is that the curve ${\cal C}_m=\{(t,\fcs_m(t)), t \in \zerun\}$, seen as a compact set, converges in distribution to ${\cal C}_{\infty}=\{(t,\fcs_\infty(t)), t \in \zerun\}$, for the Hausdorff distance as $m\to+\infty$. Furthermore, the term $\sqrt{m}d_H( {\cal C}_m,{\cal C}_{\infty})$ converges in distribution to a nontrivial random value.
\end{Remarque}

\subsection{Proof of \Cref{lem:cvYm} }

The parametrization of $\fcs_m$ in terms of the variables $(\mathbf{u}[m],\mathbf{v}[m])$ is tricky since these variables are interconnected (they sum to $1$), paired and then sorted by increasing slope (even if the parametrization in terms of the indicator of the slope allows one to get rid of this difficulty). 
\paragraph{Exponential model.}Let $\Bzeta[m],\Bxi[m]$ be two $m$-tuples of random variables exponentially distributed with mean $1$, all these variables being independent. Set $\Balpha_m^{(1)}=\sum_{i=1}^m\Bzeta_i,\Balpha_m^{(2)}=\sum_{i=1}^m\Bxi_i$.
The following equalities in distribution are classically used to represent order statistics by exponential r.v. (a more general result can be found in \cite[Theorem 3.]{jambunathan}) :
\begin{align}\label{expo.model}
	\mathbf{u}[m]\sur{=}{(d)}\frac{1}{\Balpha_m^{(1)}}\Bzeta[m]\quad\text{ and }\quad \mathbf{v}[m]\sur{=}{(d)}\frac{1}{\Balpha_m^{(2)}}\Bxi[m].
\end{align}

\begin{Lemma}\label{lem:tcl}
	
	Let $\mathbf{g}_1,\mathbf{g}_2$ two independent standard Gaussian variables. The following convergence in distribution holds in $\RR^3$:
	\begin{align}\label{eq:qhjysfhgwf}
		\sqrt{m}\left[\left(\frac{m}{\Balpha_m^{(1)}}-1\right),\left(\frac{m}{\Balpha_m^{(2)}}-1\right),\bigg(\frac{\Balpha_m^{(2)}}{\Balpha_m^{(1)}}-1\bigg)\right]\xrightarrow[m]{(d)} \left[\mathbf{g}_1,\mathbf{g}_2,\mathbf{g}_1-\mathbf{g}_2\right].
	\end{align}
\end{Lemma}
\begin{proof}
	Let us write
	\begin{align*}
		\sqrt{m}\left(\frac{m}{\Balpha_m^{(j)}}-1\right)=\frac{m}{\Balpha_m^{(j)}}\times \frac{m-\Balpha_m^{(j)}}{\sqrt{m}} \text{ for all }j\in\{1,2\}.
	\end{align*}As for the third marginal in the l.h.s. of \eqref{eq:qhjysfhgwf}, write
	\begin{align*}
		\sqrt{m}\bigg(\frac{\Balpha_m^{(2)}}{\Balpha_m^{(1)}}-1\bigg)=\frac{m}{\Balpha_m^{(1)}}\times \left[\frac{m-\Balpha_m^{(1)}}{\sqrt{m}}-\frac{m-\Balpha_m^{(2)}}{\sqrt{m}}\right].
	\end{align*}
	Now in all these cases, Slutsky's Lemma together with the Central Limit Theorem give the expected convergence.
\end{proof}

We now build an object $\widebar{\fcs}_m$ close to $\fcs_m$ whose convergence is easier to prove. We pair the tuples $\Bzeta[m],\Bxi[m]$ to form $m$ vectors $\mathbf{w}_i=(\Bzeta_i,\Bxi_i)$ for all $i\in\ent{1}{m}$. When ordered by increasing slope and summed one by one, these vectors form the boundary of a convex polygon, whose vertices form a convex chain in the triangle ${\sf Tri}(m)$ of vertices $(0,0),(\Balpha_m^{(1)},0),(\Balpha_m^{(1)},\Balpha_m^{(2)})$.
If we renormalize the $x$-coordinates of these vectors by $\Balpha_m^{(1)}$ and the $y$-coordinates by $\Balpha_m^{(2)}$, we obtain a convex chain in $\lt$ whose law is that of a generic $\lt$-normalized convex chain. However, we want to study the convex chain before renormalization. Hence, to obtain the analogous process before nomalization, we consider the contribution $\widebar{\fcs}_m(u)$ of the vectors $(\Bzeta_i,\Bxi_i)$ with slope smaller or equal than $u\in\zerun$,
\begin{align}\label{defC}
	\widebar{\fcs}_m(u):=\frac{1}{m}\sum_{i=1}^m \left(\Bzeta_i,\Bxi_i\right)\mathbb{1}_{\frac{\Bxi_i}{\Bzeta_i}\leq \tnorm{u}}.
\end{align}
We will eventually send $\widebar{\fcs}_m$ on ${\fcs}_m$ by sending ${\sf Tri}(m)$ to $\lt$; in order to control the induced slope modification, we introduce the function $\Balpha_m$ defined in $\zerun$ by
\begin{align}
	\Balpha_m(u)=\frac{2}{\pi}\arctan\left(\frac{\Balpha_m^{(2)}}{\Balpha_m^{(1)}}\tnorm{u}\right),
\end{align}
and such that $\tnorm{\Balpha_m(u)}=\frac{\Balpha_m^{(2)}}{\Balpha_m^{(1)}}\tnorm{u}.$

By what we just explained, the link between $\widebar{\fcs}_m$ and $\fcs_m$ is the following:
\begin{align}
	\fcs_m&\eqd\left(\frac{m}{\Balpha_m^{(1)}}\widebar{\fcs}^{(1)}_m\big(\Balpha_m(u)\big),\frac{m}{\Balpha_m^{(2)}}\widebar{\fcs}^{(2)}_m\big(\Balpha_m(u)\big), u\in\zerun\right).
\end{align}
Indeed, for the example of the first coordinate, \[\frac{m}{\Balpha_m^{(1)}}\widebar{\fcs}^{(1)}_m\big(\Balpha_m(u)\big)=\frac{1}{\Balpha_m^{(1)}}\sum_{i=1}^m \Bzeta_i\mathbb{1}_{\frac{\Balpha_m^{(1)}\Bxi_i}{\Balpha_m^{(2)}\Bzeta_i}\leq \tnorm{u}}\]
which takes into account the dilatation of the vectors composing the boundary of the convex chain, as well as the normalization of the slope that is affected by an affine dilatation of $\lt$.

\paragraph{Decomposition of $\fys_m$ according to the exponential model.}
Let us now decompose the process $\fys_m$ in several processes easier to manipulate. Let $u\in\zerun,$ we have for the $p^{th}$ coordinate, $p\in\{1,2\}$:
\begin{multline}\label{eq:decom}
	\fys^{(p)}_m(u)=\sqrt{m}\Bigg[\frac{m}{\Balpha_m^{(1)}}\bigg(\widebar{\fcs}^{(p)}_m\big(\Balpha_m(u)\big)-\fcs_\infty^{(p)}\big(\Balpha_m(u)\big)\bigg)+\frac{m}{\Balpha_m^{(1)}}\bigg(\fcs_\infty^{(p)}\big(\Balpha_m(u)\big)-\fcs_\infty^{(p)}\left(u\right)\bigg)\\
	+\left(\frac{m}{\Balpha_m^{(1)}}-1\right)\fcs_\infty^{(p)}(u)\Bigg]
\end{multline} 

Consider $\mathbf{g}_1,\mathbf{g}_2$ the two independent standard Gaussian variables of \Cref{lem:tcl}, and let us handle the last two terms of \eqref{eq:decom}. Notice first that
\[\frac{\fcs_\infty^{(1)}\bigg(\frac{\Balpha_m^{(2)}}{\Balpha_m^{(1)}}u\bigg)-\fcs_\infty^{(1)}\left(u\right)}{\bigg(\frac{\Balpha_m^{(2)}}{\Balpha_m^{(1)}}-1\bigg)}\xrightarrow[m]{(d)} \frac2{\pi}\frac{\tnorm{u}}{1+\tnorm{u}^2}\cdot x'(u)\]
where we have used $\frac{\Balpha_m^{(2)}}{\Balpha_m^{(1)}}\dd 1$ by \eqref{eq:qhjysfhgwf}.
This means that with \Cref{lem:tcl} we have 
\begin{multline}\label{eq:g1g2}
	\sqrt{m}\Bigg[\frac{m}{\Balpha_m^{(1)}}\bigg(\fcs_\infty^{(1)}\Big(\frac{\Balpha_m^{(2)}}{\Balpha_m^{(1)}}u\Big)-\fcs_\infty^{(1)}\left(u\right)\bigg)
	+\left(\frac{m}{\Balpha_m^{(1)}}-1\right)\fcs_\infty^{(1)}(u)\Bigg]\xrightarrow[m]{(d)} \left(\mathbf{g}_1-\mathbf{g}_2\right)\cdot \frac2{\pi}\frac{\tnorm{u}}{1+\tnorm{u}^2}\cdot x'(u)\\+\mathbf{g}_1\cdot x(u).
\end{multline} 
For the second coordinate, we get a similar convergence of the two last terms:
\begin{multline}\label{eq:g1g2bis}
	\sqrt{m}\Bigg[\frac{m}{\Balpha_m^{(1)}}\bigg(\fcs_\infty^{(2)}\Big(\frac{\Balpha_m^{(2)}}{\Balpha_m^{(1)}}u\Big)-\fcs_\infty^{(2)}\left(u\right)\bigg)
	+\left(\frac{m}{\Balpha_m^{(1)}}-1\right)\fcs_\infty^{(2)}(u)\Bigg]\xrightarrow[m]{(d)}\left(\mathbf{g}_1-\mathbf{g}_2\right)\cdot \frac2{\pi}\frac{\tnorm{u}}{1+\tnorm{u}^2}\cdot y'(u)\\+\mathbf{g}_1\cdot y(u),
\end{multline} 
where the convergence of \eqref{eq:g1g2} and \eqref{eq:g1g2bis} has to be thought as a joint convergence including the same Gaussian standard r.v. $\mathbf{g}_1,\mathbf{g}_2$. This allows us to recover the last two processes anounced in \Cref{thm:globfluc}. For the first one, corresponding to the first term of the decomposition of \eqref{eq:decom}, we first need to prove an intermediary Lemma:
\begin{Lemma}\label{lem:ezcom}
	Suppose that we have the following convergence in distribution 
	\begin{align}\label{eq:rsgsg}
		\sqrt{m}\left[\widebar{\fcs}_m-\fcs_\infty\right]\xrightarrow[m]{(d)}\widebar{\fys},
	\end{align}
	in $D(\zerun)^2$ %for the topology of uniform convergence on all compact sets
	for some process $\widebar{\fys}$. Then we have for all $u\in\zerun,$
	\begin{align}
		\sqrt{m}\left[\widebar{\fcs}_m(u)-\fcs_\infty(u),\frac{m}{\Balpha_m^{(1)}}\left(\widebar{\fcs}_m\big(\Balpha_m(u)\big)-\fcs_\infty\big(\Balpha_m(u)\big)\right)\right]\xrightarrow[m]{(d)}\left[\widebar{\fys}(u),\widebar{\fys}(u)\right].
	\end{align}
	In words, the limiting processes of the two processes of the left-hand side are equal.
\end{Lemma}
\begin{proof}
	By the strong law of large number we have $\frac{m}{\Balpha_m^{(1)}}\as 1$, as well as $\frac{\Balpha_m^{(2)}}{\Balpha_m^{(1)}}$ and thus $\Balpha_m(u)\as u$ for all $u\in\zerun$. By \eqref{eq:rsgsg}, the sequence of processes $\left(\sqrt{m}\left(\widebar{\fcs}_m-\fcs_\infty\right)\right)$ is tight in $D(\zerun^2)$. The map $F:u\mapsto \tnorm{u}$ is a continuous nondecreasing surjective map from $\zerun$ to $\entty$ (where $\entty$ is seen as a compact set). Let $(\mathsf{X}_n)$ be a sequence of processes (with values in $\mathbb{R}$) that converges in distribution to $\mathsf{X}$ in $D(\entty)$ (where for all $n\in\NN$, $\lim_{t\to+\infty}\mathsf{X}_n(t)$ is finite, as well as $\lim_{t\to+\infty}\mathsf{X}(t)$). This implies that $\mathsf{X}_n\circ F$ converges in distribution to $\mathsf{X}\circ F$ in $D(\zerun)$.
	We claim that if $(a_n,b_n)\as 1$ then
	\[\big(a_n \mathsf{X}_n(b_n F(u)),u\in\zerun\big)\dd\big(\mathsf{X}(F(u)),u\in\zerun\big)\]
	for the same topology. A proof runs as follows: by the Skorokhod representation theorem, there exists a probability space in which are defined altogether some copies $(\widebar{a}_n,\widebar{b}_n,\widebar{\mathsf{X}}_n)$ of $({a_n},{b_n},{\mathsf{X}_n})$ (and $\bar{\mathsf{X}}$ a copy of $\mathsf{X}$) and such that $(\widebar{a}_n,\widebar{b}_n,\widebar{\mathsf{X}}_n\circ F)\as (a,b,\widebar{\mathsf{X}}\circ F)$. 
	
	Set $\lambda_n:u\mapsto \frac{2}{\pi}\arctan\left(\tnorm{u}/b_n\right)$, which is a sequence of striclty increasing continuous functions mapping $[0,1]$ to itself, and such that $\lambda_n(0)=0$ and $\lambda_n(1)=1$. In particular, $\widebar{\mathsf{X}}_n(b_n F(\lambda_n(u)))=\widebar{\mathsf{X}}_n(F(u))$. Hence, $\sup_u \abso{\widebar{\mathsf{X}}_n(b_n F(\lambda_n(u)))-\widebar{\mathsf{X}}(F(u))}\to 0 $ since $(\widebar{\mathsf{X}}_n\circ F)$ converges to $(\widebar{\mathsf{X}}\circ F)$ in $D(\zerun)$ (and this, $\omega$ by $\omega$). Furthermore, $\lambda_n(u)\to u $ uniformly in $\zerun$, and then, according to Billingsley \cite[p.124]{billingsley1999convergence}, we may conclude that the claim holds true.
\end{proof}

\begin{Notation}\label{not:yb}
	In order to prove \Cref{lem:cvYm} (and thus \Cref{thm:globfluc}), it remains to prove \eqref{eq:rsgsg}. We set 
	$\widebar{\fys}_m:=\sqrt{m}\left(\widebar{\fcs}_m-\fcs_\infty\right)$.
\end{Notation}

\begin{Lemma}\label{lem:process}
	The process $\widebar{\fys}_m$ converges in distribution to $\widebar{\fys}$ in $D(\zerun,\RR^2)$, where the process $\widebar{\fys}$ is the process introduced in \Cref{lem:cvYm}.
\end{Lemma}
The proof of this Lemma is in two steps, including first the convergence of the Finite Dimensional Distributions (FDD) and the tightness of the process. Therefore, we need a suitable parametrization to complete this proof.

\paragraph{Parametrization in the exponential model.} Fix some $k\geq1$, and $(0=u_0<u_1<\ldots<u_{k-1}<u_k=1).$ We will prove the convergence of $\fys_m(u_i)$ for all $i\in\{1,\ldots,k\}$ (recall that $\fys_m(u_0)=0$ a.s.). Fix for a moment $i\in\{1,\ldots,k\}.$ The random variable \[\mathbf{n}(u_i)=\mathbf{n}_i:=\abso{\left\{j;\tnorm{u_{i-1}}< \frac{\mathbf{\Bxi}_j}{\mathbf{\Bzeta}_j}\leq \tnorm{u_i}\right\}},\] counts the number of vectors whose slope is in the interval $\left(\tnorm{u_{i-1}},\tnorm{u_i}\right].$ We denote by $\left(\mathbf{w}^{(i)}_1,\ldots,\mathbf{w}^{(i)}_{\mathbf{n}_i}\right)$ the sequence of these vectors taken in their initial order.
We have
\begin{align}
	\widebar{\fcs}_m(u_i)=\frac{1}{m}\sum_{s=1}^i\sum_{j=1}^{\mathbf{n}_s}\mathbf{w}_j^{(s)},
\end{align}
since $\widebar{\fcs}_m(u_i)$ is obtained by taking the sum of vectors with slopes smaller or equal than $\tnorm{u_i}.$ For this $i$, the variables $\mathbf{w}^{(i)}_j$ are independent and distributed as the law of a pair $(\Bzeta,\Bxi)$ conditioned on $\left\{\tnorm{u_{i-1}}< \frac{\Bxi}{\Bzeta}\leq \tnorm{u_i} \right\}$, where $\Bzeta,\Bxi$ are independent exponential variables with mean $1$. Denote $\mathbf{w}^{(i)}=(\Bzeta^{(i)},\Bxi^{(i)})$ a generic random value with this conditional law. Note that the vectors of the sequence $\left(\mathbf{w}^{(i)}_1,\ldots,\mathbf{w}^{(i)}_{\mathbf{n}_i}\right)$ have common distribution with $\mathbf{w}^{(i)}.$
\par By setting the nondecreasing mapping 
\begin{align}\label{eq:defq}
	q:u\in\zerun\mapsto \mathbb{P}\left(\frac{\Bxi}{\Bzeta}\leq \tnorm{u}\right)=1-\frac{1}{1+\tnorm{u}},
\end{align}we may then set \[p_i:=q(u_{i})-q(u_{i-1}),\] so that the tuple $\mathbf{n}[k]$ has a multinomial distribution $\mathcal{M}(m,p[k]).$

\begin{Proposition}\label{prop-1}
	Recall the mappings introduced from \eqref{eq:super} to \eqref{eq:super2}. The following properties of $\Bzeta^{(i)}$and  $\Bxi^{(i)}$ hold :
	\begin{enumerate}
		\item $p_i=f(u_{i-1},u_i).$
		\item $\mathbb{E}\left[\Bzeta^{(i)}\right]=e_1(u_{i-1},u_i)/p_i,\quad\text{ and }\quad\mathbb{E}\left[\Bxi^{(i)}\right]=e_2(u_{i-1},u_i)/p_i.$
		\item $\mathbb{V}\left[\Bzeta^{(i)}\right]=v_1(u_{i-1},u_i)/p_i,\quad\text{ and }\quad\mathbb{V}\left[\Bxi^{(i)}\right]=v_2(u_{i-1},u_i)/p_i.$
		\item For all $s\in\NN$, $\mathbb{E}\left[\abso{\Bzeta^{(i)}}^s\right]\leq s!\quad\text{ and }\quad\mathbb{E}\left[\abso{\Bxi^{(i)}}^s\right]\leq s!$
	\end{enumerate}
\end{Proposition}
\begin{proof}
	The three first points come from standard integral computations. As an illustrative example let us compute $\mathbb{E}\left[\Bxi^{(i)}\right]$. We have:
	\begin{align}\nonumber
		\mathbb{E}\left[\Bxi^{(i)}\right]&=\frac1{\mathbb{P}\left(\tnorm{u_{i-1}}< \frac{\Bxi}{\Bzeta}\leq \tnorm{u_i}\right)}\int_{0}^{+\infty}\int_{0}^{+\infty}ye^{-x}e^{-y}\mathbb{1}_{\tnorm{u_{i-1}}<\frac{y}{x}\leq\tnorm{u_i}}\mathrm{d}x\mathrm{d}y\\\nonumber
		&=\frac1{p_i}\int_{0}^{+\infty}ye^{-y}\int_{\frac{y}{\tnorm{u_i}}}^{\frac{y}{\tnorm{u_{i-1}}}}e^{-x}\mathrm{d}x\mathrm{d}y\\\nonumber
		&=\frac1{p_i}\int_{0}^{+\infty}ye^{-(1+\frac{1}{\tnorm{u_i}})y}\mathrm{d}y-\int_{0}^{+\infty}ye^{-(1+\frac{1}{\tnorm{u_{i-1}}})y}\mathrm{d}y\\
		&=\frac1{p_i}e_2(u_{i-1},u_i).\label{eq:super3}
	\end{align}
	The fourth one comes from the fact the $s^{th}$ moment of an exponential r.v. of mean $1$ equals $s!$
\end{proof}
The following proposition establishes the link between our parametrization and the limit parabola.
\begin{Proposition}
	For all $u\in\zerun,$ we have for all $m\geq1$,
	\begin{align}\label{equ-4}
		\mathbb{E}\left[\widebar{\fcs}_m(u)\right]=\fcs_\infty(u),
	\end{align}
	where $\fcs_\infty(u)$ is given in \eqref{eq:baranylimit}.
\end{Proposition}
\begin{proof}
	We have $\widebar{\fcs}_m(u)=\frac{1}{m}\sum_{j=1}^m \left(\Bzeta_j,\Bxi_j\right)\mathbb{1}_{\frac{\Bxi_j}{\Bzeta_j}\leq \tnorm{u}}$. By linearity of the expectation, it suffices to prove that $\fcs_\infty(u)=\mathbb{E}\left[\left(\Bzeta,\Bxi\right)\mathbb{1}_{\frac{\Bxi}{\Bzeta}\leq \tnorm{u}}\right].$
	Standard integral computations very similar to \eqref{eq:super3}, allow to compute $\mathbb{E}\left[\Bzeta_j\mathbb{1}_{\frac{\Bxi_j}{\Bzeta_j}\leq \tnorm{u}}\right]=x(u)$ and 
	$\mathbb{E}\left[\Bxi_j\mathbb{1}_{\frac{\Bxi_j}{\Bzeta_j}\leq \tnorm{u}}\right]=y(u)$ for all $j\in\ent{1}{m}$, which is \eqref{eq:baranylimit}.
	It also follows that
	\begin{align}\label{eq:cfgsfgh}
		\fcs_\infty(u_i)&=\sum_{s=1}^i \mathbb{E}\left[\left(\Bzeta,\Bxi\right)\mathbb{1}_{\tnorm{u_{s-1}}< \frac{\Bxi}{\Bzeta}\leq \tnorm{u_s}}\right]
		=\sum_{s=1}^i p_s\mathbb{E}\left[\mathbf{w}^{(s)}\right].
	\end{align} 
\end{proof}

%Note that for a pair $(\Bzeta,\Bxi)$ of independent exponential variables of rate $1,$ we may actually write for all $u\in\zerun$, $\fcs_\infty(u)=\mathbb{E}\left[\left(\Bzeta,\Bxi\right)\mathbb{1}_{\frac{\Bxi}{\Bzeta}\leq \tnorm{u}}\right].$
%
%This formula allows to identify  $(	\fcs_\infty(u_i),1\leq i\leq k)$: using that
%\begin{align}
%	\left(\Bzeta,\Bxi\right)\sur{=}{(d)}\sum_{s=1}^k \left(\Bzeta,\Bxi\right)\mathbb{1}_{\tnorm{u_{s-1}}< \frac{\Bxi}{\Bzeta}\leq \tnorm{u_s}},
%\end{align} it comes for all $i\in\{1,\ldots,k\}$ :

We may now come back to the %on the proof of \Cref{lem:process}.
\begin{proof}[Proof of \Cref{lem:process}]
	We start by proving the convergence of the FDD of $\widebar{\fys}_m$ to those of $\widebar{\fys}.$
	\paragraph{FDD.}
	Let us split $\widebar{\fys}_m(u_i)$ in order to make more visible the convergence result we want to prove. For $i\in\{1,\ldots,k\},$ we have :
	\begin{align}\label{equa1}
		\widebar{\fys}_m(u_i)=\frac{1}{\sqrt{m}}\left(\sum_{s=1}^i\sum_{j=1}^{\mathbf{n}_s}\left(\mathbf{w}_j^{(s)}-\mathbb{E}\left[\mathbf{w}_j^{(s)}\right]\right)\right)+\left(\frac{1}{\sqrt{m}}\sum_{s=1}^i\mathbf{n}_s\mathbb{E}\left[\mathbf{w}^{(s)}\right]-\sqrt{m}\fcs_\infty(u_i)\right)
	\end{align}
	We decompose the process as suggested by (\ref{equa1}):
	\begin{align}\label{eq:rfetjf}
		\widebar{\fys}_m(u_i)=\mathsf{A}_m(u_i)+\mathsf{B}_m(u_i),
	\end{align}
	where $\mathsf{A}_m(u_i)$ is the first contribution of the right-hand side and the second one is rewritten
	\begin{align}\label{eq:defAB}
		\mathsf{B}_m(u_i) &= \sum_{s=1}^i \left(\frac{\mathbf{n}_s-mp_s}{\sqrt{m}}\right)\mathbb{E}\left[\mathbf{w}^{(s)}\right].
	\end{align}
	
	Using $\mathbf{n}[k]\sim \mathcal{M}(m,p[k])$ and \eqref{eq:cfgsfgh}, a standard consequence of the central limit theorem is that
	\begin{align}
		\left(\frac{\mathbf{n}_s-mp_s}{\sqrt{m}},s\in\{1,\ldots,k\}\right)\sur{\sous{\longrightarrow}{m\to\infty}}{(d)}\left(\mathbf{b}_s,s\in\{1,\ldots,k\}\right)
	\end{align}
	where $\left(\mathbf{b}_s,s\in\{1,\ldots,k\}\right)$ is a centered Gaussian vector with covariance function
	\[\cov(\mathbf{b}_k,\mathbf{b}_\ell)=-p_kp_\ell+p_k\mathbb{1}_{k=\ell}.\]
	Using a concentration result of $\mathbf{n}_s$ around $mp_s$ (for example the Hoeffding inequalities), by the central limit theorem,
	\begin{align}
		\left(\sum_{j=1}^{\mathbf{n}_s}\frac{\mathbf{w}_j^{(s)}-\mathbb{E}\left[\mathbf{w}_j^{(s)}\right]}{\sqrt{m}},s\in\{1,\ldots,k\}\right)\sur{\sous{\longrightarrow}{m\to\infty}}{(d)}\left(\sqrt{p_s}\mathbf{G}_s,s\in\{1,\ldots,k\}\right)
	\end{align}
	%https://math.stackexchange.com/questions/65318/random-index-central-limit-theorem
	where for all $s\in\{1,\ldots,k\}$, the random variable $\mathbf{G}_s:=(\mathbf{G}_s^{(1)},\mathbf{G}_s^{(2)})$ is such that the r.v. $\mathbf{G}_s^{(1)}$ and $\mathbf{G}_s^{(2)}$ are independent, and $\mathbf{G}_s^{(1)}$ (resp. $\mathbf{G}_s^{(2)}$) is a centered Gaussian r.v. with variance $\mathbb{V}\left[\Bzeta^{(s)}\right]$ (resp. $\mathbb{V}\left[\Bxi^{(s)}\right]$), these variables being independent of $\mathbf{b}_s$. These considerations allow also to prove that the families of r.v $(\mathbf{G}_s^{(p)},s\in\{1,\ldots,k\})$ and $(\mathbf{b}_s,s\in\{1,\ldots,k\})$ are independent for all $p\in\{1,2\}$.
	
	This proves that the FDD of $\widebar{\fys}_m$ converge to those of $\widebar{\fys}$, where
	\begin{align}
		\widebar{\fys}(u_i):=\sum_{s=1}^i\left( \sqrt{p_s}\mathbf{G}_s + \mathbf{b}_s\mathbb{E}\left[\left(\Bxi^{(s)},\Bzeta^{(s)}\right)\right]\right),i\in\{1,\ldots,k\}.
	\end{align}

	\paragraph{Tightness.}
	Proving the tightness of the sequence $(\widebar{\fys}_m)_{m\geq0}$ in $D(\zerun)$ is the tough part of this section. The key point is the following Lemma:
	\begin{Lemma}\label{lem:decom}
		Let $(\mathsf{X}_m)_{m\geq0}$ be a sequence of processes taking their values in $D(\zerun)$.
		Assume that for any $m$, $\mathsf{X}_m=\mathsf{X^{(1)}_m}+\mathsf{X^{(2)}_m}$ where $\mathsf{X^{(1)}_m}$ is a continuous process, and
		$\mathsf{X^{(2)}_m}$ is a càdlàg process. If $\mathsf{X^{(1)}_m}$ converges in distribution to $\mathsf{X^{(1)}}$ in $C(\zerun)$, and if
		$\sup\abso{\mathsf{X^{(2)}_m}}\xrightarrow[m]{(d)}0$ then $(\mathsf{X}_m)$ converges in distribution to $\mathsf{X^{(2)}}$ in $D(\zerun)$.
	\end{Lemma}
	From this point, the proof runs in several parts: 
	\begin{enumerate}
		\item We decompose the process $\widebar{\fys}_m=\widebar{\mathcal{Y}}_m+\left(\widebar{\fys}_m-\widebar{\mathcal{Y}}_m\right)$ where $\widebar{\mathcal{Y}}_m$ is a continuous process and $\widebar{\fys}_m-\widebar{\mathcal{Y}}_m$ is a càdlàg process (on $\zerun$). We want to apply \Cref{lem:decom} to $\mathsf{X^{(1)}_m}=\widebar{\mathcal{Y}}_m$ and $\mathsf{X^{(2)}_m}=\widebar{\fys}_m-\widebar{\mathcal{Y}}_m$.
		\item We prove the convergence in distribution of the sequence $(\widebar{\mathcal{Y}}_m)$ to $\widebar{\fys}$ in $C(\zerun)$.
		%		 Note that the processes  $(\widebar{\mathcal{Y}}_m)$ will be built in such a way that its FDD are the same as these of $(\widebar{\fys}_m)$. 
		%		It will remain to prove only the tightness of $(\widebar{\mathcal{Y}}_m)$. To do so, $(\widebar{\mathcal{Y}}_m)$ will be decomposed as $\widebar{\mathcal{Y}}_m=\mathcal{A}_m+\mathcal{B}_m$, where $\mathcal{A}_m,\mathcal{B}_m$ are two continuous processes whose tightness can be proven separately. The criterion we will use is stated in \Cref{lem:billin}.
		\item We prove that $\sup_{t\in\zerun}\abso{\widebar{\fys}_m-\widebar{\mathcal{Y}}_m}\xrightarrow[m]{(d)}0$ to conclude that $(\widebar{\fys}_m)$ converges in distribution to the same limit $\widebar{\fys}$ as this of $(\widebar{\mathcal{Y}}_m)$.
	\end{enumerate}
	
	{\bf Step 1:}
	Let us define the process $\widebar{\mathcal{Y}}_m$ properly.
	Recall the mapping $q$ introduced in \eqref{eq:defq}, and set for all $j\in\{0,\ldots,m\}$ the "$j^{th}$ $m$-tile"
	\begin{align}
		v_j(m)=\inf\left\{u;q(u)\geq \frac{j}{m}\right\},
	\end{align}
	as well as the interval
	\begin{align}
		I_j(m)=[v_j(m),v_{j+1}(m)).
	\end{align}
	Recall the definition of $\mathsf{A}_m,{\mathsf{B}}_m$ in \eqref{eq:defAB}, which satisfied $\widebar{\fys}_m(u_i)=\mathsf{A}_m(u_i)+{\mathsf{B}}_m(u_i)$ for any generic point $u_i\in\zerun$.
	We define $\widebar{\mathcal{Y}}_m$ as
	\begin{align}\label{eq:continY}
		\widebar{\mathcal{Y}}_m={\mathcal{A}}_m+{\mathcal{B}}_m,
	\end{align}
	where for all $j\in\{0,\ldots,m\}$,
	\begin{align}
		\mathcal{A}_m(v_j(m)) &= \mathsf{A}_m(v_j(m)),\\
		\mathcal{B}_m(v_j(m)) &= \mathsf{B}_m(v_j(m)).
	\end{align}
	and $\mathcal{A}_m,\mathcal{B}_m$ are interpolated between the points $v_j(m)$, in order to embed them in the space of continuous processes $C(\zerun)$, and thus, so is $\widebar{\mathcal{Y}}_m$. We replaced the generic points $u_i$ by the $m$-tiles $v_j(m)$ that are more suitable to prove the tightness of $\mathcal{A}_m,\mathcal{B}_m$ (and thus, this of $\widebar{\mathcal{Y}}_m$).\\
	%	Hence, for all $m\geq0$ the process $\widebar{\fys}_m$ can be decomposed as follows
	%	\begin{align}
		%		\widebar{\fys}_m=\widebar{\mathcal{Y}}_m+(\widebar{\fys}_m-\widebar{\mathcal{Y}}_m)
		%	\end{align}
	%	where $\widebar{\mathcal{Y}}_m$ was  built to be a process in $C(\zerun,\RR)$.\\

	{\bf Step 2:}
	By construction, the FDD of $\widebar{\mathcal{Y}}_m$ are the same as these of $\widebar{\fys}_m$ on the $m$-tiles $v_j(m),j\in\ent{1}{m}$ so that it remains to prove only the tightness of $(\widebar{\mathcal{Y}}_m)$ to prove its convergence in distribution towards $\widebar{\fys}$. And by \Cref{lem:tightSlutsky} above (\cite{marckert2008}, Lemma 8.), it suffices to prove the tightnesses of $\mathcal{A}_m,\mathcal{B}_m$ separately.\footnote{In fact this establishes the convergence of the FDD in the sense that $\left(\widebar{\mathcal{Y}}_m\left(\frac{\lfloor mu_i\rfloor}{m}\right),1\leq i\leq k\right)$ converges, that is, at the districretization points. The tightness allows to see that $\sup \abso{\widebar{\mathcal{Y}}_m\left(\frac{\lfloor mu_i\rfloor}{m}\right)-\widebar{\mathcal{Y}}_m\left(u_i\right)}\overset{\text{{\tiny(proba.)}}}{\longrightarrow} 0$ (the argument is direct from \eqref{newswag4}).}
	
	\begin{Lemma}\label{lem:tightSlutsky}
		Let $\left(\mathsf{Z}^{(1)}_m,\mathsf{Z}^{(2)}_m\right)_{m\geq0}$ be a sequence of pairs of processes in $C(\zerun)^2$ . The
		tightnesses of both families $\left(\mathsf{Z}^{(1)}_m\right)_{m\geq0}$ and $\left(\mathsf{Z}^{(2)}_m\right)_{m\geq0}$ in $C(\zerun)$ imply that of $\left(\mathsf{Z}^{(1)}_m,\mathsf{Z}^{(2)}_m\right)_{m\geq0}$ in $C(\zerun)^2$.
	\end{Lemma}
	
	A criterion of tightness in $C(\zerun)$ is the following \cite[Theorem II.12.3]{billingsley1999convergence}:
	\begin{Lemma}\label{lem:billin}
		Let $(\mathsf{X}_m)_{m\geq0}$ be a sequence of stochastic processes in $C\left(\zerun,\RR\right)$. If there exist some positive numbers $\alpha>1,\beta\geq 0$ and a nondecreasing continuous function $F$ on $\zerun$ such that for all $m\in\NN$ and $(s,t) \in\zerun^2$ with $0\leq s\leq t\leq 1$,
		\begin{align}
			\mathbb{P}\left[\left\vert \mathsf{X}_m(t)-\mathsf{X}_m(s)\right\vert\geq \lambda\right]\leq \frac{1}{\lambda^{\beta}}\abso{F(t)-F(r)}^{\alpha},
		\end{align}
		for all $\lambda>0,$ then $(\mathsf{X}_m)_{m\geq0}$ is tight in $C\left(\zerun\right)$.
	\end{Lemma}
	
	Note that by Markov's inequality, it suffices to prove that $\mathbb{E}\left[\left\vert \mathsf{X}_m(t)-\mathsf{X}_m(s)\right\vert^{\beta}\right]\leq \abso{F(t)-F(r)}^{\alpha}$.
	
	We are working in $\RR^2$ so we will apply this criterion twice, one on every coordinate $\widebar{\mathcal{Y}}_m^{(p)},p\in\{1,2\}$. We shall prove that for both processes, $\beta=4,\alpha=2$ do the job, as well as the continuous nondecreasing map $F=q$ introduced in \eqref{eq:defq},	where we extended $F(1)=q(1)=1$.
	
	Let $(s,t) \in\zerun^2$ such that $0\leq s< t\leq 1$. There exists $j_1<j_2\in\{0,\ldots,m\}$ such that $s\in I_{j_1}(m)$ and $t\in I_{j_2}(m)$ and we assume that $j_1\neq j_2$ (the case $j_1=j_2$ will be treated afterwards). We may write for all $p\in\{1,2\},$
	\begin{multline}\label{newswag3}
		\mathbb{E}\left[\left\vert \mathcal{B}^{(p)}_m(t)-\mathcal{B}^{(p)}_m(s)\right\vert^{4}\right]\leq \Cst\bigg(\mathbb{E}\left[\left\vert \mathcal{B}^{(p)}_m(t)-\mathcal{B}^{(p)}_m(v_{j_2}(m))\right\vert^{4}\right]
		+\mathbb{E}\left[\left\vert \mathcal{B}^{(p)}_m(v_{j_2}(m))-\mathcal{B}^{(p)}_m(v_{j_1+1}(m))\right\vert^{4}\right]\\+\mathbb{E}\left[\left\vert \mathcal{B}^{(p)}_m(v_{j_1+1}(m))-\mathcal{B}^{(p)}_m(s)\right\vert^{4}\right]\bigg).
	\end{multline}
	
	\begin{Notation}
		For the sequel, we introduce the r.v. $\mathsf{N}_{[s,t]}$ for $(s\leq t)\in\zerun^2$, which is binomial $\left(m,q(t)-q(s)\right)$ distributed. If $j\in\{0,\ldots,m\}$ and $s=v_{j}(m)$ (or $t$), we will abuse the notation $\mathsf{N}_{[j,t]}=\mathsf{N}_{[v_{j}(m),t]}$. Notice that $q(v_j(m))=\frac{j}{m}$, so that $\mathsf{N}_{[j,j+1]}$ is binomial $(m,\frac{1}{m})$ distributed. 
	\end{Notation}
	Recall \eqref{eq:defAB}. We have the following equalities in distribution
	\begin{align}
		\mathcal{B}^{(1)}_m(v_{j_2}(m))-\mathcal{B}^{(1)}_m(v_{j_1+1}(m))&\eqd \frac{\mathsf{N}_{[j_1+1,j_2]}-(j_2-j_1-1)}{\sqrt{m}}\mathbb{E}[\Bzeta^{(j_1,j_2)}],\\
		\mathcal{B}^{(2)}_m(v_{j_2}(m))-\mathcal{B}^{(2)}_m(v_{j_1+1}(m))&\eqd \frac{\mathsf{N}_{[j_1+1,j_2]}-(j_2-j_1-1)}{\sqrt{m}}\mathbb{E}[\Bxi^{(j_1,j_2)}].
	\end{align}
	where $\mathbf{w}^{(j_1,j_2)}=(\Bzeta^{(j_1,j_2)},\Bxi^{(j_1,j_2)})$ is a random variable distributed as the law of a pair $(\Bzeta,\Bxi)$ conditioned on $\left\{\tnorm{v_{j_1+1}(m)}< \frac{\Bxi}{\Bzeta}\leq \tnorm{v_{j_2}(m)} \right\}$, where $\Bzeta,\Bxi$ are independent exponential variables with mean $1$. 
	
	Note that for any r.v. $\mathbf{x}$ binomial $(m,p)$ distributed,
	\begin{align}
		\mathbb{E}\left[ \abso{\mathbf{x}-mp}^4\right]=mp(1-p)+(3m-6)mp^2(1-p)^2&\leq mp+3m^2p^2.
	\end{align} Since $mp\leq m^2p^2$, for $p\geq1/m$ (and here $q(v_{j_2}(m))-q(v_{j_1+1}(m))\geq 1/m$),  we obtain
	\begin{align}\label{newswag}
		\mathbb{E}\left[\left\vert \mathcal{B}^{(p)}_m(v_{j_2}(m))-\mathcal{B}^{(p)}_m(v_{j_1+1}(m))\right\vert^{4}\right]
		%		&\leq 2^4\cdot \left(3\left(\frac{j_2-(j_1+1)}{m}\right)^2+\frac{j_2-(j_1+1)}{m^2}\right)\\
		%		&\leq 2^6\left(\frac{j_2-(j_1+1)}{m}\right)^2\\
		&=\Cste\abso{q(v_{j_2}(m))-q(v_{j_1+1}(m))}^2,
	\end{align}
	where we have used \Cref{prop-1} to bound $\mathbb{E}[\big(\Bxi^{(j_1,j_2)}\big)^4]\leq 4!$.
	
	Now, by definition of $\mathcal{B}_m$,
	\begin{align}\label{eq:kdfjgndfkvnsz}
		\mathcal{B}_m(t)-\mathcal{B}_m(v_{j_2}(m))&=\left[\mathcal{B}_m(v_{j_2+1}(m))-\mathcal{B}_m(v_{j_2}(m))\right]\cdot\frac{t-v_{j_2}(m)}{v_{j_2+1}(m)-v_{j_2}(m)}
	\end{align}
	and since 
	\[\frac{t-v_{j_2}(m)}{v_{j_2+1}(m)-v_{j_2}(m)}\leq \frac{q(t)-q(v_{j_2}(m))}{q(v_{j_2+1}(m))-q(v_{j_2}(m))}\]
	because $t\mapsto q(t)$ is convex on $\zerun,$ using also \eqref{newswag}, we reach
	\begin{align}\label{newswag2}
		\mathbb{E}\left[\left\vert \mathcal{B}^{(p)}_m(t)-\mathcal{B}^{(p)}_m(v_{j_2}(m))\right\vert^{4}\right]
		&\leq \Cste\abso{q(t)-q(v_{j_2}(m))}^2.
	\end{align}
	In the end, using \eqref{newswag} and \eqref{newswag2} twice in \eqref{newswag3}, we obtain
	\begin{align}\label{newswag4}
		\mathbb{E}\left[\left\vert \mathcal{B}^{(p)}_m(t)-\mathcal{B}^{(p)}_m(s)\right\vert^{4}\right]\leq\Cste\abso{q(t)-q(s)}^2.
	\end{align}
	
	We may now work on the sequence $(\mathcal{A}_m)_{m\geq0}$, so fix again $(s,t) \in\zerun^2$ such that $0\leq s< t\leq 1$. Recall \eqref{eq:defAB} and let $\mathbf{w}_1^{(j_1,j_2)},\ldots,\mathbf{w}_{\mathsf{N}_{[j_1+1,j_2]}}^{(j_1,j_2)}$ be a sequence of i.i.d r.v. having the same distribution as that of $\mathbf{w}^{(j_1,j_2)}$. We have the following equality in distribution
	\begin{align}
		\mathcal{A}_m(v_{j_2}(m))-\mathcal{A}_m(v_{j_1+1}(m))&\eqd \sum_{i=1}^{\mathsf{N}_{[j_1+1,j_2]}}\frac{\mathbf{w}_1^{(j_1,j_2)}-\mathbb{E}[\mathbf{w}^{(j_1,j_2)}]}{\sqrt{m}}
	\end{align}
	so that \begin{align}
		\mathbb{E}\left[\left\vert \mathcal{A}_m(v_{j_2}(m))-\mathcal{A}_m(v_{j_1+1}(m))\right\vert^{4}\right]		&=\sum_{r=0}^{m}\frac{1}{m^2}\mathbb{E}\left[\left\vert\sum_{j=1}^{r}\mathbf{w}_1^{(j_1,j_2)}-\mathbb{E}[\mathbf{w}^{(j_1,j_2)}]\right\vert^4\right]\mathbb{P}\left(\mathsf{N}_{[j_1+1,j_2]}=r\right).\label{equ-5}
	\end{align}
	
	If $\mathbf{H}^{(j_1,j_2)}$ is a random variable having the law of $\mathbf{w}_1^{(j_1,j_2)}-\mathbb{E}[\mathbf{w}^{(j_1,j_2)}]$, then for all $r\in\ent{1}{m}$,
	\begin{align}\nonumber
		\frac 1{m^2}\mathbb{E}\left[\left\vert\sum_{j=1}^{r}\mathbf{w}_1^{(j_1,j_2)}-\mathbb{E}[\mathbf{w}^{(j_1,j_2)}]\right\vert^4\right]	\leq \frac{r\mathbb{E}\left[(\mathbf{H}^{(j_1,j_2)})^4\right]}{m^2}+\frac{r^2\mathbb{E}\left[(\mathbf{H}^{(j_1,j_2)})^2\right]}{2m^2}.
	\end{align}
	Therefore,
	\begin{align}
		\mathbb{E}\left[\left\vert \mathcal{A}_m(v_{j_2}(m))-\mathcal{A}_m(v_{j_1+1}(m))\right\vert^{4}\right]&\leq \frac{\mathbb{E}\left[\mathsf{N}_{[j_1+1,j_2]}\right]\mathbb{E}\left[(\mathbf{H}^{(j_1,j_2)})^4\right]}{m^2}+\frac{\mathbb{E}\left[\mathsf{N}_{[j_1+1,j_2]}^2\right]\mathbb{E}\left[(\mathbf{H}^{(j_1,j_2)})^2\right]}{2m^2}.
	\end{align}
	In the same spirit as point $4.$ of \Cref{prop-1}, we may see that for all $s\in\NN,$ we have $\mathbb{E}\left[(\mathbf{H}^{(j_1,j_2)})^s\right]\leq s!$. For any r.v. $\mathbf{x}$ binomial $(m,p)$ distributed, we have 
	\[\mathbb{E}\left[ \mathbf{x}\right]\leq m^2p^2\quad\text{ and }\quad\mathbb{E}\left[ \mathbf{x}^2\right]\leq 2m^2p^2\]
	as soon as $p\geq 1/m$ (and here $q(v_{j_2}(m))-q(v_{j_1+1}(m))\geq 1/m$), hence
	\begin{align}
		\mathbb{E}\left[ \abso{\mathcal{A}^{p}_m(v_{j_2}(m))-\mathcal{A}^{(p)}_m(v_{j_1+1}(m))}^4\right]
		&\leq\Cste\abso{q(v_{j_2}(m))-q(v_{j_1+1}(m))}^2.
	\end{align}
	We can treat the other terms $\mathcal{A}^{(p)}_m(t)-\mathcal{A}^{(p)}_m(v_{j_2}(m))$ and $\mathcal{A}^{(p)}_m(v_{j_1+1}(m))-\mathcal{A}^{(p)}_m(s)$ exactly like we did for the process $\mathcal{B}_m$ in \eqref{newswag2} and finally find that 
	\begin{align}\label{newswag5}
		\mathbb{E}\left[\left\vert \mathcal{A}^{(p)}_m(t)-\mathcal{A}^{(p)}_m(s)\right\vert^{4}\right]\leq \Cste \abso{q(t)-q(s)}^2.
	\end{align}
	
	In the case $j_1=j_2$, we may immediately write, in both cases $\mathcal{A}_m,\mathcal{B}_m$ (we take $\mathcal{B}_m$ here as an example) just like in \eqref{eq:kdfjgndfkvnsz}:
	\begin{align}
		\mathcal{B}_m(t)-\mathcal{B}_m(s)&=\left[\mathcal{B}_m(v_{j_2+1}(m))-\mathcal{B}_m(v_{j_2}(m))\right]\cdot\frac{t-s}{v_{j_2+1}(m)-v_{j_2}(m)}
	\end{align}
	so that
	\begin{align}\label{eq:ref1}
		\mathbb{E}\left[\left\vert \mathcal{B}^{(p)}_m(t)-\mathcal{B}^{(p)}_m(s)\right\vert^{4}\right]&\leq \Cste\abso{q(t)-q(s)}^2
	\end{align}
	for every $p\in\{1,2\}$, by the same arguments as before.
	This proves the tightnesses of the sequences of processes $(\mathcal{A}_m),(\mathcal{B}_m)$.\\
	
	{\bf Step 3:}
	In the next step, we prove that $\sup_{t\in\zerun}\abso{\widebar{\fys}^{(p)}_m(t)-\widebar{\mathcal{Y}}^{(p)}_m(t)}\xrightarrow[m]{(proba.)} 0,$ for every $p\in\{1,2\}$; for this, since
	\begin{align}\label{eq:unionbound}
		\sup_{t\in\zerun}\abso{\widebar{\fys}^{(p)}_m(t)-\widebar{\mathcal{Y}}^{(p)}_m(t)}= \sup_{j\in\ent{0}{m-1}}\sup_{t\in I_j(m)}\abso{\widebar{\fys}^{(p)}_m(t)-\widebar{\mathcal{Y}}^{(p)}_m(t)},
	\end{align}
	and for all $j\in\ent{0}{m-1}$, 
	\begin{align}
		\sup_{t\in I_j(m)}\abso{\widebar{\fys}^{(p)}_m(t)-\widebar{\mathcal{Y}}^{(p)}_m(t)}	&\leq\sup_{t\in I_j(m)}\abso{\widebar{\fys}^{(p)}_m(t)-\widebar{\fys}^{(p)}_m(v_j(m))}+\sup_{t\in I_j(m)}\abso{\widebar{\mathcal{Y}}^{(p)}_m(v_j(m))-\widebar{\mathcal{Y}}^{(p)}_m(t)},
	\end{align}
	the proof of this point consists in showing that both processes $(\widebar{\fys}^{(p)}_m),(\widebar{\mathcal{Y}}^{(p)}_m)$ admit fluctuations larger than $\varepsilon>0$ on an interval $I_j(m)$ with probability $o(1/m)$. Fix $j\in\ent{0}{m-1}$ for the rest of the proof. Since $(\widebar{\mathcal{Y}}^{(p)}_m)$ is linear on every interval $I_j(m),$ we have first,
	\begin{align}
		\sup_{t\in I_j(m)}\abso{\widebar{\mathcal{Y}}^{(p)}_m(v_j(m))-\widebar{\mathcal{Y}}^{(p)}_m(t)}\leq \abso{\widebar{\mathcal{Y}}^{(p)}_m(v_j(m))-\widebar{\mathcal{Y}}^{(p)}_m(v_{j+1}(m))},
	\end{align}
	and second for all $\varepsilon>0,$
	\begin{align}\label{eqbound1}\nonumber
		\mathbb{P}\left(\abso{\widebar{\mathcal{Y}}^{(p)}_m(v_j(m))-\widebar{\mathcal{Y}}^{(p)}_m(v_{j+1}(m))}\geq \varepsilon\right)&\leq\frac{ \mathbb{E}\left[\abso{\widebar{\mathcal{Y}}^{(p)}_m(v_j(m))-\widebar{\mathcal{Y}}^{(p)}_m(v_{j+1}(m))}^4\right]}{\varepsilon^4}
		\leq \frac{\Cste}{\varepsilon^4m^2},
	\end{align}
	where the last inequality comes from the combination of \eqref{newswag2} and \eqref{newswag5}. By the union bound, we obtain
	\[\sup_{j\in\ent{0}{m-1}}\sup_{t\in I_j(m)}\abso{\widebar{\mathcal{Y}}^{(p)}_m(v_j(m))-\widebar{\mathcal{Y}}^{(p)}_m(t)}\xrightarrow[m]{(proba.)}0.\]

	Let us handle now the process $(\widebar{\fys}^{(p)}_m)$, and recall its very definition in \Cref{not:yb} as $\widebar{\fys}^{(p)}_m=\sqrt{m}\left(\fcs_m^{(p)}-\mathsf{C}_\infty^{(p)}\right)$. We can write
	\begin{align}
		\sup_{t\in I_j(m)}\abso{\widebar{\fys}^{(p)}_m(t)-\widebar{\fys}^{(p)}_m(v_j(m))}\leq L_1^{(p)}(j,m) +L_2^{(p)}(j,m) %+\sqrt{m}\abso{{\fcs}^{(p)}_\infty(t)-{\fcs}^{(p)}_\infty(v_j(m))}.
	\end{align}
	where $L_1^{(p)}(j,m):=   \sup_{t\in I_j(m)}\sqrt{m}\abso{{\fcs}^{(p)}_\infty(t)-{\fcs}^{(p)}_\infty(v_j(m))}$ is deterministic since ${\fcs}_\infty(t)=(x(t),y(t))$ (see \eqref{eq:baranylimit}). This terms is easily handled:
	\[\sup_{j\in\ent{0}{m-1}} 	\sup_{t\in I_j(m)}L_1^{(p)}(j,m)\leq \frac{\Cste}{\sqrt{m}} \sous{\longrightarrow}{m}0.\]
	So it suffices to prove that
	\begin{align}\label{eq:qsgjdqs}
		L_2^{(p)}(j,m) :=  \sup_{t\in I_j(m)}\sqrt{m}\abso{\widebar{\fcs}^{(p)}_m(t)-\widebar{\fcs}^{(p)}_m(v_j(m))}\xrightarrow[m]{(proba.)}0.
	\end{align}

	Back to the definition of $\widebar{\fcs}_m$ in \eqref{defC}, the term $\abso{\widebar{\fcs}_m(t)-\widebar{\fcs}_m(v_j(m))}$ keeps the contribution of the vectors $\mathbf{w}[m]$ whose slope are in the interval $\left[\tnorm{v_j(m)},\tnorm{t}\right)$, \ie
	\begin{align}\label{eq:gwfds}
		\sqrt{m}(\widebar{\fcs}_m(t)-\widebar{\fcs}_m(v_j(m)))=\frac1{\sqrt{m}}\sum_{i=1}^{m}(\Bzeta_i,\Bxi_i)\mathbb{1}_{\tnorm{v_j(m)}\leq\frac{\Bxi_i}{\Bzeta_i}<\tnorm{t}}.
	\end{align}
	%	The number of vectors that is actually kept in the upper sum has the law of $\mathsf{N}_{[t,j]}$ (a non-decreasing process in $t$) so that it is of course smaller than $\mathsf{N}_{[j,j+1]}$.
	In the r.h.s. of \eqref{eq:gwfds}, the sum is a sum of non-negative terms: it is then immediate that
	\begin{align}\label{eq:gwfds2}\sup_{j\in\ent{0}{m-1}} L_2^{(1)} (j,m)\leq\frac1{\sqrt{m}}\max\{\Bzeta_i, 1\leq i \leq m\}\max_{j\in\ent{0}{m-1}} \mathsf{N}_{[j,j+1]},\end{align}
	(and a similar property holds for $\sup_j L_2^{(2)} (j,m)$).
	By the union bound $\P(\max\{\Bzeta_i, 1\leq i \leq m\}\geq 3\log m) \leq 1/m^2$. By Markov, the binomial r.v. $\mathsf{N}_{[j,j+1]}$ satisfies $\displaystyle\mathbb{P}\left(\mathsf{N}_{[j,j+1]}\geq \flr{2\log(m)}\right)\leq \mathbb{E}\left[e^{\mathsf{N}_{[j,j+1]}}\right]e^{-\flr{2\log(m)}}\leq {e^{e-1}}/{m^2}$.
	These two bounds, together with \eqref{eq:gwfds2}, yield \eqref{eq:qsgjdqs}.
	
	This completes {\bf Step 3} of the proof, as well as the whole proof itself. Indeed, by \Cref{lem:decom}, and what we proved in {\bf Step 2} and {\bf 3}, the sequence of processes $\left(\widebar{\fys}_m\right)$ converges in distribution to the same limit as this of the sequence $\left(\widebar{\mathcal{Y}}_m\right)$. This limit has been proven to be $\widebar{\fys}$.
	
\end{proof}

\section{Random generation}
\label{sec1}
In this section, given $\kappa\geq3,$ we provide an algorithm to sample a $n$-tuple of points $z[n]$ with distribution $\Dtn$. In the special cases $\kappa=3,4$ where we have $\Dtn=\Qn{\kappa}$ (recall \Cref{not:Qn} and \Cref{not:qtn}), we will provide two alternative sharpened algorithms.
\subsection{Algorithm of $\kappa$-sampling}
We propose below a new algorithm allowing to sample $\zzn$ with distribution $\Dtn$ for $\kappa$ and $n$ fixed. The algorithm starts by computing the $\ECP$; we will use the following notation to express the side lengths with respect to $\Lj$:
put $\displaystyle P_\kappa=\kappa r_\kappa/m_\kappa$ with $\displaystyle m_\kappa=({1+\cth})/{\sth}$, and recall that by \Cref{prop1}, for $\ell[\kappa]\in\mathcal{L}_\kappa$ and $\widetilde{c}[\kappa]=\frac{1}{m_\kappa}(r_\kappa-\cl_j(\Lj))_{j\in\entk}$ 
with $\widetilde{c}=\sum_{j=1}^\kappa\widetilde{c}_j$, we have \[\widetilde{c}+\sum_{j=1}^{\kappa}\ell_j=P_\kappa.\]
\begin{algorithm}[H]
	\caption{$\kappa$-sampling}
	\begin{algorithmic}
		\State {{\hspace{2cm}\bf Step 1:} \underline{Sample $\ell[\kappa]$; Construct $\widetilde{c}[\kappa]$;}}	
		\Do\{ \State	$\quad$ Draw a beta r.v. $S\sim P_\kappa\cdot\beta(k+1,n)$ and draw $k$ i.i.d. uniform r.v. $(u_1,\ldots,u_k)$ in $[0,S]$;
		\State	$\quad$ Sort them in increasing order into $(u_{(1)},\ldots,u_{(k)})$;
		\State	$\quad$ $\ell[\kappa]\gets (u_{(1)},u_{(2)}-u_{(1)},\ldots,u_{(\kappa)}-u_{(\kappa-1)});$\}
		\hspace{1cm} \doWhile{\{$\ell[\kappa]\notin\mathcal{L}_\kappa$;\}}
		
		\State $\bullet$ Compute $\widetilde{c}[\kappa]=\frac{1}{m_\kappa}(r_\kappa-\cl_j(\Lj))_{j\in\entk}$;
		
		\State {{\hspace{2cm}\bf Step 2:} \underline{Sample $\ssk$;}}
		\Do \{$\quad$ \State $\ssk\gets\mathcal{M}(n,\frac{1}{\kappa},\ldots,\frac{1}{\kappa})$\text{\quad\# \it sampling of multinomial r.v.};
		$\quad$ \State $\displaystyle N[\kappa]\gets\mathcal{M}(2n-\kappa,\frac{\widetilde{c}_1}{\widetilde{c}},\ldots,\frac{\widetilde{c}_\kappa}{\widetilde{c}})$;\}
		\doWhile{\{$\exists j\in\entk, N_j\neq s_j+s_{\wj}-1$ and $\ssk\notin\NN_n(\kappa)$;\}}
		
		\State{{\hspace{2cm}\bf Step 3:} \underline{Compute $\zzn$;}}
		\For{$j\in\entk$}\{
		\State	$\quad$ Draw $N_j$ r.v. i.i.d. uniform $(u^{(j)}_1,\ldots,u^{(j)}_{N_j})$ in the segment $[0,c_j]$;
		\State	$\quad$ Sort them into $(u^{(j)}_{(1)},\ldots,u^{(j)}_{{(N_j)}})$;
		\State	$\quad$ Build $\Delta u^{(j)}[N_j+1]=(u^{(j)}_{(1)},u^{(j)}_{(2)}-u^{(j)}_{(1)}\ldots,u^{(j)}_{{(N_j)}}-u^{(j)}_{{(N_j-1)}},c_j-u^{(j)}_{{(N_j)}})$;\}
		\EndFor
		\For{$j\in\entk$}\{
		\State	$\quad$ Build the vectors $v_k^{(j)}=A_j(\theta_\kappa)^{-1}\begin{pmatrix} 
			\Delta u^{{(j-1)}}_{s_{j-1}+k}\\ 
			\Delta u^{{(j)}}_k
		\end{pmatrix};$
		\State	$\quad$ Sort them into $(v_{(1)}^{(j)},\ldots,v_{(s_j)}^{(j)})$ by increasing slope;\}
		\EndFor
		\State $\bullet$ Gather all vectors into $v[n]=(v_{(1)}^{(1)},\ldots,v_{(s_1)}^{(1)},\ldots,v_{(1)}^{(\kappa)},\ldots,v_{(s_\kappa)}^{(\kappa)})$;
		\State $\bullet$ Build the $\ECP$ at distance $\ell[\kappa]$ from the sides of $\Ck$. The vectors $v[n]$ form the boundary of a unique convex polygon circumscribed in this $\ECP$,  whose set of vertices is a $n$-tuple $\zzn$ in $\CVn$;
	\end{algorithmic}
\end{algorithm}

Let us prove that this algorithm returns a $n$-tuple $\zzn$ that is $\Dtn$-distributed, and that this is performed within a reasonable time. Notation $\propto$ means that two quantities are proportional.
\begin{proof}[Proof of theorem \ref{thm-2}]
	Denote by $\mathbb{P}_{\Alg}$ the probability of an event in our $\kappa$-sampling. The distribution induced in the $1^{st}$ step of the algorithm (which is nothing but a rejection algorithm) satisfies 
	\begin{align}
		\mathbb{P}_{\Alg}\left(\lk\in\prod_{i=1}^{\kappa}\mathrm{d}\ell_i\right)
		&\propto ~\mathbb{1}_{\Lj\in\mathcal{L}_\kappa}\left(P_\kappa-\sum_{j=1}^{\kappa}\ell_j\right)^{2n-\kappa}\prod_{j=1}^{\kappa}\mathrm{d}\ell_j= ~\mathbb{1}_{\Lj\in\mathcal{L}_\kappa}\widetilde{c}^{2n-\kappa}\prod_{j=1}^{\kappa}\mathrm{d}\ell_j.
	\end{align}
	The $2^{nd}$ step (which is also a rejection algorithm working at $\Lj$ fixed) induces the distribution 
	\begin{align}\nonumber
		\mathbb{P}_{\Alg}\left(\sk=\ssk~\big\vert~\lk\in\prod_{i=1}^{\kappa}\mathrm{d}\ell_i\right)&\propto~ \mathbb{1}_{\ssk\in\NN_\kappa(n)}\prod_{j=1}^{\kappa}\frac{(1/\kappa)^{s_j}}{s_j!}\frac{(\widetilde{c}_j/\widetilde{c})^{N_j}}{N_j!}\\
		&\propto~ \mathbb{1}_{\ssk\in\NN_\kappa(n)}\prod_{j=1}^{\kappa}\frac{(\widetilde{c}_j/\widetilde{c})^{s_j+s_{\wj}-1}}{s_j!(s_j+s_{\wj}-1)!}.
	\end{align}
	This gives the appropriate joint distribution for $(\lk,\sk)$ as computed in \Cref{thm:distri}.
	Now, given $\Lj,\Cj,\ssk$, and since, conditionally to the $\Cj$, the projections of the vectors $v[n]$ are uniform in the $\Cj,$ the way we are constructing our vectors is valid and so is the building of $z[n].$
	\\
	
	We saw that the $\Lj$ behaves in $c/n$ in \Cref{thmonstre}, so that $\mathbb{1}_{\Lj\in n\mathcal{L}_\kappa}\to 1$ pointwise and thus the condition $\ell[\kappa]\in\mathcal{L}_\kappa$ of {\bf Step 1} is satisfied with a probability going to $1$ with $n$. This step costs only the drawing of $2n$ uniform variates requiring $\mathcal{O}\left(n\log(n)\right)$ operations. Notice that to perfect the algorithm, the unused uniform variates of {\bf Step 1} could be reused for the next steps.
	
	As for the second step, the probability that the two multinomial samples coincide behaves in $n^{-\kappa/2}$.
	A standard efficient algorithm to simulate a multinomial distribution is the alias method presented by Walker \cite{walker1977efficient}, whose theoretical basis got provided by Kronmal and Peterson \cite{kronmal1979alias}. In our case the complexity of the alias method is $\mathcal{O}\left(n\kappa\log(\kappa)\right)$.\\
	There exists other efficient procedures like the two-stage method of Brown and Bromberg \cite{brown1984efficient}. For a discussion about the most suitable method of multinomial sampling, we refer to \cite{davis1993computer}.
	
	The last step includes several sampling and sorting of $2n$ variates, which is also a complexity in $\mathcal{O}\left(n\log(n)\right)$.
	
	The second step is obviously the costliest and this gives a global complexity of $\mathcal{O}\left(n^{\kappa/2+1}\kappa\log(\kappa)\right).$ 
	Despite important effort to find an efficient algorithm to reduce the cost of this step, we were not able to find any convincing advances regarding this question.
\end{proof}

\begin{figure}[H]
	\begin{flushleft}
		\begin{subfigure}[hbtp]{0.3\textwidth}
			\includegraphics[scale=0.4]{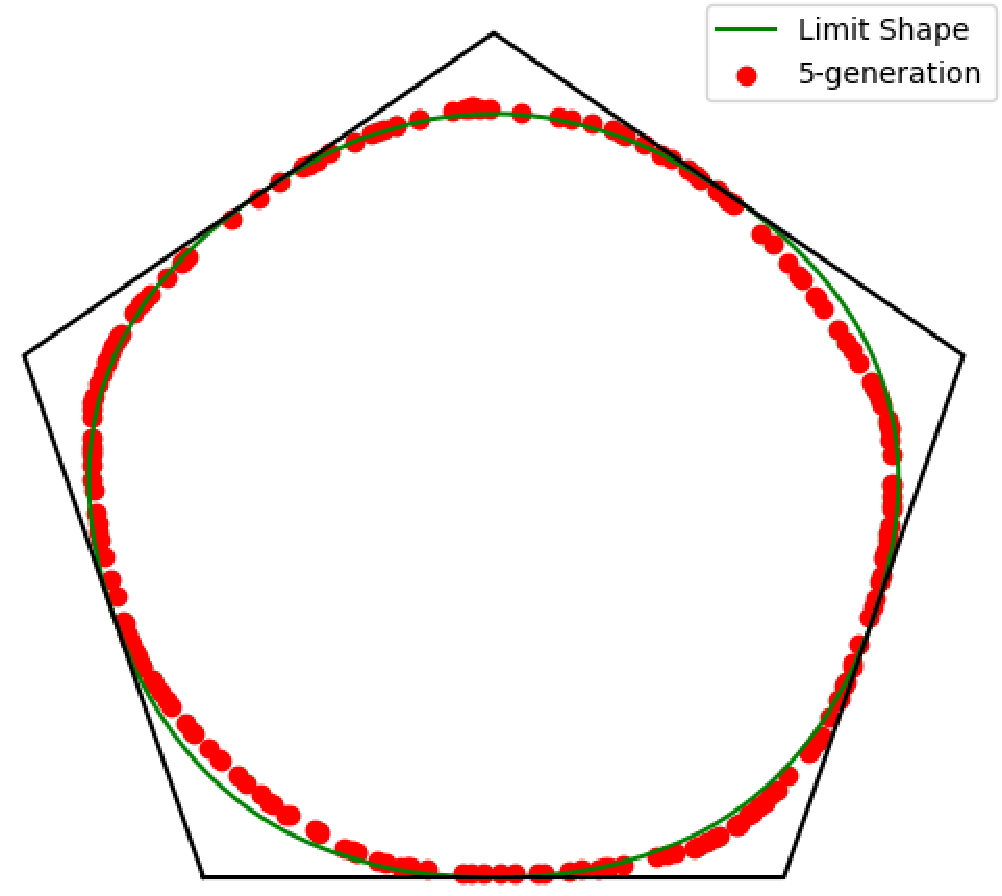}
			\caption{$5$-sampling, $n=200$}
		\end{subfigure}\hspace{4cm}
		\begin{subfigure}[hbtp]{0.3\textwidth}
			\includegraphics[scale=0.4]{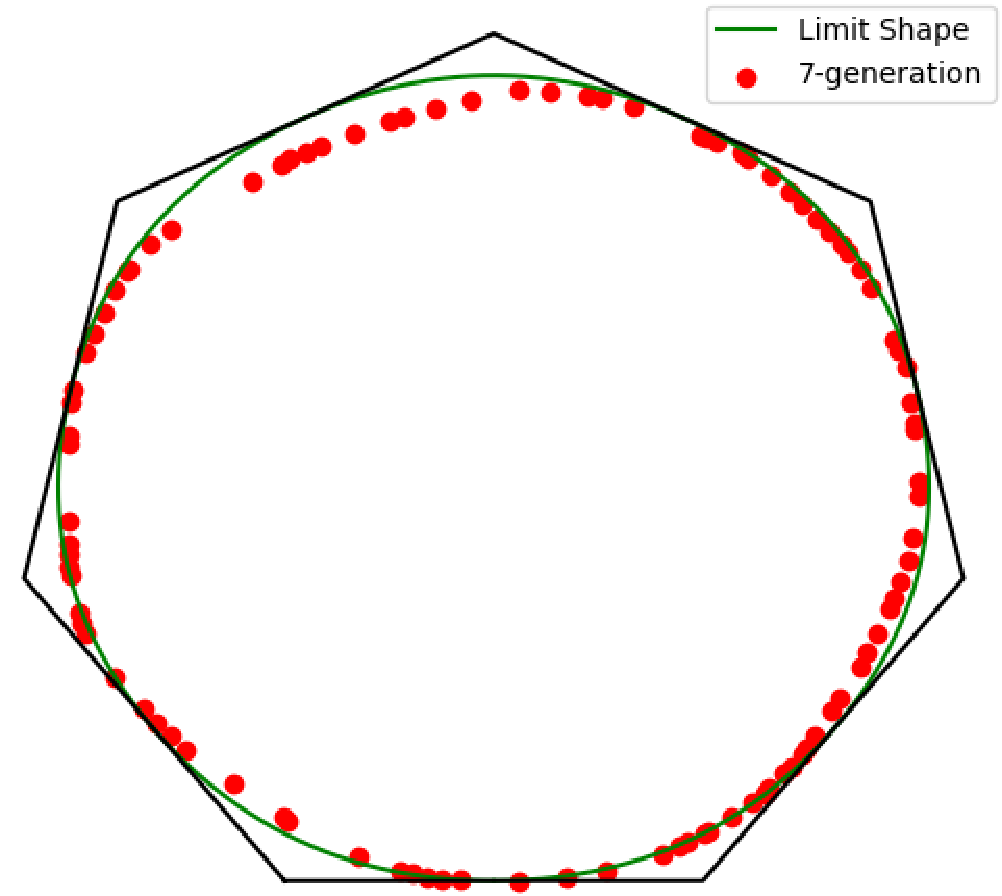}
			\caption{$7$-sampling, $n=100$}
		\end{subfigure}
	\end{flushleft}
	\caption{Two examples of $\kappa$-sampling. The set of points in red is the set of vertices of a convex $z[n]$-gon, whose boundary we can see to be very close to the limit shape drawn in green.}
\end{figure}

\subsection{Exact and fast algorithm of $\triangle$-sampling}

In the triangular case, the algorithm of $\triangle$-sampling avoids the rejection-sampling steps 1 and 2 included in the algorithm of $\kappa$-sampling, which makes the sampling direct and immediately grants a reasonable computation time.
Indeed, it happens to be that in the case $\kappa=3$, the joint distribution of $\Bell^{(n)}[3],\mathbf{s}^{(n)}[3]$ comes with simplifications (we will show this in \Cref{valtr3}), and becomes
\begin{multline}\label{equ4}
	\mathbb{P}\left(\Bell^{(n)}[3]\in\prod_{j=1}^{3}\mathrm{d}\ell_j,\mathbf{s}^{(n)}[3]=(i,j,k)\right)=\frac{n!\sth^{n-3}}{\mathbb{P}_{\triangle}(n)((n-1)!)^3}\mathbb{1}_{\ell_1+\ell_2+\ell_3\leq\frac{\sqrt{3}}{2}r_3}
	\\\times\left(r_3-\frac{2}{\sqrt{3}}(\ell_1+\ell_2+\ell_3)\right)^{2n-3}
	{n-1 \choose i}{n-1 \choose j}{n-1 \choose k}\mathbb{1}_{i+j+k=n}\prod_{j=1}^{3}\mathrm{d}\ell_j.
\end{multline}

\begin{algorithm}[H]
	\caption{$\triangle$-sampling}
	\begin{algorithmic}
		\State {\hspace{2cm}{\bf Step 1:} \underline{Sample $s[3]$;}}	
		\State	$\bullet$ Draw $(s_1,s_2,s_3)$ as three binomial r.v. $\displaystyle \mathcal{B}(n-1,\frac{1}{2})$ conditioned on $s_1+s_2+s_3=n$.
		\State $\bullet$ $N[3]\gets \left(s_1+s_2-1,s_2+s_3-1,s_3+s_1-1\right)$;
		
		\State {\hspace{2cm}{\bf Step 2:} \underline{Sample $\ell[3]$;}}
		\State $\bullet$ Draw $2n$ uniform $(u_1,\ldots,u_{2n})$ in the segment $[0,r_3]$;
		\State $\bullet$ Sort them into $(u_{(1)},\ldots,u_{(2n)})$;
		\State $\bullet$ $\ell[3]\gets (u_{(1)},u_{(2)}-u_{(1)},u_{(3)}-u_{(2)});\text{\it \# only the first 3 $u_{(i)}$ are used}$
		\State $\bullet$ Compute $c=r_3-(\ell_1+\ell_2+\ell_3)$;
		
		\State{\hspace{2cm}{\bf Step 3:} \underline{Compute $\zzn$;}}
		\For{$1\leq j\leq3$}\{
		\State $\quad$ Pick uniformly $N_j$ points among the $u_{(j)},j\in\{4,\ldots,2n\}$ to form a partition of $[0,c]$: $(u^{(j)}_{(1)},\ldots,u^{(j)}_{{(N_j)}});$
		\State $\quad$ Build $\Delta u^{(j)}[N_j+1]=(u^{(j)}_{(1)},u^{(j)}_{(2)}-u^{(j)}_{(1)}\ldots,u^{(j)}_{{(N_j)}}-u^{(j)}_{{(N_j-1)}},c-u^{(j)}_{{(N_j)}})$;\}
		\EndFor
		\For{$1\leq j\leq3$}\{
		\State	$\quad$ Build the vectors $v_k^{(j)}=A_j(\frac{\pi}{3})^{-1}\begin{pmatrix} 
			\Delta u^{{(j-1)}}_{s_{j-1}+k}\\ 
			\Delta u^{{(j)}}_k
		\end{pmatrix}$ for $1\leq k\leq s_j$;
		\State	$\quad$ Sort them into $(v_{(1)}^{(j)},\ldots,v_{(s_j)}^{(j)})$ by increasing slope;\}
		\EndFor
		\State $\bullet$ Gather vectors into $v[n]=(v_{(1)}^{(1)},\ldots,v_{(s_1)}^{(1)},\ldots,v_{(1)}^{(3)},\ldots,v_{(s_3)}^{(3)})$;
		\State $\bullet$ Build the equilateral triangle $T$ of side length $c$ at distance $\ell[3]$ from the sides of $\mathfrak{C}_3$;
		\State $\bullet$ The vectors $(v_{(1)},\ldots,v_{(n)})$, sorted by increasing slope, form the boundary of a convex polygon circumscribed in $T$. The set of its vertices is a $n$-tuple $\zzn$ in convex position.
		%		\For{$i\in\{2,\ldots,n\}$}
		%		\State $\bullet$ $z_i\gets z_{i-1}+v_{i-1}$;
		%		\EndFor
		
	\end{algorithmic}
\end{algorithm}
To draw $(s_1,s_2,s_3)$ according to three binomial  distributions $\displaystyle \mathcal{B}(n-1,\frac{1}{2})$ and conditioned on $s_1+s_2+s_3=n$, we may:\\
$\bullet$ Draw $3n-3$ Bernoulli($\frac{1}{3}$) r.v. $\mathbf{x}[3n-3]$;\\
$\bullet$ Set $(s_1,s_2,s_3)=\left(\sum_{i=1}^{n-1}\mathbf{x}_i,\sum_{i=n}^{2n-2}\mathbf{x}_i,\sum_{i=2n-1}^{3n-3}\mathbf{x}_i\right)$;\\
$\bullet$ Correct $\mathbf{x}[3n-3]$ to have $\sum_{i=1}^{3n-3}\mathbf{x}_i=n$.
The binomial correction works the following way : if $\sum_{i=1}^{3n-3}\mathbf{x}_i=n$ then we have the right distribution. Otherwise, if $\sum_{i=1}^{3n-3}\mathbf{x}_i<n$, pick uniformly some $j\in\{i;\mathbf{x}_i=0\}$ and put $\mathbf{x}_j=1$ until $\sum_{i=1}^{3n-3}\mathbf{x}_i=n$. 
If $\sum_{i=1}^{3n-3}\mathbf{x}_i>n$ though, pick uniformly some $j\in\{i;\mathbf{x}_i=1\}$ and put $\mathbf{x}_j=0$ until $\sum_{i=1}^{3n-3}\mathbf{x}_i=n$.

At the end we get that 
\begin{align}\nonumber
	\mathbb{P}_{\Alg}\left(\mathbf{s}^{(n)}[3]=(i,j,k)\right)&~\propto~ {n-1 \choose i}\left(\frac1{2}\right)^{i}\left(\frac1{2}\right)^{n-1-i}{n-1 \choose j}\left(\frac1{2}\right)^{n-1}{n-1 \choose k}\left(\frac1{2}\right)^{n-1}\mathbb{1}_{i+j+k=n}\\
	&~\propto~{n-1 \choose i}{n-1 \choose j}{n-1 \choose k}\mathbb{1}_{i+j+k=n}.
\end{align}
The analysis of the sampling of $\Bell_n[3]$ is the same as in the general $\kappa$-sampling. In the end, the algorithm of $\triangle$-generation admits a global complexity of $\mathcal{O}\left(n\log(n)\right)$ (in expectation).

\begin{figure}[H]
	\centering
	\includegraphics[scale=0.4]{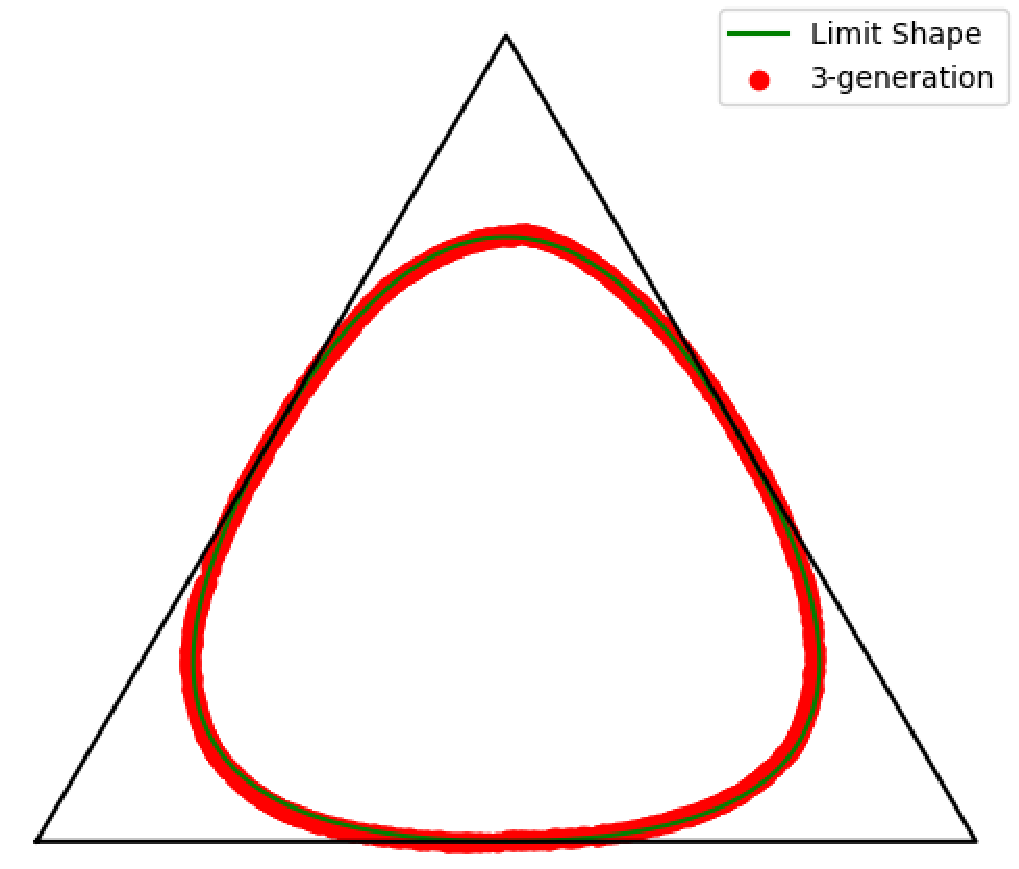}
	\caption{A $\triangle$-sampling, some $\zzn$-gon for $n=1000$, close to the limit shape}
\end{figure}

\subsection{Exact and fast algorithm of ${\Box}$-sampling}

In the square case, a fast ${\Box}$-sampling can be proposed, which is slightly different from the ${\triangle}$-sampling and from the ${\kappa}$-sampling. 

Indeed, in this case, no affine mapping intervenes in the proof since all corners are already right triangles. We may thus consider the law of the positive $x$-components $\mathbf{i}_n=\mathbf{s}_4^{(n)}+\mathbf{s}_1^{(n)}$ and $y$-components $\mathbf{j}_n=\mathbf{s}_1^{(n)}+\mathbf{s}_2^{(n)}$ of the vectors forming the boundary of a $\zzn$-gon. A quick calculus allows to compute
\begin{multline}
	\mathbb{P}\left(\Bell_n[4]\in\prod_{j=1}^{4}\mathrm{d}\ell_j,\mathbf{i}_n=i,\mathbf{j}_n=j\right)=\frac{(n!)^2}{\mathbb{P}_{\Box}(n)((n-1)!)^4}\mathbb{1}_{\ell_1+\ell_3\leq1}\mathbb{1}_{\ell_2+\ell_4\leq1}\left(1-(\ell_1+\ell_3)\right)^{n-2}
	\\
	\times\left(1-(\ell_2+\ell_4)\right)^{n-2}{n-1 \choose i}{n-1 \choose i-1}{n-1 \choose j}{n-1 \choose j-1}\prod_{j=1}^{4}\mathrm{d}\ell_j.
\end{multline}
In particular, this proves that $\mathbf{i}_n,\mathbf{j}_n$ are independent, and their law is explicit. It is easier to draw according to this distribution rather than considering $\mathbf{s}^{(n)}[4].$
This approach is, once more, inspired by Valtr's paper \cite{Valtr1995}.

\begin{algorithm}[H]
	\caption{$\Box$-sampling}
	\begin{algorithmic}
		\State {\hspace{2cm}{\bf Step 1:} \underline{Sampling of $\mathbf{i},\mathbf{j}$;}}	
		\State	$\bullet$ For both $\mathbf{i}$ and $\mathbf{j}$, draw two binomials r.v. $\mathcal{B}(n-1,\frac{1}{2})$ until their sum reaches $n$ and set $\mathbf{i}$ (or $\mathbf{j}$) as the result of the first.
		
		\State {\hspace{2cm}{\bf Step 2:} \underline{Sampling of $\ell[4]$;}}
		\State $\bullet$ Draw $n$ i.i.d. uniform r.v. $(u_1,\ldots,u_{n})$ in the segment $\zerun$ and sort them into $(u_{(1)},\ldots,u_{(n)})$;
		\State $\bullet$ $(\ell_1,\ell_3)\gets (u_{(1)},u_{(2)}-u_{(1)});\text{ \it \# use the 2 smallest}$ 
		\State $\bullet$ Draw $n$ i.i.d. uniform r.v. $(\hat{u}_1,\ldots,\hat{u}_{n})$ in the segment $\zerun$ and sort them into $(\hat{u}_{(1)},\ldots,\hat{u}_{(n)})$;
		\State $\bullet$ $(\ell_2,\ell_4)\gets (\hat{u}_{(1)},\hat{u}_{(2)}-\hat{u}_{(1)});$
		
		\State $\bullet$ Compute $c_1=1-(\ell_1+\ell_3)$ and $c_2=1-(\ell_2+\ell_4)$;
		
		\State{\hspace{2cm}{\bf Step 3:} \underline{Compute $\zzn$;}}
		\State $\bullet$ Pick uniformly $\mathbf{i}$ points $(x_1<\ldots<x_{\mathbf{i}})$ among the $u_{(k)},k\in\{3,\ldots,n\}$ and form $h^+[\mathbf{i}+1]=(x_1,x_2-x_1,\ldots,c_1-x_{\mathbf{i}})$ (positive increments). The $n-2-\mathbf{i}$ points $(\tilde{x}_1<\ldots<\tilde{x}_{n-2-\mathbf{i}})$ remaining in $u_{(k)},k\in\{3,\ldots,n\}$ are used to form $h^-[n-1-\mathbf{i}]=(-\tilde{x}_1,-\tilde{x}_2+\tilde{x}_1,\ldots,-c_1+\tilde{x}_{n-2-\mathbf{i}})$ (negative increments);
		\State $\bullet$ Pick uniformly $\mathbf{j}$ points $(y_1<\ldots<y_{\mathbf{j}})$ among the $\hat{u}_{(k)},k\in\{3,\ldots,n\}$ and form $v^+[\mathbf{j}+1]=(y_1,y_2-y_1,\ldots,c_1-y_{\mathbf{j}})$ (positive increments). The $n-2-\mathbf{j}$ points $(\tilde{y}_1<\ldots<\tilde{y}_{n-2-\mathbf{j}})$ remaining in $\hat{u}_{(k)},k\in\{3,\ldots,n\}$ are used to form $v^-[n-1-\mathbf{j}]=(-\tilde{y}_1,-\tilde{y}_2+\tilde{y}_1,\ldots,-c_1+\tilde{y}_{n-2-\mathbf{j}})$  (negative increments);
		\State $\bullet$ $h[n]={\sf Merge}(h^+,h^-)$ \text{ and } $v[n]={\sf Merge}(v^+,v^-)$;
		\State $\bullet$ Pick uniformly in $\mathcal{S}_n$ a permutation $\sigma$ and build for $1\leq i \leq n$ the vector $w_i=(h_i,v_{\sigma(i)})$;
		\State $\bullet$ Sort them into $(w_{(1)},\ldots,w_{(n)})$ by increasing slope;
		\State $\bullet$ Build the rectangle $R$ of side length $c_1$ (horizontally) and $c_2$ (vertically) at distance $\ell[4]$ from the sides of $\mathfrak{C}_4$;
		\State $\bullet$ The vectors $(w_{(1)},\ldots,w_{(n)})$, sorted by increasing slope, form the boundary of a convex polygon inscribed in $R$. The set of its vertices is a $n$-tuple $\zzn$ in convex position.
		%		\For{$i\in\{2,\ldots,n\}$}
		%		\State $\bullet$ $z_i\gets z_{i-1}+w_{(i-1)}$;
		%		\EndFor
		
	\end{algorithmic}
\end{algorithm}
The law of a r.v. $\mathbf{k}$ described in the first step satsifies:
\begin{align*}
	\mathbb{P}(\mathbf{k}=k)~\propto~{n-1\choose k}{n-1\choose d}\mathbb{1}_{k+d=n}={n-1\choose k}{n-1\choose k-1}.
\end{align*}
The probability that two binomial samples are equal is typically $\frac{1}{\sqrt{n}}$. A binomial sampling requires $\mathcal{O}\left(n\right)$ operations, and since this step is the costliest in the $\Box$-generation, the whole algorithm reaches a global complexity of $\mathcal{O}(n^{3/2})$.

\begin{figure}[H]
	\centering
	\includegraphics[scale=0.28]{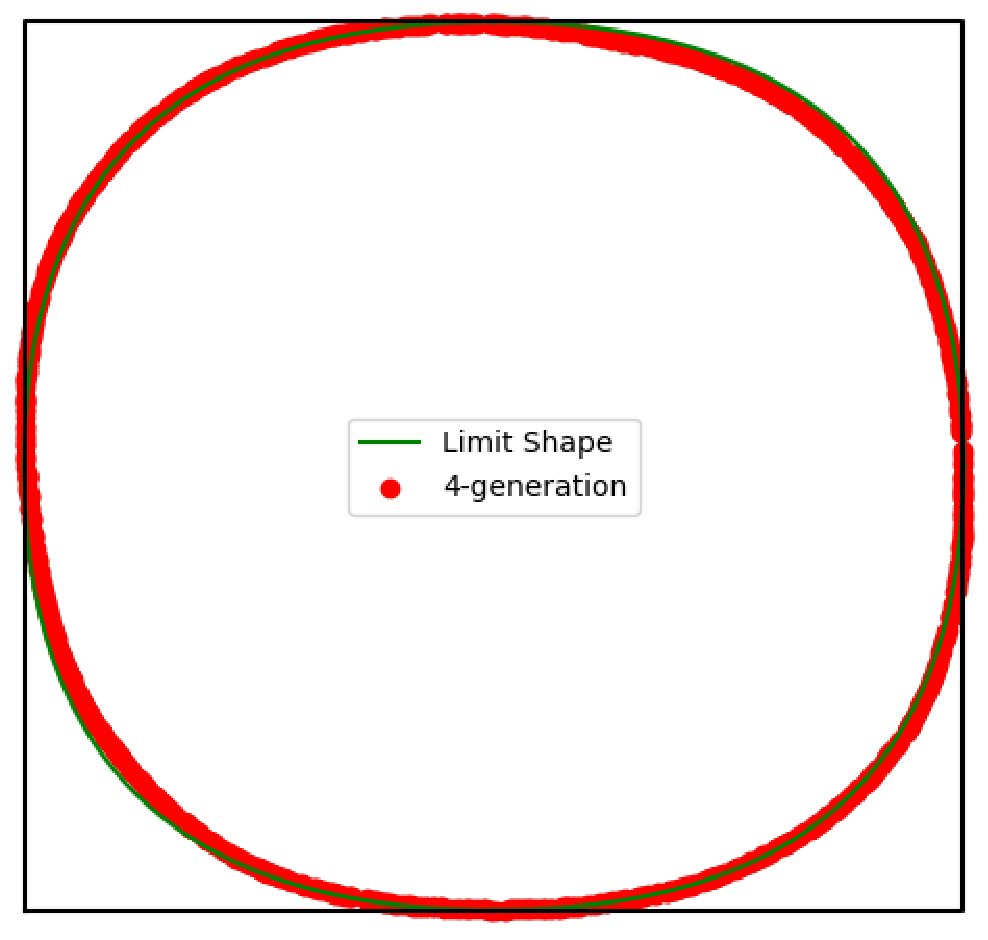}
	
	\caption{A $\Box$-sampling, some $\zzn$-gon for $n=1000$, close to the limit shape}
\end{figure}

\begin{appendices}
	\section{Proof of Lemma \ref{lem1}}\label{ann:LST}
	
	We decompose $\pk$ according to the number of sides of the $\ECP$ we are considering.
	
	Let $\zn$ has distribution $\Un$, and write
	\begin{align*}
		\pk&=n!\sum_{\substack{\In\subset\entk\\ \vert\In\vert\geq 3}}\mathbb{P}\left(\zn\in\CVn\cap \NZS(\zn)=\In\right).
	\end{align*}
	
	We now borrow to  B\'ar\'any (\cite{barany1},\cite{barany2}) some considerations:
	\begin{Definition}\label{def1}
		Given $S$ a convex compact set (nonflat), let $x_1,\ldots,x_m,x_{m+1}=x_1$ a subdivision of
		the boundary $\partial S$ and let $d_i$ be the line supporting $S$ at $x_i$ for all $i\in\{1,\ldots,m\}$.
		Write $y_i$ for the intersection of $d_i$ and $d_{i+1}$ (if $d_{i}=d_{i+1}$ then $y_i$ can be any point between
		$x_i$ and $x_{i+1}$). Let $T_i$ denote the triangle with vertices $x_i,y_i,y_{i+1}$ and also its area. We define the affine perimeter of the convex set $S$ as
		\[\AP(S)=2\lim \sum_{i=1}^m\sqrt[3]{T_i}\]
		where the limit is taken over all sequences of subdivisions $x[m]$ with $\max_{1,\ldots,m}\vert x_i-x_{i+1}\vert \to 0.$
	\end{Definition}

	\begin{Theoreme}[Limit Shape Theorem, {\bf Bárány} \cite{barany1}]\label{thm0} Let $K$ be a compact convex domain of $\RR^2$ with nonempty interior.
		\begin{enumerate}
			\item There exists a convex domain $\Dom{K}\subset K$ such that $\AP(\Dom{K})>\AP(S)$ for all convex set $S\subset K$ different from $\Dom{K}.$
			\item Let $n\geq3$, and let $\zn$ has distribution $\mathbb{Q}_K^{(n)}$. Then for all $\varepsilon>0$,
			\[\lim_{n\to+\infty}\mathbb{P}\left(d_H(\conv(\zn),\Dom{K})<\varepsilon\right)=1.\]
		\end{enumerate}
	\end{Theoreme}
	
	\begin{Definition}\label{def2}
		Define $\AP^*(K)=\max\{\AP(S), S \text{ convex sets included in }K\}.$ 
	\end{Definition}
	
	%	For the important properties of this object, we refer to B\'ar\'any's article \cite{barany1}.
	In the case of the regular $\kappa$-gons, the following result comes as a corollary of the properties of the affine perimeter:
	
	\begin{Lemma}\label{lem2}
		Let $p_1,\ldots,p_\kappa$ be the midpoints of the consecutives sides of $\Ck$, and $y_1,\ldots,y_\kappa$ the vertices of $\Ck$, so that $p_i$ is the middle of the segment $[y_i,y_{i+1}]$ (modulo $\kappa$). Let $\mathcal{C}_i$ be the unique parabola tangent to $p_i y_{i+1}$ at $p_i$ and tangent to $y_{i+1} p_{i+1}$ at $p_{i+1}.$
		The convex domain $\Dom{\Ck}$ is the subset of $\Ck$ whose boundary is formed by the parabolas $(\mathcal{C}_i)_{1\leq i\leq \kappa}$. The set $\Dom{\Ck}$ is thus tangent to $\Ck$ in the $\kappa$ points $p_1,\ldots,p_\kappa$.
	\end{Lemma}
		\usetikzlibrary{calc,decorations,decorations.text, math}
	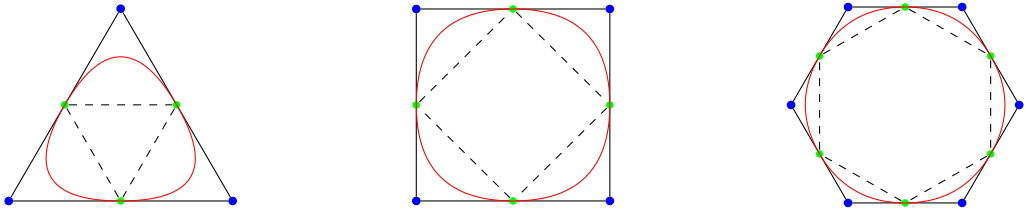
\begin{figure}[htbp]
		\begin{minipage}{0.3\textwidth}
			\centering
			\begin{tikzpicture}[scale=0.85]
				\def\R{2}
				\def\K{3}
				\def\alphaK{180/\K}

				%définition des points de l'hexagone extérieur
				\foreach \x in {1,2,...,\K} {\node[circle, draw=blue, fill = blue, inner sep = 1pt] (a\x) at ({90+360*\x/\K}:\R) {}; }
				%\foreach \x in {1,2,...,\K} {\node[above,sloped] at (a\x) {\tiny $y_\x$}}
				
				%	\node[] (a0) at ({360}:2*\R) {};
				
				\foreach \x [remember=\x as \y (initially \K)] in {1,2,...,\K} { 
					%point de l'hexagone intérieur et tracé de l'hexagone extérieur
					\path[draw] (a\x) -- (a\y) node[midway, circle, fill=green, inner sep = 1pt] (b\x) {}; 
				}
				\foreach \x [remember=\x as \y (initially \K)] in {1,2,...,\K} { 
					%tracé de l'hexagone intérieur
					\path[draw,dashed] (b\x) -- (b\y) node[midway] (c\x) {}; 
				}

				%		\foreach \i in {1,2,...,\K} {
					%		\begin{scope}[shift={(c\i)},rotate=-150+360*\i/\K](3.,.475) %*cos(180*(\K-2)/\K) %-135 pour K = 8
						%			\draw plot[ x={({sin(\thetaK/2)*\R/2/cos(\thetaK)},0)},y={(0,sqrt((1-1/(4*cos(\thetaK)*cos(\thetaK)) )/4 )*sin(\thetaK/2)*\R)},domain=-1:1,color=blue] (\x,{-1*\x*\x/2+1/2}); 
						%		\end{scope}
					%		}
				
				\foreach \i in {1,2,...,\K} {
					\begin{scope}[shift={(c\i)},rotate=240+360*\i/\K] %*cos(180*(\K-2)/\K)
						%					\draw[red] plot[ x={({sin(\thetaK/2)*\R/2/cos(\thetaK)},0)},y={(0,(1-sin(\thetaK/2)*sin(\thetaK/2) )*\R)},domain=-1:1] (\x,{-1*\x*\x/2+1/2}); 
						%					\draw[red] plot[ domain=-{cos(180/\K)*sin(180/\K)}:\Rc] (\x,{-1*\x*\x/2+1/2}); 
						\draw[red] plot[ domain=-{cos(180/\K)*sin(180/\K)}:{cos(180/\K)*sin(180/\K)}] ({\x*\R},{-\R*\x*\x/2/cos(\alphaK)/cos(\alphaK)+\R*sin(\alphaK)*sin(\alphaK)/2}); 
					\end{scope}
				}
				
				%devrait marcher, je comprends pas
				%\begin{scope}[shift={(c1)},rotate=-90] %*cos(180*(\K-2)/\K)
				%\draw plot[ x={({sin(\thetaK/2)*\R*cos(\thetaK/2)},0)},y={(0,sqrt((1-1/(4*cos(\thetaK)*cos(\thetaK)) )/4 )*sin(\thetaK/2)*\R)},domain=-1:1, color=blue] (\x,{-1*\x*\x/2+1/2}); 
				%\end{scope}
				
			\end{tikzpicture}
		\end{minipage}
		$\quad$
		\begin{minipage}{0.3\textwidth}
			\centering
			\begin{tikzpicture}[scale=0.9]
				\def\R{2}
				\def\K{4}
				\def\alphaK{180/\K}

				%définition des points de l'hexagone extérieur
				\foreach \x in {1,2,...,\K} {\node[circle, draw=blue, fill = blue, inner sep = 1pt] (a\x) at ({45+360*\x/\K}:\R) {}; }
				%\foreach \x in {1,2,...,\K} {\node[above,sloped] at (a\x) {\tiny $y_\x$}}
				
				%	\node[] (a0) at ({360}:2*\R) {};
				
				\foreach \x [remember=\x as \y (initially \K)] in {1,2,...,\K} { 
					%point de l'hexagone intérieur et tracé de l'hexagone extérieur
					\path[draw] (a\x) -- (a\y) node[midway, circle, fill=green, inner sep = 1pt] (b\x) {}; 
				}
				\foreach \x [remember=\x as \y (initially \K)] in {1,2,...,\K} { 
					%tracé de l'hexagone intérieur
					\path[draw,dashed] (b\x) -- (b\y) node[midway] (c\x) {}; 
				}

				%		\foreach \i in {1,2,...,\K} {
					%		\begin{scope}[shift={(c\i)},rotate=-150+360*\i/\K] %*cos(180*(\K-2)/\K) %-135 pour K = 8
						%			\draw plot[ x={({sin(\thetaK/2)*\R/2/cos(\thetaK)},0)},y={(0,sqrt((1-1/(4*cos(\thetaK)*cos(\thetaK)) )/4 )*sin(\thetaK/2)*\R)},domain=-1:1,color=blue] (\x,{-1*\x*\x/2+1/2}); 
						%		\end{scope}
					%		}
				
				\foreach \i in {1,2,...,\K} {
					\begin{scope}[shift={(c\i)},rotate=225+360*\i/\K] %*cos(180*(\K-2)/\K)
						%					\draw[red] plot[ x={({sin(\thetaK/2)*\R/2/cos(\thetaK)},0)},y={(0,(1-sin(\thetaK/2)*sin(\thetaK/2) )*\R)},domain=-1:1] (\x,{-1*\x*\x/2+1/2}); 
						%					\draw[red] plot[ domain=-{cos(180/\K)*sin(180/\K)}:\Rc] (\x,{-1*\x*\x/2+1/2}); 
						\draw[red] plot[ domain=-{cos(180/\K)*sin(180/\K)}:{cos(180/\K)*sin(180/\K)}] ({\x*\R},{-\R*\x*\x/2/cos(\alphaK)/cos(\alphaK)+\R*sin(\alphaK)*sin(\alphaK)/2}); 
					\end{scope}
				}
				
				%devrait marcher, je comprends pas
				%\begin{scope}[shift={(c1)},rotate=-90] %*cos(180*(\K-2)/\K)
				%\draw plot[ x={({sin(\thetaK/2)*\R*cos(\thetaK/2)},0)},y={(0,sqrt((1-1/(4*cos(\thetaK)*cos(\thetaK)) )/4 )*sin(\thetaK/2)*\R)},domain=-1:1, color=blue] (\x,{-1*\x*\x/2+1/2}); 
				%\end{scope}
				
			\end{tikzpicture}
		\end{minipage}
		$\quad$
		\begin{minipage}{0.3\textwidth}
			\centering
			\begin{tikzpicture}[scale=0.75]
				\def\R{2}
				\def\K{6}
				\def\alphaK{180/\K}

				%définition des points de l'hexagone extérieur
				\foreach \x in {1,2,...,\K} {\node[circle, draw=blue, fill = blue, inner sep = 1pt] (a\x) at ({360*\x/\K}:\R) {}; }
				%\foreach \x in {1,2,...,\K} {\node[above,sloped] at (a\x) {\tiny $y_\x$}}
				
				%	\node[] (a0) at ({360}:2*\R) {};
				
				\foreach \x [remember=\x as \y (initially \K)] in {1,2,...,\K} { 
					%point de l'hexagone intérieur et tracé de l'hexagone extérieur
					\path[draw] (a\x) -- (a\y) node[midway, circle, fill=green, inner sep = 1pt] (b\x) {}; 
				}
				\foreach \x [remember=\x as \y (initially \K)] in {1,2,...,\K} { 
					%tracé de l'hexagone intérieur
					\path[draw,dashed] (b\x) -- (b\y) node[midway] (c\x) {}; 
				}

				%		\foreach \i in {1,2,...,\K} {
					%		\begin{scope}[shift={(c\i)},rotate=-150+360*\i/\K] %*cos(180*(\K-2)/\K) %-135 pour K = 8
						%			\draw plot[ x={({sin(\thetaK/2)*\R/2/cos(\thetaK)},0)},y={(0,sqrt((1-1/(4*cos(\thetaK)*cos(\thetaK)) )/4 )*sin(\thetaK/2)*\R)},domain=-1:1,color=blue] (\x,{-1*\x*\x/2+1/2}); 
						%		\end{scope}
					%		}
				
				\foreach \i in {1,2,...,\K} {
					\begin{scope}[shift={(c\i)},rotate=-150+360*\i/\K] %*cos(180*(\K-2)/\K)
						%					\draw[red] plot[ x={({sin(\thetaK/2)*\R/2/cos(\thetaK)},0)},y={(0,(1-sin(\thetaK/2)*sin(\thetaK/2) )*\R)},domain=-1:1] (\x,{-1*\x*\x/2+1/2}); 
						%					\draw[red] plot[ domain=-{cos(180/\K)*sin(180/\K)}:\Rc] (\x,{-1*\x*\x/2+1/2}); 
						\draw[red] plot[ domain=-{cos(180/\K)*sin(180/\K)}:{cos(180/\K)*sin(180/\K)}] ({\x*\R},{-\R*\x*\x/2/cos(\alphaK)/cos(\alphaK)+\R*sin(\alphaK)*sin(\alphaK)/2}); 
					\end{scope}
				}
				
				%devrait marcher, je comprends pas
				%\begin{scope}[shift={(c1)},rotate=-90] %*cos(180*(\K-2)/\K)
				%\draw plot[ x={({sin(\thetaK/2)*\R*cos(\thetaK/2)},0)},y={(0,sqrt((1-1/(4*cos(\thetaK)*cos(\thetaK)) )/4 )*sin(\thetaK/2)*\R)},domain=-1:1, color=blue] (\x,{-1*\x*\x/2+1/2}); 
				%\end{scope}
				
			\end{tikzpicture}
		\end{minipage}
		\caption{\label{fig-AP} For each case, $\kappa=3$, 4 and 6, the inner curve drawn in red is the boundary of the domain $\Dom{\Ck}$. By the limit shape theorem, it also represents the boundary of a $\zn$-gone where $\zn$ is taken under $\Qn{\kappa}$, when $n\to+\infty$.}
	\end{figure}
	\begin{proof}
		\Cref{thm0} indicates that $\Dom{\Ck}$ is the convex domain contained in $\Ck$ which maximizes the affine perimeter. By definition of the affine perimeter, we have $\AP(\Ck)=0$ so that $\Dom{\Ck}$ lies within the interior of $\Ck$. In this case (by Bárány \cite{barany1}), the boundary of $\Dom{\Ck}$ is composed of finitely many arcs of parabola. In order to maximize the affine perimeter, $\Dom{\Ck}$ has to be tangent to at least 3 sides of $\Ck$, and the symmetry of $\Ck$ forces these tangency points to be the $\kappa$ midpoints of $\Ck$'s sides. Hence between two consecutives midpoints lies an arc of parabola.
	\end{proof}
	\begin{Lemma}\label{lem3}
		For all $\kappa\geq 3$, the supremum of affine perimeters $\AP^*(\Ck)$ is
		\begin{align}\label{eq:aff}	
		\AP^*(\Ck)=\AP(\Dom{\Ck})=\kappa\left( r_\kappa^2\sth\right)^{\frac{1}{3}}.
		\end{align}
	\end{Lemma}
	\begin{proof}
		Using notation of \Cref{lem2}, we have of course by symmetry,
		\[\AP(\Dom{\Ck})=\sum_{i=1}^\kappa\AP(\mathcal{C}_i)=\kappa\AP(\mathcal{C}_1).\]
		Now using notation of \Cref{lem2}, let $\mathcal{T}_1$ be the area of the triangle with vertices $p_1, y_{2},p_{1}$. We claim that 
		\begin{align}
			\AP(\mathcal{C}_1)=\Area{\mathcal{T}_1}^{1/3}.
		\end{align}
		This property comes from the following fact, due to Blaschke\cite{blaschkedifferential}[p.38]. Consider a triangle $T$ with vertices $a,b,c$ and with subtriangles (both with dotted areas) $T^{(1)}$ (with vertices $a,d,f$) and $T^{(2)}$ (with vertices $f,e,c$) defined such that $(d,e)\in[a,b]\times[b,c]$ and the segment $[d,e]$ is tangent to the arc of parabola $\mathcal{C}$ at $f$ just like in \Cref{fig:affper}.
		\begin{figure}[H]
\centering
% This file was created with tikzplotlib v0.10.1.
\begin{tikzpicture}
\pgfdeclarelayer{bg}    % declare background layer
\pgfsetlayers{bg,main}
\definecolor{color1bg}{HTML}{842E1B}
\definecolor{darkgray176}{RGB}{176,176,176}
\definecolor{aquamarine}{rgb}{0.5, 1.0, 0.83}
\definecolor{purplepizzazz}{rgb}{1.0, 0.31, 0.85}
\definecolor{pakistangreen}{rgb}{0.0, 0.4, 0.0}
\definecolor{coral}{rgb}{1.0, 0.62, 0.0}
\definecolor{amber}{rgb}{1.0, 0.75, 0.0}
\definecolor{limegreen}{rgb}{0.2, 0.8, 0.2}
\begin{axis}[axis lines=none,
tick align=outside,
tick pos=left,
x grid style={darkgray176},
xmin=-1, xmax=3.5,
xtick style={color=black},
y grid style={darkgray176},
ymin=2, ymax=3.1,
ytick style={color=black}
]
\addplot [thick, blue]
table {%
3.23 2.14
1.5 3
-0.88 2.4
3.23 2.14
};
\addplot [thick, blue]
table {%
2.41 2.55
0.41 2.73
};
\addplot [thick, blue]
table {%
-0.88 2.4
1.41 2.64
};
\addplot [thick, blue]
table {%
3.23 2.14
1.41 2.64
};

\addplot [thick, red]
table {%
3.23 2.14
3.17591068200918 2.1666037880114
3.1231559846469 2.19172286747583
3.07166394934264 2.21546521861233
3.02136698235169 2.23792891135161
2.97220151975871 2.2592031741913
2.92410772073625 2.27936933023488
2.87702918647547 2.29850161895251
2.83091270243759 2.3166679193369
2.78570800178883 2.33393038774634
2.74136754808004 2.3503460217406
2.69784633541315 2.36596715955476
2.65510170450224 2.38084192346102
2.61309317318731 2.39501461409502
2.57178228009501 2.4085260618317
2.53113244026408 2.42141394045674
2.4911088116642 2.43371304766667
2.45167817163773 2.44545555632425
2.41280880238417 2.45667123987757
2.37447038468849 2.46738767490838
2.33663389916817 2.4776304233947
2.29927153437957 2.48742319694605
2.26235660118386 2.49678800498782
2.22586345282599 2.50574528862776
2.18976741022856 2.51431404172702
2.15404469204528 2.52251192051533
2.11867234905825 2.53035534293072
2.08362820253753 2.537859578726
2.04889078621365 2.54503883126303
2.01443929154144 2.55190631181004
1.98025351595926 2.55847430706467
1.94631381387047 2.56475424054361
1.91260105009423 2.5707567284083
1.87909655555129 2.57649163023245
1.8457820849666 2.58196809516162
1.81263977638514 2.58719460386519
1.77965211231044 2.59217900663692
1.74680188228663 2.59692855796151
1.7140721467549 2.60144994782898
1.68144620202419 2.60574933004764
1.64890754620348 2.60983234777827
1.61643984594943 2.61370415648641
1.58402690388887 2.61736944448701
1.5516526265798 2.6208324512346
1.51930099287835 2.62409698349296
1.48695602258144 2.62716642950106
1.45460174521691 2.63004377123544
1.42222216885342 2.63273159485445
1.38980124880251 2.63523209939549
1.35732285608425 2.63754710378311
1.32477074552605 2.63967805219308
1.29212852336134 2.64162601780527
1.25937961419147 2.64339170496652
1.22650722716898 2.64497544977241
1.19349432125516 2.64637721906592
1.16032356939764 2.64759660783868
1.12697732146587 2.64863283500892
1.09343756577302 2.6494847375379
1.05968588900223 2.65015076283393
1.02570343434275 2.65062895937971
0.991470857627966 2.65091696550476
0.956968281251331 2.65101199620937
0.922175245618906 2.65091082793064
0.887070657877044 2.65060978112338
0.85163273763161 2.65010470050949
0.815838959349913 2.64939093282858
0.779665991108409 2.6484633018994
0.74308962931756 2.6473160807759
0.706084729019734 2.64594296075344
0.668625129316109 2.64433701694882
0.630683573433718 2.64249067014228
0.592231622893369 2.6403956445296
0.553239565182314 2.63804292098771
0.513676314271697 2.63542268540627
0.47350930324654 2.63252427158054
0.432704368234576 2.62933609809545
0.39122562272802 2.62584559855728
0.349035321287964 2.62203914444458
0.306093711502629 2.61790195975418
0.262358872936138 2.61341802650775
0.217786541651324 2.60856998005846
0.172329918715503 2.60333899299198
0.125939460898757 2.59770464624949
0.0785626515461227 2.59164478590734
0.030143749343551 2.58513536382548
-0.0193764876029523 2.57815026011856
-0.0700611053014155 2.57066108510342
-0.12197728834666 2.56263695802765
-0.17519672911837 2.55404425947625
-0.229796036025779 2.54484635387611
-0.285857185741169 2.53500327795889
-0.343468025087561 2.52447139038476
-0.4027228290902 2.51320297695316
-0.463722922689679 2.5011458049097
-0.526577374774608 2.48824261876978
-0.591403774557144 2.47443056878548
-0.658329101926927 2.45964056163822
-0.727490705327899 2.4437965210919
-0.799037402970523 2.4268145441225
-0.873130725894949 2.40860193536673
};

\node[inner sep=1.pt,circle,draw=black,fill=limegreen] (P3) at (axis cs:3.23,2.14){};	
\draw[above,color=limegreen](P3) node {\small $c$};

\node[inner sep=1.pt,circle,draw=black,fill=limegreen] (P4) at (axis cs:-0.88,2.4){};	
\draw[above,color=limegreen](P4) node {\small $a$};

\node[inner sep=1.pt,circle,draw=black,fill=limegreen] (P1) at (axis cs:0.41,2.73){};	
\draw[above,color=limegreen](P1) node {\small $d$};

\node[inner sep=1.pt,circle,draw=black,fill=limegreen] (P2) at (axis cs:2.41,2.55){};	
\draw[above,color=limegreen](P2) node {\small $e$};

\node[inner sep=1.pt,circle,draw=black,fill=limegreen] (P5) at (axis cs:1.5,3.){};	
\draw[above,color=limegreen](P5) node {\small $b$};

\node[inner sep=1.pt,circle,draw=black,fill=red] (P6) at (axis cs:1.41,2.64){};	
\draw[above,color=red](P6) node {\small $f$};

\end{axis}

\begin{pgfonlayer}{bg}

\fill[pattern=dots,pattern color=color1bg] (P4.center)--(P1.center)--(P6.center)--cycle;
\fill[pattern=dots,pattern color=color1bg] (P6.center)--(P2.center)--(P3.center)--cycle;
\end{pgfonlayer}

\end{tikzpicture}
\captionof{figure}{\label{fig:affper}Blaschke's property for arcs of parabolas}
\end{figure}
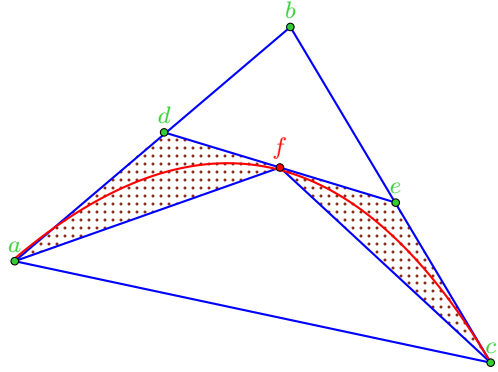
		In this case, we have 
		\[\Area{T}^{1/3}=\Area{T^{(1)}}^{1/3}+\Area{T^{(2)}}^{1/3}.\]
		Therefore, for any integer $m$, any tuple of points $x[m]\in\mathcal{C}_1$ and triangles $T_i,i\in\ent{1}{m}$ (both defined as in \Cref{def1}), the quantity $\lim_{x[m]}\sum_{i=1}^{m}\Area{{T}_i}^{1/3}$ is constant hence
		\[\AP(\mathcal{C}_1):=\Area{\mathcal{T}_1}^{1/3}.\]
		Now with petty computations one gets
		\[{\Area{\mathcal{T}_1}}^{1/3}=\frac{1}{2}\left(r_\kappa^2\sth\right)^{1/3},\]
		which is \eqref{eq:aff}.
	\end{proof}

	\begin{proof}[Proof of \Cref{lem1} :]
		For a $n$-tuple $\zn$ that is $\Un$-distributed, Bárány's \Cref{thm0} states that for all $\varepsilon>0$,
		\[\mathbb{P}\big(d_H(\conv(\zn),\Dom{\Ck})>\varepsilon~\vert~ \zn\in\CVn\big)\underset{n\to+\infty}{\longrightarrow}0,\] which implies immediately that
		\[\mathbb{P}\big(\NZS(\zn)=\entk~\vert~ \zn\in\CVn\big)=\frac{\ptk}{\pk}\underset{n\to+\infty}{\longrightarrow}1.\]
	\end{proof}

	\section{Valtr's results}\label{sec:valtr}
	
	The surprising simplicity of Valtr's formulas in the cases of the parallelogram and the triangle can be seen as a consequence of the fact that the sets $\mathcal{L}_4$ and $\mathcal{L}_3$ are easily computable. Let us recover these results with \Cref{thm:distri}.

	\subsection{The triangle}
	
	\begin{Theoreme}[{\bf Valtr }\cite{valtr1996probability}]\label{valtr3}
		For all $n\geq3$, we have 
		\[\mathbb{P}_\triangle(n)=\frac{2^n (3n-3)!}{(2n)!((n-1)!)^3}.\]
	\end{Theoreme}
	
	We propose a new proof of Valtr's result:
	\begin{proof}
		In the case $\kappa=3$, the side length of $\mathfrak{C}_3$ is $ r_3=2/3^{1/4}$. Pick $c_1,c_2,c_3,\ell_1,\ell_2,\ell_3$ verifying equations $(\mathcal{C}_j)_{1\leq j\leq 3}$. Since the only equiangular polygon with three sides is the equilateral triangle, we have $c_1=c_2=c_3=c$. This forces  
		\[(\mathcal{C}_1)=(\mathcal{C}_2)=(\mathcal{C}_3) : c+\frac{2}{\sqrt{3}}(\ell_1+\ell_2+\ell_3)= r_3.\]
		We thus understand that \[\mathcal{L}_3=\left\{(\ell_1,\ell_2,\ell_3)\in\left[0, \frac{\sqrt{3}}{2}r_3\right]^3\text{ with }\ell_1+\ell_2+\ell_3\leq \frac{\sqrt{3}}{2}r_3\right\}.\]
		We also have $\NN_3(n)=\left\{(i,j,k)\in \{0,\ldots,n-1\}^3, i+j+k=n\right\},$ so that, we have ${\mathbb{D}}_n^{(3)}={\mathbb{Q}}_n^{(3)}$ and together with
		\[\sum_{s[3]\in\NN_3(n)}\int_{\RR^3}f_n^{(3)}\left(s[3],\ell[3]\right)\mathrm{d}\ell_1\mathrm{d}\ell_2\mathrm{d}\ell_3=1,\]
		we obtain
		\[\mathbb{P}_{\triangle}(n)=n!\sin(\frac{\pi}{3})^{n-3}\sum_{s[3]\in\NN_3(n)}\int_{\ell[3]\in\mathcal{L}_3} \prod_{j=1}^{3}\frac{c^{s_{{j-1}}+s_j-1}}{s_j!(s_{{j-1}}+s_j-1)!}\mathrm{d}\ell_1\mathrm{d}\ell_2\mathrm{d}\ell_3.\]
		
		Put $(i,j,k)=s[3]$ and perform the substitution $\ell=\frac{2}{\sqrt{3}}(\ell_1+\ell_2+\ell_3)$ to get :
		\begin{align*}
			\mathbb{P}_{\triangle}(n)&=n!\sin(\frac{\pi}{3})^{n-3}\sum_{(i,j,k)\in\NN_3(n)}\int_{0}^{ r_3}\frac{1}{2}\ell^2\frac{( r_3-\ell)^{i+j-1}}{i!(i+j-1)!}\frac{( r_3-\ell)^{j+k-1}}{j!(j+k-1)!}\frac{( r_3-\ell)^{k+i-1}}{k!(k+i-1)!}\left(\frac{\sqrt{3}}{2}\right)^3\mathrm{d}\ell\\
			&=n!\left[\frac{1}{((n-1)!)^3}\sum_{i+j+k=n}{n-1 \choose i}{n-1 \choose j}{n-1 \choose k}\right]\left(\frac{\sqrt{3}}{2}\right)^n\int_0^{ r_3}\frac{1}{2}\ell^2( r_3-\ell)^{2n-3}\mathrm{d}\ell\\
			&=\frac{n!}{((n-1)!)^3}{3n-3 \choose n}\frac{(2n-3)!}{(2n)!}\left(\frac{\sqrt{3}}{2}\right)^{n} r_3^{2n}\\
			&=\frac{2^n (3n-3)!}{(2n)!((n-1)!)^3}.
		\end{align*}
	\end{proof}
	
	\subsection{The square}
	
	\begin{Theoreme}[{\bf Valtr }\cite{Valtr1995}]\label{valtr4}
		For all $n\geq3$, we have 
		\[\mathbb{P}_\Box(n)=\frac{1}{(n!)^2}{2n-2\choose n-1}^2.\]
	\end{Theoreme}
	Again, we propose a new proof of Valtr's result:
	\begin{proof}
		Consider a square (\ie case $\kappa=4$) of side length $ r_4=1.$ Pick $c_1,c_2,c_3,c_4,\ell_1,\ell_2,\ell_3,\ell_4$ verifying equations $(\mathcal{C}_j)_{1\leq j\leq 4}$. Since the only equiangular polygons with four sides are rectangles, we have $c_1=c_3$ and $c_2=c_4$ . This implies
		\[(\mathcal{C}_1=\mathcal{C}_3):c_1+\ell_1+\ell_3=1,\]
		\[(\mathcal{C}_2=\mathcal{C}_4):c_2+\ell_2+\ell_4=1.\]
		This means in particular that
		\[\mathcal{L}_4=\big\{(\ell_1,\ell_2,\ell_3,\ell_4)\in[0, 1]^4\text{ with }\ell_1+\ell_3\leq 1 \text{ and } \ell_2+\ell_4\leq 1\big\}.\] Just like before, we have ${\mathbb{D}}_n^{(4)}={\mathbb{Q}}_n^{(4)}$ so that with
		\[\sum_{s[4]\in\NN_4(n)}\int_{\RR^4}f_n^{(4)}\left(s[4],\ell[4]\right)\mathrm{d}\ell_1\mathrm{d}\ell_2\mathrm{d}\ell_3\mathrm{d}\ell_4=1,\]
		from what we deduce
		\[\mathbb{P}_{\Box}(n)=n!\sum_{s[4]\in\NN_4(n)}\int_{\ell[4]\in\mathcal{L}_4} \prod_{j=1}^{4}\frac{c_j^{s_{{j-1}}+s_j}}{s_j!(s_{{j-1}}+s_j)!}\mathrm{d}\ell_1\mathrm{d}\ell_2\mathrm{d}\ell_3\mathrm{d}\ell_4,\]
		where  $\NN_4(n)=\left\{(s_1,s_2,s_3,s_4)\in\NN^4\text{ such that }s_1+s_2+s_3+s_4=n\text{ and }s_j+s_{\wj}\geq1\right\}.$
		Put $(h,i,j,k)=s[4]$ and perform both substitutions $c_1=1-(\ell_1+\ell_3),c_2=1-(\ell_2+\ell_4)$ to get :
		\begin{align}\label{mldfqkjngdq}\nonumber
			\mathbb{P}_{\Box}(n)&=n!\sum_{(h,i,j,k)\in\NN_4(n)}\int_{0}^1\int_{0}^1
			\frac{(1-c_1)c_1^{h+i-1}}{h!(h+i-1)!}\frac{(1-c_2)c_2^{i+j-1}}{i!(i+j-1)!}\frac{c_1^{j+k-1}}{j!(j+k-1)!}\frac{c_2^{h+k-1}}{k!(h+k-1)!}\mathrm{d}c_1\mathrm{d}c_2\\
			&=n!\bigg(\int_{0}^1(1-c)c^{n-2}\mathrm{d}c\bigg)^2\sum_{(h,i,j,k)\in\NN_4(n)}\frac{1}{h!(h+i-1)!}\frac{1}{i!(i+j-1)!}\frac{1}{j!(j+k-1)!}\frac{1}{k!(h+k-1)!}
		\end{align}
		The integral is a standard beta-integral so that
		$\displaystyle	n!\bigg(\int_{0}^1(1-c)c^{n-2}\mathrm{d}c\bigg)^2=n!\left(\frac{(n-2)!}{n!}\right)^2.$
		It remains to compute the big sum $S$ apart in \eqref{mldfqkjngdq}:
		\begin{align*}
			S			&=\frac{1}{(n-2)!((n-1)!)^2}\sum_{(h,i,j,k)\in\NN_4(n)}{n-1\choose h+i}{n-1\choose j+k}{h+i \choose h}{j+k \choose k}{n-2 \choose i+j-1}\\
			%		&=\frac{1}{(n-2)!((n-1)!)^2}\sum_{r=1}^{n-1}{n-1\choose r}{n-1\choose n-r}\sum_{h=0}^r\sum_{k=0}^{n-r}{r \choose h}{n-r \choose k}{n-2 \choose n-(k+h)-1}\\
			&=\frac{1}{(n-2)!((n-1)!)^2}\sum_{r=1}^{n-1}{n-1\choose r}{n-1\choose n-r}{2n-2\choose n-1}\\
			%		&=\frac{1}{(n-2)!((n-1)!)^2}{2n-2\choose n}{2n-2\choose n-1}\\
			&=\frac{1}{n!((n-2)!)^2}{2n-2\choose n-1}^2.
		\end{align*}
		which is Valtr's formula.
	\end{proof}
	
	\section{Computation of $\mathbb{d}_\kappa$.}\label{sec:det}
	
	In this section, we aim at proving that the determinant $\mathbb{d}_\kappa$ of the matrix $\Sigma_\kappa^{-1}$ of size $(\kappa-1)\times(\kappa-1),$ defined in \Cref{thmonstre} as
	
	\[\Sigma_\kappa^{-1}:=\frac{1}{2}\begin{pNiceMatrix}[nullify-dots,xdots/line-style=loosely dotted]
		6&4&3&\Cdots&\Cdots&3&2\\
		4&8&5&4&\Cdots&4&3\\
		3&5&\Ddots&\Ddots&\Ddots&\Vdots&\Vdots\\
		\Vdots&4&\Ddots&\Ddots&\Ddots&4&\Vdots\\
		\Vdots&\Vdots&\Ddots&\Ddots&\Ddots&5&3\\
		3&4&\Cdots&4&5&8&4\\
		2&3&\Cdots&\Cdots&3&4&6\\
	\end{pNiceMatrix}\text{ for }\kappa \text{ large enough},\]
	is indeed 
			\[\mathbb{d}_\kappa=\frac{\kappa}{3\cdot2^\kappa}\left(2(-1)^{\kappa-1}+(2-\sqrt{3})^{\kappa}+(2+\sqrt{3})^{\kappa}\right),\]
		as given in \Cref{thm-1}.
		
	\begin{proof}
		We define the matrix $D_\kappa$ as
		\[D_\kappa:= 2\Sigma_\kappa^{-1}\begin{pNiceMatrix}[nullify-dots,xdots/line-style=loosely dotted]
			4/3&0&\Cdots&&0\\
			0&1&&&\\
			&\Ddots&\Ddots&\Ddots&\Vdots\\
			\Vdots&&&1&0\\
			0&\Cdots&&0&4/3\\
		\end{pNiceMatrix}=\begin{pNiceMatrix}[nullify-dots,xdots/line-style=loosely dotted]
		8&5&3&\Cdots&\Cdots&3&8/3\\
		16/3&8&5&4&\Cdots&4&4\\
		4&5&\Ddots&\Ddots&\Ddots&\Vdots&\Vdots\\
		\Vdots&4&\Ddots&\Ddots&\Ddots&4&\Vdots\\
		\Vdots&\Vdots&\Ddots&\Ddots&\Ddots&5&4\\
		4&4&\Cdots&4&5&8&16/3\\
		8/3&3&\Cdots&\Cdots&3&5&8\\
		\end{pNiceMatrix},\]
		where, the factor $2$ taken apart, we just multiplied by $\frac{4}{3}$ the first and last columns of $\Sigma_\kappa^{-1}$. This means of course that
		\[\det(D_\kappa)=2^{\kappa-1}(4/3)^2\mathbb{d}_\kappa.\]
		
		We now decompose $D_\kappa$ as $D_\kappa:=Q_\kappa+E_\kappa$ with 
		\[Q_\kappa:=\begin{pNiceMatrix}[nullify-dots,xdots/line-style=loosely dotted]
			3&\Cdots&\Cdots&3\\
			4&\Cdots&\Cdots&4\\
			\Vdots&&&\Vdots\\
			4&\Cdots&\Cdots&4\\
			3&\Cdots&\Cdots&3\\
		\end{pNiceMatrix},\quad\text{ and }\quad E_\kappa:=\begin{pNiceMatrix}[nullify-dots,xdots/line-style=loosely dotted]
		5&1&0&\Cdots&\Cdots&0&-1/3\\
		4/3&4&\Ddots&0&\Cdots&0&0\\
		0&1&\Ddots&\Ddots&\Ddots&\Vdots&\Vdots\\
		\Vdots&0&\Ddots&\Ddots&\Ddots&0&\Vdots\\
		\Vdots&\Vdots&\Ddots&\Ddots&\Ddots&1&0\\
		0&0&\Cdots&0&\Ddots&4&4/3\\
		-1/3&0&\Cdots&\Cdots&0&1&5\\
		\end{pNiceMatrix}.\]
	
	We first claim that 
	\begin{align}\label{eq:Ek}
		\det(E_\kappa)=\frac{4}{9}\left(2(-1)^{\kappa-1}+(2-\sqrt{3})^{\kappa}+(2+\sqrt{3})^{\kappa}\right).
	\end{align}
	To prove this, notice first that two Laplace expansions of the determinant of the $(m\times m)$ matrix \[L_m:=\begin{pNiceMatrix}[nullify-dots,xdots/line-style=loosely dotted]
		4&1&0&\Cdots&0\\
		1&\Ddots&\Ddots&\Ddots&\Vdots\\
		0&\Ddots&\Ddots&\Ddots&0\\
		\Vdots&\Ddots&\Ddots&\Ddots&1\\
		0&\Cdots&0&1&4\\
	\end{pNiceMatrix}\] 
	gives a constant-recursive sequence of order 2 for its determinant:
	\[\det(L_m)=4\det(L_{m-1})-\det(L_{m-2}),\] which can be solved immediately to get 
	\[\det(L_m)=\frac1{2\sqrt{3}}\left[(2+\sqrt{3})^{m+1}-(2-\sqrt{3})^{m+1}\right],m\geq 1.\]
	
	Several Laplace expansions of $\det(E_\kappa)$ along the first column allows one to either deal with diagonal matrices (leading to the term $\frac{8}{9}(-1)^{\kappa-1})$), or with $L_{\kappa-2}$ and $L_{\kappa-3}$ to get \eqref{eq:Ek} at the end.
	
	How do we compute $\det(D_\kappa)=\det(Q_\kappa+E_\kappa)$ ? In general, since the determinant is a multilinear alternating map of the columns of the matrix, for two $(m\times m)$ matrices $A=(A_i)_{1\leq i\leq m}$ and $B=(B_i)_{1\leq i\leq m}$, (where $A_i$ is the $i^{th}$ column of $A$), we can write
	\begin{align*}
		\det(A+B)&=\det(A_1+B_1,\ldots,A_m+B_m)\\
		&= \sum_{I\sqcup J=\ent{1}{m}} \det((A_i\mathbb{1}_{i\in I}+B_i\mathbb{1}_{i\in J})_{i\in\ent{1}{m}}),
	\end{align*}
	where $I\sqcup J=\ent{1}{m}$ means that $I,J$ forms a partition of $\ent{1}{m}$.
	
	In the case where all columns of $B$ are the same, the sum above only keeps the partition $(I,J)$ of $\ent{1}{m}$ where either $\abso{J}=0$ (hence we retrieve $\det(A)$), or $\abso{J}=1$.
	We therefore introduce the matrix $E_\kappa^{(i)}$ for all $i\in\ent{1}{\kappa-1}$ which is the matrix $E_\kappa$ where we replaced its $i^{th}$ column by $(3~~4 \cdots 4~~3)^t.$ By the previous argument we have
	\begin{align}\label{eq:Dk}
	\det(D_\kappa)=\det(E_\kappa)+\sum_{i=1}^{\kappa-1}\det(E_\kappa^{(i)}).
	\end{align}
	
	We now claim that 
		\begin{Lemma}\label{lem:lastlem}
		For all $\kappa$ large enough, we have:
		\begin{enumerate}
			\item $\det(E_\kappa^{(i)})=\frac{2}{3}\det(E_\kappa)$ for all $i\in\ent{2}{\kappa-2},$
			\item $\det(E_\kappa^{(1)})=\det(E_\kappa^{(\kappa-1)})=\frac{1}{2}\det(E_\kappa).$
		\end{enumerate}
	\end{Lemma}
	
		This lemma allows to conclude since by \eqref{eq:Dk}, we now have
	\[\det(D_\kappa)=\det(E_\kappa)\left(1+\frac{2}{3}(\kappa-3)+1\right)=\frac{2}{3}\kappa\det(E_\kappa).\]
	\end{proof}
	
	\begin{proof}[Proof of \Cref{lem:lastlem}]
		The proof relies on some determinant preserving column manipulations on $E_\kappa^{(i)},$ providing matrices equal to $E_\kappa$, up to a constant factor.
		
		Pick $i\in\ent{2}{\kappa-2},$ and consider the matrix $E_\kappa^{(i)}$ where the $i^{th}$ column is multiplied by $3/2$. Then substract all other columns in the $i^{th}$. Then add again $-1/4$ times the first and last columns in the $i^{th}$ column and you end up with $E_\kappa$. This gives the first point.
		
		If $i=1$ or $\kappa-1$ the same reasoning works: multiply the $i^{th}$ column by $2$ and then there exists a linear combination of the columns other than $i$ that you can add to the $i^{th}$ column to retrieve $E_\kappa$.
	\end{proof}

	\subsection*{Acknowledgements}	
	I would like to thank Jean-François Marckert for his valuable guidance and advices throughout the long research and writing process of this paper.
	
	\noindent I also thank Zoé Varin for her precious help when struggling with Tikz. Some of the figures of this paper were created entirely by her expert hand.
	
	\noindent Many thanks to the referees for their wise comments and suggestions that considerably improved this paper.
\end{appendices}

\bibliographystyle{abbrv}
\bibliography{godlikebiblio.bib}

\end{document}